\documentclass[11pt, leqno]{article}

\usepackage{constants}
\usepackage{titletoc,url}
\usepackage{stackrel,graphicx}
\usepackage{enumitem,amsthm,amsmath,amsfonts,amssymb,mathrsfs}
\usepackage[margin=1.1 in]{geometry}
\usepackage[margin={1.5cm, 1.5cm},font=small, labelsep=endash]{caption}

\usepackage{vwcol}

\dottedcontents{section}[4em]{}{2.9em}{0.7pc}
\dottedcontents{subsection}[0em]{}{3.3em}{1pc}

\newconstantfamily{c}{symbol=c}
\makeatother

\usepackage{titlesec}
\titleformat{\subsection}[runin]{\normalfont\bfseries}{\thesubsection.}{.5em}{}[.]\titlespacing{\subsection}{0pt}{2ex plus .1ex minus .2ex}{.8em}
\titleformat{\subsubsection}[runin]{\normalfont\itshape}{\thesubsubsection.}{.3em}{}[.]\titlespacing{\subsubsection}{0pt}{1ex plus .1ex minus .2ex}{.5em}

\theoremstyle{theorem}

\numberwithin{equation}{section}
\newtheorem{theorem}{Theorem}[section]
\newtheorem{lemma}[theorem]{Lemma}
\newtheorem{proposition}[theorem]{Proposition}

\theoremstyle{remark}

\newtheorem{remark}[theorem]{Remark}
\newtheorem{defn}[theorem]{Definition}

\newcommand{\R}{\mathbb R}
\newcommand{\Z}{\mathbb Z}
\newcommand{\E}{\mathbb E}

\newcommand{\N}{\mathbb N}

\usepackage{mathtools}
\usepackage{centernot}

\renewcommand{\P}{\mathbb P}
\newcommand{\lr}[4]{#3\xleftrightarrow[#1]{#2} #4}
     \newcommand{\nlr}[4]{#3\mathrel{\mathop{\centernot\longleftrightarrow}_{#1}^{#2}} #4}

\makeatletter
\newcommand\footnoteref[1]{\protected@xdef\@thefnmark{\ref{#1}}\@footnotemark}
\makeatother

\usepackage{titletoc}

\dottedcontents{section}[4em]{}{2.9em}{0.7pc}
\dottedcontents{subsection}[0em]{}{3.3em}{1pc}

\usepackage{multicol}
\usepackage{parcolumns}

\title{{\textbf{\large{On the radius of Gaussian free field excursion clusters}}}}
\date{}
\begin{document}
\thispagestyle{empty}
\maketitle
\vspace{0.1cm}
\begin{center}
\vspace{-1.7cm}
Subhajit Goswami$^1$, Pierre-Fran\c cois Rodriguez$^2$ and Franco Severo$^{3}$

\end{center}
\vspace{0.1cm}
\begin{abstract}
\begin{minipage}{0.80\textwidth}
We consider the Gaussian free field $\varphi$ on $\mathbb{Z}^d$, for $d\geq 3$, and give sharp bounds on the probability that the radius of a finite cluster in the excursion set $\{\varphi \geq h \}$ exceeds a large value $N$ for any height $h\neq h_{*}$, where $h_{*}$ refers to the corresponding percolation critical parameter. 
In dimension~$3$, we prove that this probability 
is sub-exponential in $N$ and decays as $\exp\{-\frac{\pi}{6}(h-h_*)^2\frac{N}{\log N}\}$ as $N \to \infty$ to principal exponential order. When $d\geq 4$, we prove that these tails decay exponentially in~$N$. Our results extend to other quantities of interest, such as truncated two-point functions and the two-arms probability for annuli crossings at scale $N$.
\end{minipage}
\end{abstract}

%\vspace{1.0cm}
%\begin{minipage}{0.89\textwidth}
%{\small
%\tableofcontents
%}
%\end{minipage}

\vspace{7cm}

\begin{flushleft}
\thispagestyle{empty}
\vspace{0.3cm}
\noindent\rule{5cm}{0.35pt} \hfill September 2022 \\[2em]

\begin{multicols}{2}
\small
$^1$School of Mathematics
\\Tata Institute of Fundamental Research \\ 1, Homi Bhabha Road \\ Colaba, Mumbai 400005, India. \\\url{goswami@math.tifr.res.in}\\[2em]

$^2$Imperial College London\\
Department of Mathematics\\
London SW7 2AZ \\
United Kingdom\\
\url{p.rodriguez@imperial.ac.uk}\columnbreak

\hfill$^3$Department of Mathematics\\
\hfill ETH Zurich\\% (I.H.\'E.S.)
\hfill R\"amistrasse 101\\ 
\hfill 8092 Zurich, Switzerland.\\
\hfill\url{franco.severo@math.ethz.ch}\\[2em]

\end{multicols}
\end{flushleft}

\newpage

		\section{Introduction}\label{sec:intro}
	
This article investigates the percolative properties of excursion sets $\{\varphi \geq h\}$ of the Gaussian free field $\varphi$ on $\Z^d$ in dimensions $d \geq 3$, for varying height parameter $h \in \R$. This model, the rigorous study of which was initiated in \cite{BLM87}, and more recently re-instigated in \cite{RoS13}, serves as a benchmark example of a (non-planar) percolation model with strong, algebraically decaying correlations. One of its appealing features is the 
rich interplay with potential theory for the underlying random walk, which is beneficial to its 
study. A central role is thus played by electrostatic notions such as capacity, see 
e.g.~\cite{BDZ95},~\cite{deuschel1999entropic}, and more recently \cite{Sz15},~\cite{nitzschner2017solidification},~\cite{nitzschner2018},~\cite{chiarini2018entropic},~\cite{sznitman2018macroscopic},~\cite{chiarini2020disconnection}, which also feature prominently in the present work.

Our main focus concerns the radii of finite clusters in the excursion sets $\{\varphi \geq h\}$, which we set out to introduce. Under a suitable probability $\P$, the field $\varphi$ is the centered Gaussian field with covariance $\E[\varphi_x \varphi_y]=g(x,y)$, for $x,y \in \Z^d$, where $g(\cdot,\cdot)$ denotes the Green function of the simple random walk on $\Z^d$, see \eqref{eq:Green_U}. The critical parameter for percolation of the associated excursion sets $\{\varphi \geq h\} \coloneqq \{x \in \Z^d: \varphi_x \geq h \}$ is defined as
\begin{equation}
\label{eq:h_*}
h_*=h_*(d) \coloneqq \inf \{h \in \R: \lim_N \P[\lr{}{\varphi\geq h}{0}{\partial B_N}] =0 \},
\end{equation}
where, with hopefully obvious notation, the event in \eqref{eq:h_*} refers to a (nearest-neighbor) path in $\{\varphi \geq h\}$ connecting $0$ and $\partial B_N$, where $B_N=B_N(0)$, $B_N(x)\coloneqq \{ y \in \Z^d: |y-x|_{\infty} \leq N\}$ for all $x \in \Z^d$, $N \geq 1$, and $\partial K$ refers to the inner (vertex) boundary of a set $K \subset \Z^d$. It is known that $0< h_* < \infty$ for all $d \geq 3$, see \cite{BLM87}, \cite{RoS13}, \cite{DrePreRod}, and that the infinite cluster, when existing, is almost surely unique. Auxiliary parameters $\bar h$ and $h_{**}$ satisfying $\bar h \leq h_* \leq h_{**}$ were frequently used in the past,  respectively characterizing a phase of `well-behavedness' for the infinite cluster and a strongly subcritical regime, in which connectivities decay rapidly. Recently it was proved in \cite{DCGRS20} that
\begin{equation}
\label{eq:sharp}
\bar h = h_* = h_{**}.
\end{equation}
As a consequence, one knows the following: there exist $\Cl[c]{c:strexp}= \Cr{c:strexp}(d,h) >0$  and $ c=c(h,d) >0$ such that, for all $N \geq 1$
\begin{align}
\P[\lr{}{\varphi\geq h}{0}{\partial B_N}]&\leq e^{-c\,N^{\Cr{c:strexp}}}, \text{ if {$h>h_{*}(d)~(=h_{**}(d))$},} \label{eq:def_h**} \\[0.2em]
\P\left[ \text{LocUniq}(N,h)^c \right] &\leq e^{-c\,N^{\Cr{c:strexp}}}, \text{ if $h<h_*(d)~(=\bar{h}(d))$},		\label{eq:def_hbar}
\end{align}
where the `local uniqueness' event  in \eqref{eq:def_hbar} is defined as
\begin{equation}	
\label{eq:def_locuniq}	
 \text{LocUniq}(N,h) =\left\{	
 \begin{array}{l}	
\text{$\{\varphi \geq h\}$ has a \textit{unique} connected}\\
\text{component crossing $ B_{2N} \setminus B_N$}%\end{array}
\end{array}
 \right\}.
\end{equation}
	Here and in the sequel, a set $S\subset \Z^d$ is said to cross $V\setminus U$, for $U \subset V \subset \Z^d$ if $S$ has a connected component intersecting both $U$ and $\partial V$.	
	The estimate \eqref{eq:def_hbar} is inherited from the bounds for the `existence' and `uniqueness' events usually appearing in the definition of $\bar h$, see~e.g.~(1.10)--(1.11) in \cite{DPR18}, which assert that for all $h<\bar{h} \, (=h_*)$ and $N \geq 1$,
\begin{equation}
 \label{eq:EXISTUNIQUE}
\begin{split}
&\P\left[\begin{array}{c}  \text{there exists a connected component in }\\ \text{$\{\varphi \geq h \}\cap B_N$ with diameter at least  $N/5$}\end{array}\right] \geq 1- e^{-c\,N^{\Cr{c:strexp}}}  \text{ and }\\[0.3cm]
&\P \left[\begin{array}{c}\text{any two clusters in $\{\varphi\geq h\}\cap B_{N}$ having diameter at}\\ \text{least $N/10$ are connected to each other in $\{\varphi\geq h\}\cap B_{2N}$ } \end{array}\right] \geq 1- e^{-c\,N^{\Cr{c:strexp}}}, 
\end{split}
\end{equation}
where the diameter of a set is with respect to the sup-norm. Indeed, \eqref{eq:def_hbar} follows by straightforward gluing arguments, combining the events in \eqref{eq:EXISTUNIQUE} at a fixed number of scales commensurate with $N$. Stretched exponential bounds such as \eqref{eq:def_h**}, \eqref{eq:def_hbar} and \eqref{eq:EXISTUNIQUE} typically arise as a by-product of certain static renormalization methods, see e.g.~\cite{RoS13} regarding \eqref{eq:def_h**}, which exemplifies this phenomenon.
Little is otherwise known about the true order of decay for the probabilities in \eqref{eq:def_h**} and \eqref{eq:def_hbar}.
To date, the best available results are due to \cite{PR15}, \cite{PT12}, which {\em solely} concern the subcritical regime and yield that  \eqref{eq:def_h**}  holds with $\Cr{c:strexp}(d,h)=1$ if $d \geq 4$ and $h>h_*$, along with an upper bound for $\P[\lr{}{\varphi\geq h}{0}{\partial B_N}]$ of exponential order $N/(\log N)^{3+\varepsilon}$, for any $\varepsilon > 0$, when $d =3$ and $h>h_*$. 

Our findings address these matters. Our main results are most easily formulated in terms of a `truncated one-arm event'. We refer to the discussion following the statement of Theorem~\ref{thm:main_d>3} below regarding extensions of \eqref{eq:main_d=3} and \eqref{eq:main_d>3} to other quantities of interest. Upper bounds in the spirit of those obtained below for $h>h_*$ have also been derived in \cite{DiWi} for the so-called metric graph associated to $\Z^d$. Contrary to what is suggested in Section 1.3 of~\cite{DiWi}, the logarithmic factor in dimension three is \textit{not} an artefact. 

	\begin{theorem}\label{thm:main_d=3} 
		For $d=3$ and all $h \in \R$, 
		\begin{equation}\label{eq:main_d=3}
		\lim_{N\to\infty}\, \frac{\log N}{N} \log \P[\lr{}{\varphi\geq h}{0}{\partial B_N},\, \nlr{}{\varphi\geq h}{0}{\infty}] = -\frac{\pi}{6}(h-h_*)^2.
		\end{equation}
	\end{theorem}

In higher dimensions, we have the following:	

	\begin{theorem}\label{thm:main_d>3} 
		For every ${d\geq4}$ and $h \neq h_{*}$, there exist $C=C(d,h),\,c=c(d,h) \in (0, \infty)$ such that
		\begin{equation}\label{eq:main_d>3}
		e^{-CN} \leq \P[\lr{}{\varphi\geq h}{0}{\partial B_N},\, \nlr{}{\varphi\geq h}{0}{\infty}]\leq e^{-cN}.
		\end{equation}
	\end{theorem}

Theorems~\ref{thm:main_d=3} and~\ref{thm:main_d>3} follow immediately by combining the results of Theorems~\ref{P:lb} and~\ref{P:ub} below, which separately deal with the corresponding 
lower and upper bounds, respectively. In fact, as asserted in these two theorems, \eqref{eq:main_d=3} and \eqref{eq:main_d>3} continue to hold when $h \le h_*$ if one replaces 
the event in question by $\text{LocUniq}(N,h)^c$, see \eqref{eq:def_locuniq}. %Together with the disconnection upper bound from \cite[Theorem 5.5]{Sz15}, which yields that disconnecting $B_N$ from $\partial B_{2N}$ decays exponentially at scale $N^{d-2}$ when $h<h_*$, this is easily seen to imply that 
We also have the similar asymptotic for a slightly modified version of the event 
$\text{LocUniq}(N,h)^c$, namely
\begin{equation}
\label{eq:intro2arm}
\lim_{N\to\infty}\, \frac{\log N}{N} \log \P[\text{2-arms}(N,h)] = -\frac{\pi}{6}(h-h_*)^2, \text{ if $h \le h_*$ and $d=3$},
\end{equation}
along with a statement similar to \eqref{eq:main_d>3} when $d\geq 4$, where `2-arms' refers to the existence of two disjoint crossing clusters of $\{ \varphi \geq h\} \cap( B_{2N} \setminus B_{N})$. %i.e.~intersecting each component of~$\partial(B_{2N} \setminus B_{N})$.

Concerning the truncated two-point function, defined as 
\begin{equation}
\label{eq:intro2point}
\tau_h^{\text{tr}}(x,y)= \P[\lr{}{\varphi\geq h}{x}{y},\, \nlr{}{\varphi\geq h}{x}{\infty}],  \text{ for }x,y \in \Z^d \text{ and } h\in \R,
\end{equation}
one readily sees that this is symmetric in $x$ and $y$ and satisfies $\tau_h^{\text{tr}}(x,y)= \tau_h^{\text{tr}}(0,y-x)$ by translation invariance of the set $\{\varphi \geq h\}$. Our 
results readily imply (cf.~the proof of Theorem~\ref{P:lb} below) that the asymptotics \eqref{eq:main_d=3} and \eqref{eq:main_d>3} also hold for $\tau_h^{\text{tr}}(0, Ne_1)$, with $h \in 
\R$ and $h \neq h_*$ respectively, where $e_1$ denotes the unit vector in a coordinate 
direction of $\Z^d$. More generally, with $|\cdot|$ denoting the Euclidean distance, one may 
also expect that for arbitrary $x,y \in \Z^d$ and $h \in \R$,
	\begin{equation}
\label{eq:intro2point}
\lim_{|x-y|\to\infty}\, \frac{\log |x-y|}{|x-y|} \log \tau_h^{\text{tr}}(x,y) = -\frac{\pi}{6}(h-h_*)^2, \text{ when $d=3$.}
\end{equation}	
We refer to Remarks~\ref{remark:cap_tube} and~\ref{R:final},~2) below regarding the (technical) modifications to our argument needed to prove \eqref{eq:intro2point} and compelling evidence for its truthfulness. %In a related fashion, there is actually a certain flexibility in the choice of the norm for $B_N$ in \eqref{eq:main_d=3}. Indeed, whereas the $\ell^{\infty}$-distance is convenient and natural for the multiscale arguments we have in mind, the $\ell^{p}$-norm for any $p \geq 2$ could plausibly be used in \eqref{thm:main_d=3} without affecting the leading asymptotics; see Remark~\ref{R:final},~3) for more on this. 
Finally, let us emphasize that, while \eqref{eq:sharp} leads to a form of Theorem~\ref{thm:main_d=3} indicating that $h_*$ is approached at the \textit{same} rate both as $h \searrow h_*$ and $h \nearrow h_*$, a version of our findings could be stated in terms of $\bar h$ and $h_{**}$ only, much as in \cite{Sz15}, \cite{nitzschner2018}, \cite{chiarini2018entropic}, thus yielding \eqref{eq:main_d=3} and \eqref{eq:main_d>3} upon applying \eqref{eq:sharp}. This is the sole place where \eqref{eq:sharp} is used.

Next, we compare the findings of Theorems~\ref{thm:main_d=3} and~\ref{thm:main_d>3} to related existing results. For Bernoulli percolation (i.e.~taking $\varphi$ to be a field of independent standard Gaussians, traditionally parametrized using $p= \P[\varphi_0 \geq h]$, or the analogue bond percolation problem), exponential decay of the radius function at a leading rate $e^{- N/\xi(p)}$ can be obtained for all $p \neq p_c$ \cite{Me86}, \cite{AB87}, \cite{10.2307/2244436}, \cite{GM90}, see also \cite[Sections 6.2 and 8.4]{Gr99}. Little is known apart from existence and qualitative properties of $\xi(\cdot)$ in intermediate dimensions. One can show polynomial lower bounds of the form $\xi(p) \geq |p-p_c|^{-c(d)}$ for $p< p_c$, see \cite[Section 8]{zbMATH07395544} and refs.~therein. The currently best available upper bounds~\cite{zbMATH07395544}, valid for all $p \neq p_c$, do not witness a polynomial scaling. One also has estimates similar to \eqref{eq:intro2point} for the two-point function $ \tau_p $, including improved Ornstein-Zernike asymptotics for $p<p_c$ \cite{zbMATH04131412, 10.1214/aop/1023481005} or $p$ close to $1$ \cite{zbMATH05996809}, but the relevant norm $\xi_p(x)\coloneqq -\lim_N N^{-1} \log \tau_p (0,Nx)$ in the leading order term is implicit. In particular, the possible isotropy of $\xi_p(\cdot)$ (near $p_c$),~cf.~\eqref{eq:intro2point}, see also \cite[Chap.~XII.5]{McCoyWu+2013} regarding the Ising model in dimension $d=2$ and \cite[Fig.~2.2]{ott2020asymptotics} for related simulations, is not clear. In this regard \eqref{eq:intro2point} hints at the rotational symmetry of a possible scaling limit at criticality. 

The situation improves dramatically in high dimensions or when $d=2$. Indeed, one knows by \cite{zbMATH04123029}, \cite{zbMATH04145109}, that $\xi(p) \sim |p-p_c|^{-\nu}$ with $\nu=\frac12$ as $p \nearrow p_c$ when $d$ is large enough.  As far as we understand, the case of $p \searrow p_c$, i.e.~the behavior of truncated observables, is completely open in high dimensions. For site percolation on the planar triangular lattice, $\nu=\frac43$ and $\xi_p(x)$ is independent of the direction of $x$ near $p_c$, see \cite{zbMATH01744230} and refs.~therein.

The percolation problem \eqref{eq:h_*} has a close cousin, involving the excursion sets of the free field $\tilde{\varphi}$ on the corresponding cable system $\widetilde{\Z}^d$, which can be understood as a purely discrete \textit{bond} percolation problem involving additional disorder on the edges: given $\varphi$, each edge $e=\{x,y\}$ is closed independently with probability $\exp\{-2(\varphi_x-h)_+ (\varphi_y-h)_+\}$, cf.~the discussion around \cite[(1.6)]{drewitz2021critical}; see also \cite{zbMATH06603570} or \cite[Section 2]{drewitz2021cluster} for further details. Variants of this observation have been recently used in several related contexts \cite{zbMATH06603570}, \cite{zbMATH06583846}, \cite{zbMATH07322630}.

For percolation of the excursion sets $\{\tilde{\varphi} \geq h \}$, one knows that the associated critical parameter $\tilde{h}_*$ corresponding to \eqref{eq:h_*} satisfies $\tilde{h}_*= 0$ in great generality \cite[Corollary 1.2]{drewitz2021cluster}, including all  transitive graphs (and in particular $\widetilde{\Z}^d$, $d \geq 3$). This is due to the behavior of the \textit{cluster capacity} observable, which was recently realized to play a central role in this context; to wit, this observable naturally emerges as part of the differential identities derived in \cite[Section 2]{drewitz2021critical} which correspond to the classical Margulis-Russo formula \cite[Section 2.4]{Gr99} for independent percolation. Essentially, the fact that $\tilde{h}_*=0$ on \textit{most} graphs is a direct consequence of such a formula and the fact that the capacity of clusters of $\{\tilde{\varphi}>0\}$ is finite on \textit{all} transient graphs \cite[Theorem 1.1,1)]{drewitz2021cluster}. Following \cite{drewitz2021critical}, \cite{drewitz2021cluster}, and somewhat in the spirit of Theorems~\ref{thm:main_d=3} and~\ref{thm:main_d>3} above, 
the leading exponential order for the tail of the corresponding capacity observable for $\varphi$ on $\Z^d$, $d \geq3$, was subsequently derived in \cite[Theorem 1.3]{panagiotis2021analyticity} for $h \neq h_*$, thus mirroring parts of \cite[Theorem 3.7]{drewitz2021cluster} for $\tilde{\varphi}$; see also \cite[(1.14)]{drewitz2021critical} for corresponding (near-)critical tails.

With regards to connectivity functions, the result of Theorem~\ref{thm:main_d=3} (as well as \eqref{eq:intro2point}) can be drastically improved when considering the percolation problem for $\{ \tilde\varphi \geq h \}$. Indeed, as recently shown in \cite[Theorem 1.4]{drewitz2021critical}, abbreviating by $\tilde{\tau}_h^{\text{tr}}(\cdot,\cdot)$ the quantity in \eqref{eq:intro2point} with $\tilde{\varphi}$ in place of $\varphi$ and $\xi(h)=|h-\tilde{h}_*|^{-2}=|h|^{-2}$, one has when $d=3$ that \begin{equation}
\label{eq:radius-scaling}
c'\tilde{\tau}_0^{\text{tr}}(0,x)  
 \exp \Big\{-  \frac{C|x|/ \xi(h)}{  \log \big(\textstyle\frac{|x|}{\xi(h)}\vee2\big)} \displaystyle \Big\} \leq \tilde{\tau}_h^{\text{tr}}(0,x) \leq \tilde{\tau}_0^{\text{tr}}(0,x) \exp \Big\{-  \frac{c|x|/ \xi(h)}{ \log (|x| \vee2)} \Big\},
\end{equation}
which is valid for all $h \in \R$ and $x \in \Z^3$, where the lower bound requires that $\frac{|x|}{\xi(h)} \geq (\log \xi(h))^{\delta}$ for some $\delta \in (0,1)$, along with analogous results for the radius observable of \eqref{eq:main_d=3}. In particular, the bounds \eqref{eq:radius-scaling} witness the onset of the leading exponential asymptotics of Theorem~\ref{thm:main_d=3} as occurring \textit{at scale} $\xi$, which justifies its use as a correlation length. Together with the scaling behavior of $\tilde{\tau}_0^{\text{tr}}(0,x)$, which follows readily from the Markov property, as observed in \cite{zbMATH06603570}, the bounds \eqref{eq:radius-scaling} almost display the full scaling behavior of $\tilde{\tau}$, and are sufficient to yield Fisher's scaling law, see \cite[Corollary 1.5]{drewitz2021cluster}; see also \cite{hutchcroft2021critical} regarding recent findings displaying analogous scaling to that of $\tilde{\tau}_0^{\text{tr}}$ for the critical two-point function of independent long-range percolation on the hierarchical lattice.

Noteworthily, whereas upper bounds will be the main difficulty in proving Theorems~\ref{thm:main_d=3} and \ref{thm:main_d>3} below, the lower bounds in \eqref{eq:radius-scaling} are far more difficult. These bounds require understanding connection mechanisms up to the critical scale $\xi$ and they rely on different ideas, which will not be discussed here; see \cite{drewitz2021critical}. On the other hand, the upper bound in \eqref{eq:radius-scaling} follows readily by comparison with the afore mentioned cluster capacity observable, see \cite[Section 4]{drewitz2021critical}; see also \cite[Theorem~4]{DiWi} for related if somewhat weaker bounds.

Results corresponding to Theorem~\ref{thm:main_d>3} for $\tilde{\varphi}$ in dimensions $d \geq 4$ follow by combining \cite[(1.26)]{drewitz2021cluster} for $\nu=d-2$ and an adaptation of the argument leading to \eqref{eq:main_lb_dgeq4} below to the cables. This produces \textit{both} Gaussian upper and lower bounds in $h$ when $|h|$ is at least of the order of local fluctuations, i.e.~$\text{var}(\varphi_0)^{1/2}$ or so, see also \cite[Remark 8.3]{drewitz2021critical}. However, currently available upper and lower bounds in the spirit of \eqref{eq:radius-scaling} for $d \geq 4$ are far from matching, cf.~\cite[Theorem 1.4 and Proposition 6.1]{drewitz2021critical} for known results. Improving on this, i.e.~obtaining the correct analogue of \eqref{eq:radius-scaling} in higher dimensions for $|h| \ll \text{var}(\varphi_0)^{1/2}$ along with the correct choice of $\xi$, is presently open; see also \cite[Theorem 1.7]{drewitz2021cluster} for some results in this direction when $d=4$ (corresponding to $\nu=2=\frac{\alpha}{2}$ in \cite{drewitz2021cluster}).

\medskip
We now highlight some ideas behind our proofs of Theorems~\ref{thm:main_d=3} and~\ref{thm:main_d>3}. One is immediately struck by the discrepancy in the strength of the above results. This is closely related to the fact that the random walk 
does not `see' one-dimensional sets (such as bounded off-critical percolation clusters) when $d\geq 4$. Our proofs witness this structural difference between the cases $d=3$ and $d \geq 4$ very clearly. To see this, first observe that (see Lemma~\ref{lem:cap_tube} for precise statements) as $N \to \infty$,
\begin{equation}
\label{eq:caplineintro}
\begin{array}{ll}
\displaystyle
\text{cap}\big(([0,N]\cap \mathbb{Z})\times\{0 \}^{d-1}\big) \sim \frac{\pi}{3}\frac{N}{\log N}, &\text{ when $d=3$, whereas}\\[1.0em]
\displaystyle
\text{cap}\big(([0,N]\cap \mathbb{Z})\times\{0 \}^{d-1}\big) \asymp N, & \text{ when $d\geq 4$.}
\end{array}
\end{equation}
Now, the coarse-graining described in more detail below (from which we eventually deduce the upper bounds in \eqref{eq:main_d=3}, \eqref{eq:main_d>3}), yields a sum of two terms for the probability in question. One of them corresponds to a truncated version of $\varphi$ (a local field, independent at large scales), for which a corresponding one-arm event decays exponentially in $N$, regardless of the dimension~$d$. The other term, which carries the long-range dependence, stems from the behavior of the harmonic field in a collection of well-separated boxes, and will turn out to behave in a manner proportional to $\text{cap}\big(([0,N]\cap \mathbb{Z})\times\{0 \}^{d-1}\big)$ to leading exponential order. In view of \eqref{eq:caplineintro}, this means that the harmonic term clearly dominates in dimension $3$, whereas the two terms live at the same exponential scale in dimension four and higher (and in fact the local term is typically larger).

The lower bounds derived in Section~\ref{sec:lower} further reflect this disparity. For $d=3$, in the subcritical regime, we use a change of measure argument in order to draw a finite path in $\{\varphi \geq h\}$ in a thin horizontal tube. The supercritical regime requires a more delicate treatment, as discussed below. Intuitively, the field shifts itself by the right amount in a suitable region as to make the event in question typical, cf.~Lemma~\ref{lem:entrop_lower} for a general result in this direction, which is of independent interest. The limit on the right-hand side of \eqref{eq:main_d=3} thereby emerges in the corresponding Radon-Nikodym derivative as half of the leading order pre-factor for the capacity of the shifted region, which is close to that of a line of length $N$, see \eqref{eq:cap_line_asymp}, times the square of the height gap. Similar arguments have  been 
used in the study of hard wall conditions for $\varphi$, see \cite{BDZ95}, and disconnection 
probabilities for supercritical excursion sets, see \cite{Sz15}. Importantly, the monotonicity 
of the events in question (common to these references) is absent for the one in 
\eqref{eq:main_d=3} when $h<h_*$, which requires that we `insulate' the path, i.e.~build an 
interface in $\{\varphi < h\}$ to shield it away from $\infty$. This makes the implementation 
of our lower bound strategy relatively involved in the supercritical regime and forces us to 
introduce Dirichlet boundary conditions to decorrelate constituents of opposite monotonicity. 
In sharp contrast, the lower bounds in \eqref{eq:main_d>3} follow by `FKG-type' arguments, 
which do not witness the critical parameter $h_*$ at all, see \eqref{eq:main_lb_dgeq4}.

Most of our work goes into proving the upper bounds required for Theorems~\ref{thm:main_d=3} and~\ref{thm:main_d>3}, summarized in Theorem~\ref{P:ub} below. A stepping stone
towards this is a certain coarse-graining scheme for paths, developed in Section~\ref{sec:coarsegrain} (see in particular Proposition~\ref{prop:coarse_paths} below), which we now briefly describe. Roughly speaking, for a path $\gamma$ of linear size $N$, the coarse-graining of $\gamma$, formalized in Definition~\ref{def:admissible}, only retains the trace of $\gamma$ in a system of `well-separated' boxes at scale $L  \ll N$. Importantly, the scheme walks the fine line of operating at a preferential entropic cost (parametrized by a function $\Gamma(\cdot)$, see \eqref{def:admissible3}, \eqref{eq:cg_gamma}), while retaining a sufficiently `large' piece of path when measured in terms of capacity. This latter property, ensured by Proposition~\ref{prop:coarse_paths}, see \eqref{eq:cg_capd=3}, is crucial for the precise estimates we aim at.

In the subcritical phase $h>h_*$, the above scheme is  used to cascade a connection event such as $\{\lr{}{\varphi\geq h}{0}{\partial B_N}\}$ from scale $N$ down to scale $L(\ll N)$. For each of the boxes at scale $L$ in the resulting collection, the occurrence of a crossing in that box is split into a similar event for a localized field with good decorrelation properties as the box is varied, and the occurrence of an atypical behavior for the corresponding harmonic average, see \eqref{eq:psibadsub}-\eqref{eq:bnd_psi_xi}. The leading-order contribution is thereby carried by the harmonic field in all but a small fraction of $L$-boxes, which we control by means of state-of-the-art estimates developed in \cite{Sz15}, cf.~Lemma~\ref{lem:BTIS_Szn} below. The strength of these estimates hinges on a suitable capacity lower bound for the underlying collection of boxes, which Proposition~\ref{prop:coarse_paths} provides. 

The resulting two-scale estimate for the one-arm event can then be applied iteratively, see Proposition~\ref{prop:bootstrap} below, to boost an a-priori bound such as \eqref{eq:def_h**} (but see Remarks~\ref{remark:bootstrap} and~\ref{remark:stronger_bootstrap} below to accommodate much weaker a-priori bounds) to the desired decay in a finite number of steps, if $L$ is carefully chosen as a function of $N$ (as will turn out, $L$ needs to grow poly-logarithmically in $N$).  In fact, two steps suffice if one starts from \eqref{eq:def_h**}.

The derivation of the desired upper bounds in the supercritical regime, see \eqref{eq:main_ubsupercrit_d=3} and \eqref{eq:main_ubsupercrit_d>4} in Theorem~\ref{P:ub}, is considerably more involved. When $h<h_*$, connections become typical and the cost displayed in \eqref{eq:main_d=3} and \eqref{eq:main_d>3} measures the difficulty to avoid the infinite cluster. Our approach revolves around an event $G_N$, see \eqref{def:goodevent}, ensuring, roughly speaking, that any macroscopic path at scale $N$ will have $a_N$ `contact points' in 
each of $b_N$ interfaces, all of which are connected to infinity in $\{\varphi \geq h\}$. These 
contact points are in fact local areas at a microscopic scale $L_0$ in which a certain 
insertion tolerance property holds (which the model does not possess as such due to the strength of the correlations), thus yielding a small i.i.d.~cost to avoid connecting to the 
infinite cluster. This property is conveniently defined in terms of a `mid-point' extension 
of $\varphi$ that was used in \cite{DPR18}, see \eqref{eq:LB1_dim4} and \eqref{eq:ubsuper1}. 
Incidentally, we also take advantage of this extension to deal with competing monotonicity 
properties of the path and the insulating interface when deriving the lower bounds for $d \geq 
4$ and $h<h_*$.

An upper bound on the key quantity $\P[G_{N}^c]$ is then derived using a bootstrapping scheme, see  Proposition~\ref{lem:inclusion_fxn_goodevent} below, which works roughly as follows. Starting from a certain (localized) good event $\mathcal{G}_z$ at base scale $L \gg L_0$, comprising a local uniqueness property at that scale and a number $a_L$ ($=1$ to begin with) of contact points to the ambient cluster for any large path, see Definition~\ref{def:goodevent2}, for which a suitable a-priori estimate is available (cf. Lemma~\ref{lem:loc_amb_cluster1}), the scheme does one of two things: i) in intermediate steps, it re-produces the same event $\mathcal{G}_z$ at larger scale $N$, improving on both its likelihood and the number $a_N$ of contact points (eventually we need $a_Nb_N$ to grow linearly with $N$ when $d \geq 4$ and sub-linearly but with $a_Nb_N \gg N/\log N$ when $d=3$); ii) in the final step, the scheme generates the target event $G_N$, creating multiple interfaces by stacking good boxes at scale $L$. In either case, the scheme witnesses this improvement on a certain event, see \eqref{def:good_box_event2}, defined in terms of the coarse-graining from Proposition~\ref{prop:coarse_paths}, and for which a dichotomy (involving local fields and harmonic averages) holds, see \eqref{eq:psibadsup}-\eqref{eq:EF_decomp1}. The proofs of the desired upper bounds then follow somewhat similarly as in the subcritical case.
\medskip	
	
We now briefly describe the organization of this article. Section~\ref{sec:preliminaries} gathers several preliminary results that will be used in subsequent sections. Section~\ref{sec:lower} proves the lower bounds corresponding to Theorems~\ref{thm:main_d=3} and~\ref{thm:main_d>3}, see Theorem~\ref{P:lb}. Section~\ref{sec:coarsegrain}
	 supplies the coarse-graining scheme for paths, see Proposition~\ref{prop:coarse_paths}, which will be instrumental in deriving the upper bounds. The proof differs depending on whether $d=3$ or $d \geq 4$, which are dealt with separately in Sections~\ref{sec:coarse3d} and~\ref{sec:coarse4d}. The desired upper bounds are then derived in Section~\ref{sec:upper}. The sub- and supercritial phases are considered separately in Sections~\ref{sec:upper_sub} and~\ref{sec:upper_sup}. 
	 
	 Our convention regarding constants is the following. Throughout, $c,c',C,C',\dots$ denote positive constants that may change from place to place. Numbered constants are defined the first time they appear and remain fixed thereafter. All constants may depend implicitly on the dimension $d$. Their dependence on other parameters will be made explicit.

	\section{Preliminaries and capacity estimates for tubes}\label{sec:preliminaries}

In this section, we gather several ingredients that will be used in the sequel. We first introduce some more notation and state a topological condition on paths yielding the existence of blocking interfaces, see Lemma~\ref{lem:blocking_layers} below. We proceed to recall certain aspects of potential theory for the random walk on $\Z^d$ and supply suitably precise capacity estimates for `tubular' sets, including `porous' versions thereof, see Lemmas~\ref{lem:cap_tube}--\ref{lem:cap_tube2} below. Finally, we discuss important properties of the free field $\varphi$, including a certain mid-point extension of~$\varphi$.

We consider $\Z^d$, $d\geq 3$, endowed with the usual nearest-neighbor graph structure. We write $x \sim y$ if $x$ and $y$ are neighbors, i.e. if $x,y \in \Z^d$ and $|x-y| =1$. We use $|\cdot|$ to denote the 
Euclidean and $|\cdot|_{\infty}$ the $\ell^{\infty}$-norm in $\Z^d$ as well as $d(\cdot, \cdot)$ and $d_\infty(\cdot, \cdot)$ to denote the corresponding distances between sets. Recall 
that $B_N(x)$ denotes the box of radius $N$ around $x$ with respect to $|\cdot|_{\infty}$, and 
let $B_N(U) \coloneqq \bigcup_{x\in U} B_N(x)$ for $U \subset \Z^d$. For $U \subset \Z^d$, $\partial U \coloneqq \{ x \in U: \exists y \notin U \text{ s.t. } y \sim x\}$
is the {\em inner (vertex) boundary} of $U$ and $U^c =\Z^d \setminus U$ is the complement of 
$U$ in $\Z^d$. We also define the outer boundary of a set $U \subset \Z^d$ as 
$\partial_{\text{out}} U = \partial(U^c)$. For $U, V \subset \mathbb{Z}^d$, we write $U \subset 
\subset V$ to indicate that $U$ has finitely many elements. A \textit{path} $\gamma$ 
in $\Z^d$ is a map $\gamma:\{ 0,\dots, k\} \to \Z^d$ for some integer $k \geq 0$ such that 
$|\gamma(i+1)-\gamma(i)|=1$ for all $0\leq i < k$. A \textit{$*$-path} is defined similarly, 
with $|\cdot|_{\infty}$ replacing $|\cdot|$. A ($*$-)connected set $U\subset \Z^d$ is a set 
such that any points $x,y \in U$ can be joined by a ($*$-)path whose range is contained in $U$. Throughout, we use the words \textit{connected component} and \textit{cluster} interchangeably to refer to maximal connected sets.

We now state a useful criterion for the existence of  `dual' surfaces separating two sets, which is interesting in its own right. In the sequel for any $U\subset \subset\Z^d$, let 
$U^c_{\infty}$ denote the (unique) connected component of $U^c$ having infinite cardinality, 
and define $\partial_{\text{ext}}U=  \partial(U_{\infty}^c)$, the exterior boundary of $U$. 
For any two finite sets $U_1, U_2 \subset \Z^d$, we say $U_1$ is {\em surrounded} by $U_2$, denoted as $U_1 \preceq U_2$, if $U_1$ is contained in some finite connected component of $\Z^d \setminus U_2$. Notice that the relation `$\preceq$' is in fact a partial order.
\begin{lemma}[Existence of blocking interfaces]
	\label{lem:blocking_layers}
	Let $V \subset \Z^d$ be a box and $U \subset V$. Also let $\Sigma \subset V \setminus U$ be such that any $*$-path between $U$ and $\partial V$ intersects $\Sigma$ in at least $k 
	\ge 1$ points. Then there exist $*$-connected subsets $O_1, \ldots, O_{k}$ of $\Sigma$ such that $U \preceq O_1 \preceq \ldots \preceq O_{k}$.
\end{lemma}
\begin{proof}[Proof of Lemma~\ref{lem:blocking_layers}]
	An obvious consequence of the hypothesis of the lemma is that $U$ is not $*$-connected to $\partial V$ in $V \setminus \Sigma \supset U$. It then follows e.g.~by \cite[Lemma~2.1]{DeuschelPisztora96} that the 
	exterior %$\Z^d$-boundary 
	boundary of the  $*$-connected component $\mathscr C_{U}^*$ of $U$ in 
	$V \setminus \Sigma$ is itself $*$-connected, which we pick as $O_1$. Notice that $U \preceq 
	O_1$ and $O_1 \subset \Sigma$ by definition. Now observe that the hypothesis of the lemma still holds with $k-1$, $\mathscr C_{U}^* \cup O_1$ -- which is a $*$-connected set --  and $\Sigma \setminus \mathscr (C_{U}^* \cup O_1)$ substituting for $k$, $U$ and $\Sigma$ respectively. Thus, 
	by iterating the same argument $k$ times we deduce the lemma.
\end{proof}

We now review various aspects of potential theory on $\Z^d$ which will be used in the sequel.
We denote by $P_x$ the canonical law of the discrete-time (symmetric) simple random walk on 
$\Z^d$ starting at $x \in \Z^d$. We write $(X_n)_{n \geq 0}$ for the corresponding canonical 
process and $(\theta_n)_{n \geq 0}$ for the canonical time shifts. For $U \subset \Z^d$, we 
introduce the  following stopping times:  the entrance time $H_U \coloneqq \inf \{ n \ge 0 : 
X_n \in U\}$ in $U$, the exit time $T_U \coloneqq H_{\Z^d \setminus U}$ from $U$ and the hitting time $\widetilde{H}_U \coloneqq \inf \{ n \geq 1 : X_n \in U\}$ of $U$. We write
\begin{equation}
\label{eq:Green_U}
g_U(x,y) \coloneqq \sum_{n \geq 0} P_x[X_n =y, \, n < T_U], \quad \text{ for }x,y \in \Z^d
\end{equation}
for the Green function of the walk killed outside $U$. By \cite{La91},~Theorem 1.5.4, with $g=g_{\Z^d}$, one has the asymptotic formula
\begin{equation}
g(x)\coloneqq g(0,x) \sim  \Cl[c]{c:green} |x|^{2-d}, \qquad \text{as } |x| \to \infty,  \label{eq:Greenasympt}
\end{equation}
(where $\sim$ means that the ratio of both sides tends to $1$ in the given limit), for an explicit constant $\Cr{c:green}= \Cr{c:green}(d) \in (0,\infty)$ with $\Cr{c:green}(3)=\frac{3}{2\pi}$. For $K \subset\subset U \subset \Z^d$, %, where ``$\subset \subset$'' indicates the finiteness of $K$, 
we introduce the equilibrium measure of $K$ relative to $U$,
\begin{equation}
\label{eq:e_K_U}
e_{K,U}(x) \coloneqq P_x[ \widetilde{H}_K > T_U] 1_{x \in \partial K}
\end{equation}
and its total mass
\begin{equation}
\label{eq:cap_K_U} 
\text{cap}_{U}(K) \coloneqq \sum_x e_{K,U}(x),
\end{equation}
the capacity of $K$ (relative to $U$).  We will omit $U$ from all notation whenever $U=\Z^d$. One has the last-exit decomposition, see, e.g.~\cite[Lemma~2.1.1]{La91} for a proof, valid for all $K \subset\subset U\subset \Z^d$, 
\begin{equation}
\label{eq:lastexit}
P_x[H_K < T_U] = \sum_y g_U(x,y) e_{K,U}(y), \text{ for all } x \in \Z^d.
\end{equation}
Summing \eqref{eq:lastexit} over $x\in K$, one immediately sees that
	\begin{equation}\label{eq:cap_bound_Green}
	\frac{|K|}{\max_{x\in K}\sum_{y\in K} g_U(x,y)}\leq \mathrm{cap}_U(K)\leq \frac{|K|}{\min_{x\in K}\sum_{y\in K} g_U(x,y)}.
	\end{equation}
One also has the following sweeping identity (see for instance (1.12) of \cite{Sz15} when $U=\Z^d$): 
	\begin{equation}\label{eq:sweeping}
	e_{K,U}(y)=P_{e_{K',U}}[H_K<T_U, ~X_{H_K}=y], \text{ for every $K\subset K'\subset \subset U$ and $y\in \Z^d$,}
	\end{equation}
	where $P_{\mu} \coloneqq \sum_x \mu(x)P_x$ for any measure $\mu$ on $\Z^d$.
	Summing over $y$ in \eqref{eq:sweeping} gives
	\begin{equation}\label{eq:sweeping_2}
	\mathrm{cap}_U(K)=\mathrm{cap}_U(K')P_{\overline{e}_{K',U}}[H_K<T_U],
	\end{equation} 
	where $\overline{e}_{K',U}(\cdot)= {e}_{K',U}(\cdot)/\mathrm{cap}_U(K')$ is the normalized equilibrium measure.
	In particular, it follows immediately from \eqref{eq:sweeping_2} that  $\mathrm{cap}_U(K)$ 
	is increasing in $K$. %Another important consequence of the sweeping identity is that ${\rm 
	%cap}_U(\cdot)$ is sub-additive as a function on (finite) sets. 
	Note also that $\mathrm{cap}_U(K)$ is decreasing in $U$ for fixed $K$. We will also use the following 
	variational characterization of the capacity, see e.g.~\cite[(1.10)]{Sz15} and \cite[Prop.~1.9]{Sz12b} for the proof of a similar statement: for $K \subset\subset U \subset \Z^d$,
\begin{equation}
\label{eq:cap_var}
\text{cap}_{U}(K) = \frac{1}{\inf_{\nu} E_U(\nu)}, \text{ where }  E_U(\nu)=\sum_{x,y}\nu(x)g_U(x,y) \nu(y)
\end{equation}
and the infimum runs over all probability measures supported on $K$.

We now give precise bounds on the capacity of certain sets of interest. The capacity of a ball classically satisfies
\begin{equation}
\label{eq:capball}
cN^{d-2} \leq \mathrm{cap}(B_N) \leq CN^{d-2}, \text{ for all $N \geq 0$,}
\end{equation}
see, e.g.,~\cite[(2.16)]{La91}. We are typically going to work in certain (cylindrical) \textit{`tube domains'}, which we introduce now. Given $L\leq N$, the tube of length $N$ and width $L$, which we denote by  $T_N(L)$, 
is defined as the  $L$-neighborhood of the $N$-line segment $[0,N]\times\{0\}^{d-1}$. Formally,
	\begin{equation}\label{eq:def_TN}
	T_{N}(L):= ([-L,N+L]\cap \Z)\times ([-L,L]\cap\Z)^{d-1}.
	\end{equation}
	We abbreviate $T_N(0)=T_N$, which is a line of length $N $, and routinely omit the intersection with $ \Z$ from our notation below. We now derive certain 
	capacity estimates for tube domains which will be useful in the sequel. We start with the line. %Throughout this paper we write $f(N)\sim g(N)$ if the ratio $f/g$ tends to $1$ as $N \to \infty$. 
	
	\begin{lemma}[Capacity of lines] For $d=3$, one has
	\label{lem:cap_tube}
	\begin{align}
	&\label{eq:cap_line_asymp}
	\mathrm{cap}(T_N)\sim \frac{\pi}{3}\frac{N}{\log N},\quad \text{as } N \to \infty,
	\end{align}
	whereas for $d\geq 4$, there exists $\Cl[c]{c:captube}(d) \in (0,1)$  such that for all $N \geq 1$,
	\begin{align}
	&\label{eq:cap_line_asymp4d}
\Cr{c:captube}N  \leq	\mathrm{cap}(T_N) \leq \Cr{c:captube}^{-1} N.
	\end{align}
		\end{lemma}
	\begin{proof}
Using \eqref{eq:Greenasympt} with the precise value of $\Cr{c:green}(3)$, we obtain 
\begin{equation*}
%\label{eq:max_green_sum}
\sum_{y \in T_N} g(x, y) \le 2 \sum_{y \in x + T_N} g(x, y) \sim \frac 3 \pi \log N, \text{ for all $x \in T_N$}.
\end{equation*}
Substituting this into \eqref{eq:cap_bound_Green} with the choice $K = T_N$ and $U = \Z^d$ yields the asserted lower bound in \eqref{eq:cap_line_asymp}. By a similar argument, using \eqref{eq:cap_bound_Green} and noting that the Green function is summable along one-dimensional sets when $d \ge 4$, one obtains \textit{both} upper and lower bound in \eqref{eq:cap_line_asymp4d}.

It remains to show the upper bound in \eqref{eq:cap_line_asymp}. For $\delta \in (0, 1)$, letting $$T_{N}^-= T_{N}^-(\delta)  \coloneqq T_N \setminus (([0, N^{1-\delta}] \cup [N - N^{1-\delta} , N])\times \{ 0\}^2),$$ one bounds the equilibrium measure by $1$ to obtain
\begin{equation}
\label{eq:cap_upper_bnd0}
\mathrm{cap}(T_N) \le 2(1+N^{1-\delta}) + \sum_{x \in T_{N}^-}e_{T_N}(x).
\end{equation}
To take care of the sum over $T_N^-$ on the right-hand side, one sums \eqref{eq:lastexit} for $K = T_N$ and $U = \Z^d$ over $x \in T_N$ and foregoes the terms with $y \in T_N \setminus  T_{N}^-$. Together with \eqref{eq:cap_upper_bnd0} this yields \begin{equation}
\label{eq:cap_upper_bnd}
\mathrm{cap}(T_N) \le  2(1+N^{1-\delta}) + \frac{N + 1}{\inf_{y \in T_{N}^-} \sum_{x \in T_N} g(x,y)}.
\end{equation}
Now by definition of $T_N^-$ and using \eqref{eq:Greenasympt} again we obtain for any $x \in T_{N}^-$,
\begin{equation*}
%\label{eq:min_green_sum}
\sum_{y \in T_N} g(x, y) \ge 2 \sum_{y \in x + T_{N^{1-\delta}}} g(x, y) \sim (1 - \delta)\frac 3 \pi \log N.
\end{equation*}
Plugging this into \eqref{eq:cap_upper_bnd}, we get that $\limsup_{N \to \infty}\frac{{\rm cap}(T_N)}{\frac{\pi}{3}\frac{N}{\log N}} \le \frac{1}{1 - \delta}$, whereupon the upper bound in \eqref{eq:cap_line_asymp} follows by taking $\delta \to 0$.
	\end{proof}
%	The proof of Lemma~\ref{lem:cap_tube} is postponed to the Appendix. \textcolor{red}{[or do we prove it here instead?]}
\begin{remark}[Rotational invariance of asymptotic capacity for lines]
	\label{remark:cap_tube}
Let $\rm u\in \R^3$ with $|\rm u|=1$ be any unit vector. Then the asymptotic expression in \eqref{eq:cap_line_asymp} remains valid if one replaces $T_N$ 
by the line segment joining $0$ and $N{\rm u}$ discretized in the following manner. For any $x \in \R^3$, let $[x]$ denote a point in $\Z^3$ 
achieving the minimum distance between $x$ and $\Z^3$. Now let $T_{N, {\rm u}} \subset \Z^3$ 
consist of the points $[j\sqrt{3}{\rm u}]$ for all integers $j$ between $0$ and $\lceil N / 
\sqrt{3}\rceil$. Notice that it is always possible to choose the points in such a way that they 
are distinct. By this construction and the triangle inequality we have, for any $x, y \in \R^3$ such that $[x], [y] \in T_{N, \rm u}$,
\begin{equation*}
	%\label{eq:lattice_approx}
|x - y| - \sqrt{3} \le |[x] - [y]|	\le |x - y| + \sqrt{3}
\end{equation*}
and consequently $g(x, y) \sim g([x], [y])$ as $|x - y| \to \infty$. The asymptotics on the right-hand side of \eqref{eq:cap_line_asymp} 
now follow for ${\rm cap}(T_{N, {\rm u}})$ by the exact same arguments as in the proof of 
Lemma~\ref{lem:cap_tube}. Indeed, the additional $1 / \sqrt{3}$ factor appearing in the 
numerator in \eqref{eq:cap_bound_Green} owing to reduced cardinality compared to $T_N$ gets 
canceled by the $1 / \sqrt{3}$ factor appearing in the denominator because of the increased 
separation between successive points in $T_{N, {\rm u}}$. In fact, the asymptotics \eqref{eq:cap_line_asymp}  should hold for any `reasonable' discretisation of the line segment between $0$ and $N {\rm u}$.
	\end{remark}
We will need the following upper bound on the escape probability from a sufficiently dense 
subset of the line in order to derive capacity estimates for thicker tube regions. We will also 
use this result in Section~\ref{sec:coarsegrain} while proving Lemma~\ref{Claim:cap2}, which will involve porous versions of these sets (i.e.~containing holes).
\begin{lemma}[Visibility of (porous) lines, $d=3$]\label{lem:visibility}
	For all $N \geq 1$, $T \subset T_N$ and $x \in \Z^3$ such that $d_{\infty}(x,T)< N/10^3$, the following holds. If, for some $ \gamma >0$,
	\begin{equation}
	\label{eq:vis1}
	| B_r(x) \cap T| \geq \gamma |B_r(x) \cap T_N| , \text{ for all $r$ satisfying $2d_{\infty}(x,T)\leq r < N$},
	\end{equation}
	then for some $c(\gamma)>0$ and $C\in ( 1,\infty)$, one has
	\begin{equation}
	\label{eq:vis2}
	P_x[H_T =\infty] \leq  C\,\Big[\frac{ \log (1 + d_{\infty}(x,T))}{ \log N}\Big]^{c(\gamma)}.
	\end{equation}
\end{lemma}
\begin{proof}
%Throughout the proof, constants may depend implicitly on $\gamma$. 
Let $k_0, k_1$  be two integers with $k_0$ smallest so that $2d_{\infty}(x,T) \leq 
	20^{k_0}$ and $k_1$ largest such that $20^{k_1}\leq N$. Notice that $k_0=k_0(x)$ and that $k_1 > k_0$ when $N \geq C$, which we may assume. Consider the boxes $U_k \coloneqq B_{20^k}(x)$, for $k_0 \leq k\leq k_1$. By 
	\eqref{eq:vis1}, one knows that
	\begin{equation}
	\label{eq:capcompa24new}
	\big|U_k \cap T \big|  \geq \gamma |B_{20^k}(x) \cap T_N|  \geq c(\gamma) 20^k, \text{ for all $k_0 \leq k \leq k_1$}.
	\end{equation}
	Since $$\sum_{z' \in U_k \cap T} g(z,z') \le \sum_{z' \in U_k \cap T_N} g(z,z') \le Ck $$ uniformly in $z \in U_k 
	\cap T$ by \eqref{eq:Greenasympt}, it follows from 
	\eqref{eq:cap_bound_Green} and \eqref{eq:capcompa24new} that $\text{cap}(U_k \cap  
	T) \geq  c'(\gamma)20^k /k$ for all $k_0  \leq k \leq k_1$. Therefore, fixing $L \geq 1$ such that
\begin{equation}
\label{eq:vis2.1}
 2^{-1} \leq \Cr{c:green}^{-1} (g(x)|x|) \leq 2, \text{ if $|x|_{\infty} \geq L$}
\end{equation}	
it follows that for all $x$ such that $d_{\infty}(x,T) \geq L$, all	 $k_0 (= k_0(x)) \leq k < k_1$ and $y \in \partial_{\text{out}} U_k$,
	\begin{equation}
	\label{eq:capcompa25new}
	P_y\big[ H_{U_k \cap T} < 
	T_{U_{k+1}}\big]\stackrel{\eqref{eq:lastexit}}{\geq}  \inf_{z,z' \in U_k} 
	g_{U_{k+1}}(z,z')  \text{cap}(U_k \cap T) \geq c 20^{-k} 
	\frac{c'20^k}{k} \geq \frac{c}{k},
	\end{equation}
	for suitable $c,c'$ depending on $\gamma$.
	In obtaining \eqref{eq:capcompa25new}, we also used the fact that $g_{U_{k+1}}(z,z') \geq 
	cg(z,z')$ for $z,z' \in U_k$. Indeed, by the Markov property we have $g(z,z')= 
	g_{U_{k+1}}(z,z') + E_z[g(X_{T_{U_{k+1}}}, z')]$, and on the other hand, the definition of 
	$(U_k)_{k_0\leq k\leq k_1}$ readily implies that $|y - z'|\geq 5 (|z-z'| \vee L), $ for all 
	$y\in \partial U_{k+1}$, which together with \eqref{eq:vis2.1} gives $ 
	E_z[g(X_{T_{U_{k+1}}}, z')]\leq \frac45 g(z,z')$ .
	
	Now consider the process $\{ Z_k: k \geq 0\}$ on $\mathbb{N}\cup \{ \Delta \}$ defined by 
	$Z_0=0$ and for $k \geq 1$, conditionally on $Z_0,\dots Z_{k-1}$, 
	\begin{equation}
	\label{eq:capcompa26new}
	Z_{k} \coloneqq \begin{cases}
	k, & \text{ if $Z_{k-1} \neq \Delta$ and $ H_{U_{k-1} \cap T} \circ \theta_{T_{U_{k-1}}} > T_{U_k}$,}\\
	\Delta, & \text{otherwise}.
	\end{cases}
	\end{equation}
Using the strong Markov property we get that $Z_{\cdot}$ is a Markov chain under $P_x$ and \eqref{eq:capcompa25new} 
implies that $P_x[Z_{k}\neq \Delta | Z_{k-1} \neq \Delta] \leq 1-\frac{c(\gamma)}{k}$ for all $k_0< k 
\leq k_1$. It follows that for all $x$ with $d_{\infty}(x,T) \geq L$,
	\begin{equation}
	\label{eq:capcompa27new}
	\begin{split}
	P_x[H_{{T}}= \infty] &\stackrel{\eqref{eq:capcompa26new}}{\leq} P_x[Z_k \neq \Delta, k_0 \leq k \leq k_1] \\
	&\ \, \leq \prod_{k_0 < k \leq k_1} \left(1-\frac{c(\gamma)}{k}\right) 
	\leq \exp\Big\{-c(\gamma) \sum_{k_0 < k \leq k_1}\frac{1}{k}\Big\} \leq \Big(\frac{k_0}{k_1}\Big)^{c'(\gamma)},
	\end{split}
	\end{equation}
	which yields \eqref{eq:vis2} for such $x$, as $k_0 \leq C \log d_{\infty}(x,T)$ and $k_1 
	\geq c \log N$ (see above \eqref{eq:capcompa24new}) and we can assume $c'(\gamma) < 1$. To handle the case $d_{\infty}(x,T) \leq L$, the strong 
	Markov property at the time of first exit from $T_N(L)$, see \eqref{eq:def_TN}, with $L$ 
	given by \eqref{eq:vis2.1}, implies that $P_x[H_{{T}}= \infty] \leq \sup_y P_y[H_{{T}}= 
	\infty]$, with the supremum ranging over $y \in \partial_{\text{out}}T_N(L)$ and 
	\eqref{eq:capcompa24new} still follows from \eqref{eq:capcompa27new} as $d_{\infty}(y,T) \geq L$. This completes the proof.
\end{proof}

We now move on to the capacities of tubes whose width is a fractional power of their length. In the sequel, let
$T_{N}^{\delta}= T_N(N^\delta)$ for $\delta > 0$.%, $k >0$ (cf.~\eqref{eq:def_TN}), and abbreviate $T_{N}^{\delta}=T_{N,1}^{\delta}$.
\begin{lemma}[Capacity of tubes, $d=3$]\label{lem:cap_tube2} Let $T_{N}^{\delta}= T_N(N^\delta)$ for $\delta > 0$ (cf.~\eqref{eq:def_TN}). There exist $\Cl{thick-tube} \in 
[1,\infty)$ and $\Cl[c]{thick-tube'} \in (0,1)$ such that for every $\delta\in(0,\frac12)$ and $N \geq C(\delta)$,
		\begin{align}
		&\mathrm{cap}(T_N^{\delta})\leq (1 + \Cr{thick-tube} \delta^{\Cr{thick-tube'}})\,\mathrm{cap}(T_N), \label{eq:cap_tube_d=3}\\[0.2em]
		&\label{eq:cap_tube_d=3bis}
		\mathrm{cap}_{T_{N}^{2\delta}}(T_N^{\delta})\leq \Cr{thick-tube}\delta^{-1}\,\mathrm{cap}(T_N^{\delta}).
		\end{align}
	\end{lemma}
\begin{proof}
We claim that for every $x\in \partial T_N^{\delta}$ and $y\in \partial_{\text{out}} 
T_{N}^{2\delta}$, one has
\begin{align}
	\label{eq:scape_line} P_x[H_{T_N}=\infty]\leq C \delta^{c} \mbox{ and } P_y[H_{T_N^{\delta}}=\infty] \geq c\delta, \text{ for all $N\geq C(\delta)$.}
\end{align}
 Before proving \eqref{eq:scape_line} let us deduce the 
lemma from it. By the sweeping identity \eqref{eq:sweeping},
\begin{equation*}
	\mathrm{cap}(T_N)=P_{e_{T_N^{\delta}}}[H_{T_N}<\infty]=\mathrm{cap}(T_N^{\delta})\,P_{\overline{e}_{T_N^{\delta}}}[H_{T_N}<\infty],
\end{equation*}
and \eqref{eq:cap_tube_d=3} follows directly from the upper bound in \eqref{eq:scape_line}. Also, by decomposing on the first exit time of $ T_{N}^{2\delta}$, one finds that for all $x\in T_N^{\delta}$, 
\begin{equation*}
	%e_{T_N^{\delta}}(x)=
	P_x[\widetilde{H}_{T_N^{\delta}}=\infty]=
	\sum_{y \in \partial_{\text{out}}T_{N}^{2\delta}} P_x[\widetilde{H}_{T_N^\delta}> H_{( T_{N}^{2\delta})^c},~X_{ H_{( T_{N}^{2\delta})^c}}=y] \,P_y[H_{T_N^{\delta}}=\infty],
\end{equation*}
which combined with the lower bound in \eqref{eq:scape_line} implies that $e_{T_N^{\delta}}(x)\geq c\,\delta \,e_{T_N^{\delta}, T_{N}^{2\delta}}(x)$. Summing over $x\in T_N^{\delta}$ yields \eqref{eq:cap_tube_d=3bis}. We proceed to the proof of the bounds in \eqref{eq:scape_line}.

\medskip

\noindent{\em The upper bound in \eqref{eq:scape_line}.} For all $\delta \in (0, 1/2)$ the hypotheses of 
Lemma~\ref{lem:visibility} hold with $T = T_N$ and $\gamma=\frac1{10}$ for any $x \in \partial T_N^{\delta}$ and $N \geq C(\delta)$, whence the upper bound in \eqref{eq:scape_line} follows from \eqref{eq:vis2}.

\medskip

\noindent{\em The lower bound in \eqref{eq:scape_line}.} Below we will use $B_L'$ to denote the two-dimensional box $\{0\} \times[-L, L]^2$. First of all, notice that for any $y 
=(y_1, y_2, y_3) \in \partial_{\text{out}} T_{N}^{2\delta}$, either $(y_2, y_3) \in \partial B_{ \lfloor N^{2\delta}\rfloor + 1}'$ 
or $y \in \{-\lfloor N^{2\delta} \rfloor - 1, N + \lfloor  N^{2\delta}\rfloor + 1\} \times [- N^{2\delta},  N^{2\delta}]^{2}$. We deal with the former case first. To this end let us consider the projection $X'=(X'_n)_{n\ge 0}$ of 
$(X_n)_{n\ge 0}$ onto its last two coordinates, which has the law of a (lazy) simple random 
walk in $\Z^2$. Let $H_{U}'$ denote the entrance time in $U$ for $X'$ and abbreviate $H'_{{\rm in}} = H'_{B'_{N^\delta}}$ and $H'_{{\rm out}} = H'_{\partial B'_{10N}}$.
Applying Exercise~1.6.8 in \cite{La91}, we get
\begin{equation}
\label{eq:gambler_bnd}
%\begin{split}
P_{y'}[H'_{{\rm out}} < H'_{{\rm in}}] \ge \frac{\log (N^{2\delta}) - \log (N^{\delta}) - C}{\log (10N) - \log (N^{\delta}) + C} \ge c\delta
%\end{split}
\end{equation}%\begin{equation}
%\label{eq:gambler_bnd}
%P_{y'}[H'_{{\rm out}} < H'_{{\rm in}}] \ge c \delta, \text{ if } |y'| \ge  N^{2\delta},
%\end{equation}
whenever $N \geq C(\delta)$ and $|y'| \ge N^{2\delta}$. Notice that the errors are of constant order as opposed to $O(N^{-\delta})$ as in \cite{La91} since we are working with $\ell^{\infty}$-balls instead of $\ell^2$-balls. Now since 
$\{(x_2, x_3): x \in T_{N^\delta}\} \subset B'_{N^\delta}$, the inclusion
\begin{equation*}
\label{eq:hitting_decomp}
\{H'_{{\rm out}} < H'_{{\rm in}}, H_{T_{N}^{\delta}} \circ \theta_{H'_{{\rm out}} }= \infty\} \subset \{H_{T_{N}^{\delta}}= \infty\}
\end{equation*}
holds and consequently, by the strong Markov property, we have
\begin{equation}
	\label{eq:hitting_bnd}
	P_y[H_{T_N^{\delta}}=\infty] \ge P_{y'}[H'_{{\rm out}} < H'_{{\rm in}}] \,\inf_{z: \, d_{\infty}(z, T_{N}) \ge 10N} P_z[H_{T_{N}^{\delta}}  = \infty].
\end{equation}
However, by the last-exit decomposition \eqref{eq:lastexit} and \eqref{eq:Greenasympt},
\begin{equation*}
	P_z[H_{T_{N}^{\delta}}  < \infty] \le {\rm cap}(T_N^{\delta}) \max_{z' \in T_{N}^{\delta}}g(z, z') \le C \frac{{\rm cap}(T_N^{\delta})}{N}, \text{ when }d_{\infty}(z, T_{N}) \ge 10N.
\end{equation*}
In view of Lemma~\ref{lem:cap_tube} and \eqref{eq:cap_tube_d=3} which, let us recall, requires 
only the upper bound in \eqref{eq:scape_line}, the right-hand side in the previous display is 
bounded by $\tfrac{C}{\log N}$. Plugging this and the bound \eqref{eq:gambler_bnd} into 
\eqref{eq:hitting_bnd}, we deduce the lower bound in \eqref{eq:scape_line} in the case $(y_2, 
y_3) \in \partial B_{ N^{2\delta}}'$.

To deal with the case when $y \in \{- \lfloor N^{2\delta} \rfloor - 1, N + \lfloor N^{2\delta} \rfloor + 1\} \times [- N^{2\delta}, 
 N^{2\delta}]^{2}$ let us assume without loss of generality that $y_1 = -\lfloor N^{2\delta} \rfloor - 1$ and note that, by means of the strong Markov property and the previous case, it suffices to show that
\begin{equation}
	\label{eq:escape_left}
	P_y[H' < H_{T_{N}^{\delta}}] \ge c, \text { where } H'\coloneqq H'_{\partial B'_{ N^{2\delta}}}.
	\end{equation}
To this end, let $(X^1_n)_{n \ge 0}$ denote the projection of $(X_n)_{n \ge 0}$ onto its first coordinate and consider the event
\begin{equation*}
	G \coloneqq \{H_{-\lfloor5N^{2\delta}\rfloor}(X^1) < H_{-\lfloor N^\delta\rfloor}(X^1)\}
	\end{equation*}
which has a constant positive probability under $P_y$ by the standard gambler's ruin estimate (one could even afford to replace $\lfloor N^\delta\rfloor$ by $\frac12\lfloor N^{2\delta}\rfloor$ on the right-hand side). It follows from the definition of $T_N^{\delta}$ that $H_{T_{N}^{\delta}} > H_{-\lfloor5N^{2\delta}\rfloor}(X^1) \eqqcolon H^1$ on $G$ and hence
\begin{equation}
	\label{eq:reach_bndry_before}
G \cap \{H^1 \ge  H'\} \subset \{H' < H_{T_{N}^{\delta}}\}.
\end{equation}
On the other hand, we have $$G \cap \{H^1 <  H'\} \subset \{X_{H^1}^1 = -\lfloor5N^{2\delta}\rfloor, X'_{H^1} \in B'_{ N^{2\delta}}\}.$$
An implication of the condition on the right hand side above is that $B_{4N^{2\delta}}(X_{H^1}) \cap T_N^{\delta} = \emptyset$, whereas $B'_{ N^{2\delta}} \subset B'_{4N^{2\delta}}(X'_{H^1})$ for all $N \ge C(\delta)$.  Denoting $H \coloneqq  H'_{\partial B'_{4N^{2\delta}}(X_0')} \circ \theta_{ H^1}$, we therefore have
\begin{equation}
\label{eq:reach_bndry_after}
(G \cap \{H^1 <  H'\}) \cap \{X'_{H} \subset \partial B'_{4N^{2\delta}}(X'_{H^1})\} \subset \{H' < H_{T_{N}^{\delta}}\}.
\end{equation}
However, since all two-dimensional axial projections of $(X_n)_{n \ge 0}$ have the same law, we can deduce via a union bound that $P_x[X_{H_{\partial B_{4N^{2\delta}}(x)}}' \in \partial 
B'_{4N^{2\delta}}(x)] \ge c$ uniformly for all $x \in \Z^d$. Hence, in view of \eqref{eq:reach_bndry_after}, we get, applying the strong Markov property at time $H^1$,
\begin{equation*}
	P_y[H'  < H_{T_{N}^{\delta}} \, | \, G, H^1 <  H'] \ge c,
	\end{equation*}
with $y$ as in \eqref{eq:escape_left}. Combined with \eqref{eq:reach_bndry_before} and the fact that $P_y[G] \ge c$, this yields \eqref{eq:escape_left} and consequently the lower bound in \eqref{eq:scape_line} in this case.
\end{proof}

We conclude this section by reviewing some important features of Gaussian free fields. For $U \subset \Z^d$, we write $\P_U$ for the law of the centered Gaussian process with covariance $g_U(\cdot, \cdot)$, with $g_U$ as given by \eqref{eq:Green_U} (in particular $\P=\P_{\Z^d}$ following our above convention). Notice that under $\P_U$, the field $\varphi$ is almost surely $0$ on $\Z^d\setminus U$. %For finite $U$, one can alternatively define $\P_U$ given by
%	\begin{equation}\label{eq:GFF_density}
%	\frac{\P_U(d\varphi)}{d\varphi} \propto \exp\left(-\frac{1}{2}\mathcal{E}(\varphi,\varphi)\right) 1_{\varphi\equiv 0 \text{ on } U^c},
%	\end{equation} 
%	where $d\varphi$ represents the Lebesgue measure on $\R^U$. 
For $U \subset \mathbb{Z}^d$, we further introduce the Gaussian fields (functions of $\varphi$)
	\begin{equation}\label{eq:decomp}
\xi^U_x\coloneqq  E_x\big[\varphi_{X_{T_{U}}}\big] =\sum_{y} P_x[X_{T_U}=y]\varphi_y, \quad  \psi^U_x \coloneqq \varphi_x- \xi_x^U, \quad  \text{ for }x \in \mathbb{Z}^d.
	\end{equation}
The field $\xi^U$ will be referred to as the harmonic average of $\varphi$ in $U$ and $\psi^U$ as the local field in $U$. %{\color{red} It seems convenient, for the purpose of measurability in Defintion~\ref{def:goodevent}, to define the harmonic average in terms of the free field values at $\mathbb M^d$.} 
Plainly, $\xi^U_x=\varphi_x$ (and therefore $\psi^U_x=0$) for all $x\in \Z^d\setminus U$. As in \cite[Lemma~1.2]{RoS13}, one observes that $\xi^U$ is independent of $\psi^U$ and that $(\psi^U_x)_{x\in \mathbb{Z}^d}$ has law $\mathbb{P}_{U}$ under $\mathbb{P}$.

	It will be convenient at times to consider a certain extension of the above setup. Let $\mathbb{M}^d$ denote the set of mid-points of the edges of $\Z^d$. We regard $\tilde{\Z}^d \coloneqq \Z^d\cup\mathbb{M}^d$ as the graph obtained from $\Z^d$ by splitting every edge of $\mathbb{Z}^d$ into two (and adding the corresponding mid-point to the vertex set).
% $xy$ into two edges $xm$ and $my$, where $m$ is the midpoint of $xy$. By considering the random walk on this graph (which is simply the random walk on $\Z^d$ when restricted to it), one can define an extended GFF taking values on $\Z^d\cup\mathbb{M}^d$. 
Let $\tilde{X}= (\tilde{X}_n)_{n \geq 0}$ be the discrete-time random walk on $\tilde{\Z}^d$, which at each step jumps with uniform probability to one of its neighboring vertices in $\tilde{\Z}^d$. Let $\tilde{P}_{\tilde{x}}$ denote the canonical law of $ \tilde{X}$ with starting point $ \tilde{X}_0 = \tilde{x} \in \tilde{\Z}^d$. By suitable extension of $\mathbb{P}$, one defines a centered Gaussian field $\tilde{\varphi}=(\tilde{\varphi}_{ \tilde{x}})_{\tilde{x} \in \tilde{\Z}^d}$ such that 
 \begin{equation}
 \label{eq:phi_extend1}
 \mathbb{E}[\tilde{\varphi}_{ \tilde{x}}\tilde{\varphi}_{ \tilde{y}}]\coloneqq \frac12 \sum_{n \geq 0}\tilde{P}_{\tilde{x}}[\tilde{X}_n= \tilde{y}], \text{ for }\tilde{x}, \tilde{y} \in \tilde{\Z}^d.
 \end{equation}
Indeed, it follows from \eqref{eq:phi_extend1} and \eqref{eq:Green_U} that $\mathbb{E}[\tilde{\varphi}_{ {x}}\tilde{\varphi}_{ {y}}] =g(x,y)$ whenever $x,y \in \mathbb{Z}^d$, whence 
\begin{equation}
 \label{eq:phi_extend2}
\tilde{\varphi}\vert_{\Z^d}= \varphi. 
\end{equation}
The decomposition \eqref{eq:decomp} also extends and one obtains that
	\begin{equation}\label{eq:Markov_GFFtilde}
	\tilde{\varphi}=\tilde{\xi}^{V}+\tilde{\psi}^{ V},  \text{ for }   V \subset \tilde{\Z}^d, \text{ where $\tilde{\xi}^V_{\tilde{x}}:=\tilde{E}_{\tilde{x}}\big[\tilde{\varphi}_{\tilde{X}_{T_V}}\big]$};
	\end{equation}
here $\tilde{\xi}^{V}$ and $\tilde{\psi}^{ V}$ are independent Gaussian fields and 
\begin{equation}
 \label{eq:phi_extend3}
\E\big[\tilde{\psi}_{\tilde{x}}^{ V} \tilde{\psi}_{\tilde{y}}^{ V}\big]=\frac12 \sum_{n \geq 0}\tilde{P}_{\tilde{x}}[\tilde{X}_n= \tilde{y}, n < T_V], \text{ for }\tilde{x}, \tilde{y} \in \tilde{\Z}^d,
\end{equation} 
where with hopefully obvious notation, $T_V=T_V(\tilde{X})$ denotes the exit time of $\tilde{X}$ from $V$. The analogue of the restriction property \eqref{eq:phi_extend2} for the harmonic extension is then the following. For $U \subset \mathbb{Z}^d$, defining $\tilde{U}\coloneqq U \cup \{ \tilde{x} \in \tilde{\Z}^d : \exists x \in U \text{ s.t. } |x-\tilde{x}|= \frac12\}$, noting that $X_{T_{U}}$ under $P_x$ has the same law as $\tilde{X}_{T_{\tilde{U}}}$ under $\tilde{P}_x$ for any $x \in \Z^d$ and using \eqref{eq:phi_extend2}, one sees that
\begin{equation}
 \label{eq:phi_extend4}
\tilde{\xi}^{\tilde{U}}\vert_{\Z^d}= \xi^U. 
\end{equation}

We conclude with a particular instance of \eqref{eq:Markov_GFFtilde}, which will prove useful on several occasions. Let
\begin{equation}
	\label{eq:LB1_dim4}
	\tilde{\varphi}=\hat{\xi}+\hat{\psi}%\hat{\xi}^{V}+\hat{\psi}^{ V},  \text{ for }   V \subset \tilde{\Z}^d
	\end{equation}
	be the decomposition \eqref{eq:Markov_GFFtilde} corresponding to the choice $V\coloneqq \Z^d (\subset \tilde{\Z}^d
)$, i.e. $\hat{\xi}_{x} = \hat{\xi}_{{x}}^{\Z^d}=\frac{1}{2d}\sum_{\substack{m\in \mathbb{M}^d : m \sim x}} ~\tilde{\varphi}_m$ if ${x} \in \mathbb{Z}^d$ (where $\sim$ refers to neighbors in $\tilde{\mathbb Z}^d$), $\hat{\xi}_{m} = \hat{\xi}_{m}^{\Z^d}= \tilde{\varphi}_{m}$ if $m\in \mathbb{M}^d$, and by \eqref{eq:phi_extend3},
\begin{equation}
	\label{eq:LB2_dim4}
	(\hat{\psi}_x)_{x\in \Z^d} \text{ is a field of i.i.d.~centered Gaussian variables with variance $1/2$ each.}%\hat{\xi}^{V}+\hat{\psi}^{ V},  \text{ for }   V \subset \tilde{\Z}^d
	\end{equation}

	\section{Lower bounds}\label{sec:lower}
	%\todo{Explain organization of the section.}
	The main result of this section is the following
	\begin{theorem}[Lower bounds] \label{P:lb} The following holds:
		\begin{enumerate}
			\item[i)] If $d=3$, then 
			\begin{align}
			&\text{for all $h\geq h_{*}$, }\liminf_{N\to\infty}\, \frac{\log N}{N} \log \P[\lr{}{\varphi\geq h}{0}{\partial B_N}%,\, \nlr{}{\varphi\geq h}{0}{\infty}
			] \geq -\frac{\pi}{6}(h-h_*)^2,  \label{eq:main_lbsubcrit_d=3} \\
			&\text{for all $h\leq h_*$, }\liminf_{N\to\infty}\, \frac{\log N}{N} \log \P[\lr{}{\varphi\geq h}{0}{\partial B_N},\, \nlr{}{\varphi\geq h}{0}{\infty}
			] \geq -\frac{\pi}{6}(h- h_{*})^2. \label{eq:main_lbsupercrit_d=3}
			\end{align}
			\item[ii)] If $d \geq 4$, then for all $h \in \R$,
			\begin{align}
			&\liminf_{N\to\infty}\, \frac{1}{N} \log \P[\lr{}{\varphi\geq h}{0}{\partial B_N},\, \nlr{}{\varphi\geq h}{0}{\infty}
			] > -\infty \label{eq:main_lb_dgeq4}.
			%&\text{for all $h< h_*$, }\liminf_{N\to\infty}\, \frac{\log N}{N} \log \P[\lr{}{\varphi\geq h}{0}{\partial B_N},\, \nlr{}{\varphi\geq h}{0}{\infty}
			%] \geq -\frac{\pi}{{\color{red}6}}(h- h_{**})^2. \label{eq:main_lbsupercrit_dgeq4}
			\end{align}
		\end{enumerate}
		Moreover, the bounds \eqref{eq:main_lbsupercrit_d=3} %for $h < h_*$ 
		and \eqref{eq:main_lb_dgeq4} also hold for the events $\textnormal{LocUniq}(N,h)^{c}$ and $\text{2-arms}(N,h)$ (see~\eqref{eq:def_locuniq} and \eqref{eq:intro2arm}) in place of $ 
		\displaystyle \{ \lr{}{\varphi\geq h}{0}{\partial B_N},\, \nlr{}{\varphi\geq 
			h}{0}{\infty}\}$.
	\end{theorem}
	%\todo{We could choose to be a bit quantitative in $h$, i.e. the liminf is greater or equal than $c(d)h^2$ (check) in \eqref{eq:main_lb_dgeq4}. }
	%Note: alternatively we could also define something like $d_{\pm}(h', h)= |h'-h|$, if $h' \geq  h$ (for $+$) and if $h' \leq  h$ (for $-$) and $\infty$ otherwise and replace \eqref{eq:main_lbsubcrit_d=3}, \eqref{eq:main_lbsupercrit_d=3} by the following, valid for all $h \in \R$:
	
	%	\begin{align}
	%	&\liminf_{N\to\infty}\, \frac{\log N}{N} \log \P[\lr{}{\varphi\geq h}{0}{\partial B_N}%,\, \nlr{}{\varphi\geq h}{0}{\infty}
	%	] \geq -\frac{\pi}{{\color{red}6}}d_+(h, \bar h)^2,  \label{eq:main_lbsubcrit_d=3bis} \\
	%	&\liminf_{N\to\infty}\, \frac{\log N}{N} \log \P[\lr{}{\varphi\geq h}{0}{\partial B_N},\, \nlr{}{\varphi\geq h}{0}{\infty}
	%	] \geq -\frac{\pi}{{\color{red}6}}d_-(h, h_{**})^2. \label{eq:main_lbsupercrit_d=3bis}
	%	\end{align}
%	\noindent \fbox{\begin{minipage}{40em}
%			Organization of this section: \\
%			
%			Perhaps it's not such a good idea to divide into ``subcritical phase'' (Section 3.2) and ``supercritical phase'' (Section 3.3) since the bounds for $d \geq 4$ don't witness this distinction at all (and for instance, the current ``subcritical'' lower bound when $d \geq 4$ follows directly from the ``supercritical'' one, which is actually valid for all $h \in \R$). Rather, we could do the case $d=3$ in Section 3.2 and the case $d \geq 4$ in Section 3.3.
%			
%	\end{minipage}}

	\subsection{General entropic lower bound}\label{sec:lower_tools}
	The following lemma will be used in the course of proving Theorem~\ref{P:lb}, but is of independent interest. The  lower bound it asserts in \eqref{eq:entrop_lower2} will follow by a change of measure argument, see e.g.~the proof of Theorem 2.1 in \cite{Sz15}, or Lemma~2.3 in \cite{BDZ95}, for results of a similar flavor. Given an event $A \in \mathcal B(\R^K)$, $K\subset \Z^d$ and a height parameter $h\in \R$, we define %the {\em shifted event}
	\begin{equation}\label{eq:LBevent}
	A^h = A^h(\varphi) = \{\varphi_{|_K}-h\in A\},
	\end{equation}
	%For example, if $A = \{0\xleftrightarrow[]{\varphi \ge 0}\infty\}$ then $A^h = \{0\xleftrightarrow[]{\varphi \ge h}\infty\}$. %Below $\P_{G}$ represents the distribution of the GFF in $G$ with $0$ boundary condition. \textcolor{red}{[should have already been defined in Section~\ref{sec:preliminaries}]}. 
where, with hopefully obvious notation $\varphi_{|_K}-h$ refers to the field (restricted to $K$) shifted by $-h$ coordinatewise.

	\begin{lemma}[Entropic lower bound] \label{lem:entrop_lower}
		Let $K_N\subset \subset U_N\subset \Z^d$ be subsets with $\mathrm{cap}_{U_N}(K_N)\to\infty$. Let $A_N\in\mathcal{B}(\R^{K_N})$ and $I\subset \R$ be an interval such that, for every $h'\in I$,
		\begin{equation}\label{eq:entrop_lower1}
		\P_{U_N}[A^{h'}_N]\to 1.
		\end{equation}
		Then for every $h\notin I$,
		\begin{equation}\label{eq:entrop_lower2}
		\liminf_{N \to \infty} \frac{1}{\mathrm{cap}_{U_N}(K_N)}\log \P_{U_N}\big[A^h_N\big]\geq -\frac{1}{2}d(h, I)^2.
		\end{equation}
	\end{lemma}

	\begin{proof}
		Recall the following fact, which is a consequence of Jensen's inequality, see e.g.~%~\cite[p.~76]{DS89}, or 
		the discussion following (2.7) in \cite{BDZ95} for a proof. Given two probability measures $\tilde{\P}$ and $\P$ such that $\tilde{\P}$ is absolutely continuous with respect to $\P$, and an event $A$ with positive $\tilde{\P}$-probability, one has
		\begin{equation}\label{eq:gen.entrop.lower}
		\P[A]\geq\tilde{\P}[A]e^{-(1/\tilde{\P}[A])(H(\tilde{\P}|\P)+1/e)}.
		\end{equation}
		where $H(\tilde{\P}|\P):=\tilde{\E}\left[\log\frac{d\tilde{\P}}{d\P}\right]$ is the relative entropy of $\tilde{\P}$ with respect to $\P$. Abbreviate $K=K_N$ and $U=U_N$ in the sequel. Pick $h\notin I$ and $h' \in I$. Using \eqref{eq:lastexit}, it follows by an application of the Cameron-Martin theorem, see e.g.~\cite[Theorem 14.1]{janson_1997} with the choice (in the notation of \cite{janson_1997}) $\xi \coloneqq (h-h') \big\langle e_{K,U}, \varphi \big\rangle  \in L^2(\P)$, that $\tilde{\P}_U$ defined by
		\begin{equation}
		\label{eq:shift_density}
		\frac{d\tilde{\P}_U}{d\P_U} = \exp \Big\{  (h-h') \big\langle e_{K,U}, \varphi \big\rangle - \frac{ (h-h')^2}{2} \text{cap}_U(K)\Big\}
		\end{equation}
		is a probability such that $\varphi$ has the same law under $\tilde \P_U$ as $\varphi+f $ under $\P_U$, where 
		\begin{equation}
		\label{eq:shift_f}
		f(x)=(h-h')P_{x}[H_K < T_U], \ x \in \Z^d.
		\end{equation}
		(indeed observe to this effect that in (14.3) of \cite{janson_1997}, one obtains $\rho_{\xi}(\varphi_{\cdot})=\varphi_{\cdot}+ \mathbb{E}_U[\xi\varphi_{\cdot}]= \varphi_{\cdot}+f_{\cdot}$  by \eqref{eq:lastexit}).
		In particular, $f=h-h'$ on $K$, whence $\tilde{\P}_U[A^h(\varphi)]= \P_U[A^h(\varphi+f)] = \P_U[A^{h'}(\varphi)]$ which tends to~$1$ as $N \to \infty$ by \eqref{eq:entrop_lower1}. Moreover, by \eqref{eq:shift_density} and \eqref{eq:shift_f}, noting that $ \tilde{\E}_U\big[ \big\langle e_{K,U}, \varphi \big\rangle \big] =  \big\langle e_{K,U}, f \big\rangle  $, one sees that
		\begin{equation*}
		\begin{split}
		H\big(\tilde{\P}_U\big|\P_U\big)&= (h-h') \tilde{\E}_U\big[ \big\langle e_{K,U}, \varphi \big\rangle \big]   - \frac{ (h-h')^2}{2} \text{cap}_U(K) 
		%=  (h-h') \big\langle e_{K,U}, f \big\rangle - \frac{ (h-h')^2}{2} \text{cap}_U(K)
		=  \frac{ (h-h')^2}{2} \text{cap}_U(K).
		\end{split}
		\end{equation*}
		Applying \eqref{eq:gen.entrop.lower}, taking logarithms and letting $N \to \infty$ now readily yields \eqref{eq:entrop_lower2}, since $h'\in I $ was arbitrary.
	\end{proof}

	\subsection{Lower bounds for $d=3$}\label{sec:lower_sub}
	%{\bf The case $\mathbf{d\geq 4}$.} The exponential lower bound is an easy consequence of the FKG inequality. 

	%\noindent{\bf The case $\mathbf{d = 3}$.} %({\color{red}recall at this point that our asymptotic formula for $d = 3$ as given in Theorem~\ref{thm:main_d=3} should hold for all values of $h$}). 
	In this section, we show the lower bounds \eqref{eq:main_lbsubcrit_d=3} and \eqref{eq:main_lbsupercrit_d=3}, which will both follow from an application of Lemma~\ref{lem:entrop_lower}, with carefully chosen events $A^h$ in \eqref{eq:LBevent} as to implement sufficiently cost-effective strategies for connection. The asserted bound \eqref{eq:main_lbsupercrit_d=3} bears the additional difficulty that the event in question is not monotone, which makes its proof more involved than \eqref{eq:main_lbsubcrit_d=3}.

\begin{proof}[Proof of \eqref{eq:main_lbsubcrit_d=3} and \eqref{eq:main_lbsupercrit_d=3}]
We begin with the proof of \eqref{eq:main_lbsubcrit_d=3}. Recalling the notation from \eqref{eq:def_TN}, define the thin cylinder $K_N \coloneqq T_N(L)$, with
\begin{equation}
\label{eq:LBpf1}
 L\coloneqq \lfloor N^{\delta}\rfloor, \text{ for $\delta \in (0,\textstyle\frac1{d-1})$,}
\end{equation}
and let $F_N^-= \{ 0\} \times [-L,L]^{d-1}$, $F_N^+= \{ N\} \times [-L,L]^{d-1}$. Note that $F_N^{\pm} \subset K_N$ and that $0\in F_N^-$, while $F_N^+ \subset \partial B_N$. For $h \in \R$, consider the event
\begin{equation}
\label{eq:LBpf2}
A_N^h \coloneqq A_{N,L}^h \coloneqq  \{ \text{$F_N^-$ and $F_N^+$ are connected by a path in $\{ \varphi \geq h\} \cap K_N$}  \},
\end{equation}
which is of the form \eqref{eq:LBevent}. Let us first check that 
\begin{equation}
	\label{eq:lower_sup_500}
	\text{$\lim_N \P[{A}^{h'}_N]= 1$ for all $h' < h_*$.} 
	\end{equation}
In view of \eqref{eq:def_locuniq}, one readily sees that
$$
\bigcap_{k \in \{-L, -L+1,\dots, N+L\} }\big( \{\lr{}{\varphi\geq h'}{B_{L/8}}{\partial B_{L}}\} \cap \text{LocUniq}(L/4,h') \big) \circ \tau_{ke_1}\subset A_N^{h'},
$$
where $\tau_x$, $x \in \Z^d$, denote the canonical space shifts on $\R^{\Z^d}$.
Now, %using the bounds (1.6) and (1.7) of \cite{DCGRS20}, which hold for $\alpha = h'$, and \eqref{eq:LBpf1}, 
combining the bounds of \eqref{eq:EXISTUNIQUE} and in view of \eqref{eq:LBpf1} one deduces that $\P[\lr{}{\varphi\geq h'}{B_{L/8}}{\partial B_{L}}] \geq  1-\exp\{-cN^{\delta \Cr{c:strexp}(h')}\}$. Together with \eqref{eq:def_hbar} and a union bound, and in view of \eqref{eq:LBpf1}, this is easily seen to imply $\lim_N \P[A^{h'}_N]= 1$, as desired. Lemma~\ref{lem:entrop_lower} thus applies with $U_N = \Z^d$, $I:=(-\infty, h_*)$, and yields for that for all $h\geq h_*$,
	\begin{equation}\label{eq:pf_low_sub}
	\liminf_{N \to \infty} \frac{\log N}{N} \log\P\big[A^h_N\big]\geq -\frac{\pi}{6}(h- h_*)^2 (1 + \Cr{thick-tube}\delta^{\Cr{thick-tube'}}),
	\end{equation}
	using that $\mathrm{cap}(K_N)\sim (1 + \Cr{thick-tube}\delta^{\Cr{thick-tube'}})\frac{\pi}{3}\frac{N}{\log N}$ as $N \to \infty$, as follows directly from \eqref{eq:cap_tube_d=3} and \eqref{eq:cap_line_asymp}. 
	
	Next, by the FKG-inequality for $\varphi$, one knows that $\P[ F_N^{-} \subset \{\varphi \geq h\}  ]\geq \exp\left\{-C(h)L^{d-1}\right\}.$ In particular, in view of \eqref{eq:LBpf1}, it follows that
	\begin{equation}\label{eq:pf_low_sub2}
	\liminf_{N \to \infty} \frac{\log N}{N} \log\P\big[F_N^{-} \subset \{\varphi \geq h\}\big]=0, \text{ for all $h \in \R$.}
	\end{equation} 
	Recalling $A_N^h$ from \eqref{eq:LBpf2} and the fact that $0\in F_N^-$, one observes that $A_N^h\cap \{ F_N^{-} \subset \{\varphi \geq h\}\}$ implies $\{0\xleftrightarrow[]{\varphi \ge h}\partial B_N\}$. Hence, \eqref{eq:main_lbsubcrit_d=3} follows directly from \eqref{eq:pf_low_sub}, \eqref{eq:pf_low_sub2} and the FKG-inequality upon letting $\delta \to 0$.  
	
\medskip	
We now show \eqref{eq:main_lbsupercrit_d=3}. For arbitrary $\delta \in (0,\frac13)$, let $D_N \subset U_N \subset K_N$ be defined as $D_N\coloneqq T_N( N^\delta)$ $U_N\coloneqq T_N( 
N^{2\delta})$, $K_N\coloneqq T_N( N^{3\delta})$ (see \eqref{eq:def_TN} for notation) and abbreviate $\xi^N=\xi^{U_N}$, $\psi^N=\psi^{U_N}$, cf. \eqref{eq:decomp}. For an additonal parameter $\varepsilon > 0$ and $h < h_* +\varepsilon$, consider the event
	\begin{equation}\label{eq:lower_sup_1}
	C_N^{h}:=\big\{\nlr{}{\varphi\geq h}{\partial_{\text{out}}U_N}{\partial K_N}\big\} \cap \big\{\textstyle \inf_{D_N} \displaystyle \xi^N \geq -(h_* + \varepsilon - h)\big\},
		\end{equation}
and note that $C_N^{h}$ is of the form \eqref{eq:LBevent} with $K=K_N$; observe to this 
effect that due to \eqref{eq:decomp}, the condition $\xi_{\cdot}^N \geq -(h_* + \varepsilon - 
h)$ can be recast as $\sum_{y\in U_N} P_{\cdot}[X_{H_{U_N}}=y](\varphi_y - h)\geq -(h_* + 
\varepsilon)$, i.e. as a condition on the field $\varphi_{\cdot}-h$ restricted to $U_N 
(\subset K_N)$.

	%A_N^h:=\big\{\nlr{}{\varphi\geq 0}{T_N( N^{2\delta})}{\partial T_N(3N^\delta)}\big\} \cap \Big\{\sum_{y\in T_N( N^{2\delta})} P_x[X_{H_{T_N( N^{2\delta})}}=y]\varphi_y \geq -(h_* + \varepsilon)\Big\}.
	We now argue that $C_N^{h'}$ is typical as $N \to \infty$ for every $h'\in (h_{*} , h_* + 
	\varepsilon)$. For such $h'$, the probability of $\{\lr{}{\varphi\geq 
		h'}{\partial_{\text{out}}U_N}{\partial K_N}\}$ vanishes as $N \to \infty$ by \eqref{eq:def_h**} and a union bound. Next, for $x \in D_N$ one has
\begin{equation*}
\begin{split}
\E[ (\xi_x^N)^2] &\stackrel{\eqref{eq:decomp}}{=} \sum_{y,z} P_x[X_{T_{U_N}}=y] P_x[X_{T_{U_N}}=z]g(y,z)=  \sum_{y} P_x[X_{T_{U_N}}=y] E_x[g(X_{T_{U_N}},y)] \\
& \ \,=  \sum_{y} P_x[X_{T_{U_N}}=y] g(x,y) \stackrel{\eqref{eq:Greenasympt}}{\leq} C\, \text{dist}(x, U_N^c)^{-(d-2)} \leq CN^{-2\delta(d-2)},
\end{split}
\end{equation*}
where the equality in the second line is obtained by applying the strong Markov property at 
time $T_{U_N}$, noting that $x \in U_N$ whereas $y \notin U_N$. It then follows from a union 
bound over $D_N$ and a standard Gaussian tail estimate that $\P[\textstyle \inf_{D_N} 
\displaystyle \xi^N \geq -(h_* + \varepsilon - h')] \to 1$ as $N \to \infty$ for all $h'\in 
(h_{*} , h_* + \varepsilon)$ (alternatively one could also use Lemma~\ref{lem:BTIS_Szn} below 
to deduce this). All in all, in view of \eqref{eq:lower_sup_1}, one obtains
\begin{equation}
\label{eq:lower_sup_2.}
\text{$\P[C_N^{h'}]\to 1$ as $N \to \infty$ for every $h' \in (h_{*} , h_* + \varepsilon)$.}
\end{equation}
Using \eqref{eq:lower_sup_2.} and applying Lemma~\ref{lem:entrop_lower}, one infers that for every $h \leq h_*$ (and all $\delta, \varepsilon > 0$, implicit in the definition of $C_N^h$), 
	\begin{equation}\label{eq:lower_sup_2}
	\liminf_{N \to \infty} \frac{1}{\mathrm{cap}(K_N)}\log \P\big[C_N^h\big]\geq -\frac{1}{2}(h_* - h)^2.
	\end{equation}
Next, recall the event ${A}_N^h\coloneqq \tilde{A}_{N,\lfloor N^\delta \rfloor}^h$ from \eqref{eq:LBpf2}. In words, ${A}_N^h$ is the event that the cylinder $D_N$ contains a crossing in $\{\varphi \geq h\}$ intersecting both `slices' $F_{N}^{\pm}$ as defined below \eqref{eq:LBpf1}, with $L= \lfloor N^\delta \rfloor$. Combining the occurrence of ${A}_N^h$
and the insulating property of the disconnection event in \eqref{eq:lower_sup_1}, one deduces that 
$${A}_N^h\cap C_N^h \subset \bigcup_{x\in F_N^-} \{\lr{}{\varphi\geq h}{x}{\partial B_{N}(x)},\, \nlr{}{\varphi\geq h}{x}{\infty}\}$$ 
for all $h \in \R$, and hence by a union bound and translation invariance, that
	\begin{equation}\label{eq:lower_sup_5}
	\P[\lr{}{\varphi\geq h}{0}{\partial B_N},\, \nlr{}{\varphi\geq h}{0}{\infty}]\geq \frac{1}{|F_N^-|}\P[A_N^h\cap C_N^h].
	\end{equation}
Since the event $C_N^h$ in \eqref{eq:lower_sup_1} is  $\mathcal{F}_{U_N^c}= \sigma(\varphi_x; x \in U_N^c)$-measurable, applying the decomposition \eqref{eq:decomp}, one sees that for all $h \leq h_*$,
	\begin{equation}\label{eq:lower_sup_3}
	\P[{A}_N^h \cap C_N^h] = \E\big[ \P[{A}_N^h | \mathcal{F}_{U_N^c}]  1_{C_N^h} \big]  = \E\big[ \P_{U_N}[{A}_N^h (\cdot + \xi^{N})]  1_{C_N^h} \big] 
	\geq \P_{U_N}[{A}_N^{h_* + \varepsilon}] \P[C_N^h],
	\end{equation}
	where in the last step, we used monotonicity of the event ${A}_N^h(\cdot )$ along with the control on $\xi^{N}$ supplied by \eqref{eq:lower_sup_1}. A meaningful lower bound on $\P_{U_N}[{A}_N^{h_* + \varepsilon}] $ is obtained as follows. First recall that $\P[A^{h}_N]\to 1$ for all $h< h_*$ by \eqref{eq:lower_sup_500}. Then one writes, for all $h<h_*$, with $\varepsilon_0 = \frac12(h_*-h)$, 
$$
\P_{U_N}[{A}^{h}_N]=\P[{A}^{h}_N(\psi^N)] \geq \P\big[{A}^{h}_N(\psi^N) ,  \, \sup_{D_N}\xi^N \leq \varepsilon_0\big] \geq \P[{A}^{h+\varepsilon_0}_N(\varphi)] -  \P\big[ \sup_{D_N}\xi^N > \varepsilon_0\big].
$$
By \eqref{eq:lower_sup_500} and since $h +\varepsilon_0 = \frac12(h+h_*) < h_*$, the first term on the right-hand side tends to $1$ as $N \to\infty$. By similar considerations as above \eqref{eq:lower_sup_2.}, one sees that $\P\big[ \sup_{D_N}\xi^N > \varepsilon_0\big] \to 0$, yielding that $\P_{U_N}[{A}_N^h]\to 1$ as $N \to \infty$ for every $h<h_*$. Thus, Lemma~\ref{lem:entrop_lower} applies and gives, for all $\delta, \varepsilon > 0$,
	\begin{equation}\label{eq:lower_sup_4}
	\liminf_{N} \frac{1}{\mathrm{cap}_{U_N}(D_N)}\log \P_{U_N}\big[{A}^{h_* + \varepsilon}_N\big]\geq -\frac{1}{2}\varepsilon^2
	\end{equation}
(note also that $\mathrm{cap}_{U_N}(D_N) \geq \mathrm{cap}(D_N) \to \infty$, as required for 
Lemma~\ref{lem:entrop_lower} to apply). Finally, we put all these estimates together to arrive 
at \eqref{eq:main_lbsupercrit_d=3}: to see this, first plug
\eqref{eq:lower_sup_4} and \eqref{eq:lower_sup_2} into \eqref{eq:lower_sup_3} to get, for all $h\leq h_*$ and 
$\varepsilon,\delta > 0$, with $\rho_1 \coloneqq \liminf_{N }\frac{{\rm cap}(K_N)}{{\rm cap}(T_N)}$ and $\rho_2 \coloneqq \liminf_{N} \frac{{\rm cap}_{U_N}(D_N)}{{\rm cap}(T_N)}$,%for all $h \le h_*$ and $\varepsilon, \delta > 0$,
\begin{equation*}%\label{eq:lower_sup_4*}
\liminf_{N \to \infty} \frac{1}{\mathrm{cap}(T_N)}\log \P\big[A_N^h\cap C_N^h\big] \geq -\frac{1}{2}\Big((h_* - h)^2 \rho_1 + \varepsilon^2 \rho_2 \Big).
\end{equation*}
Now using \eqref{eq:cap_tube_d=3} and 
\eqref{eq:cap_tube_d=3bis} to bound $\rho_1$ and $\rho_2$ respectively, and subsequently applying \eqref{eq:lower_sup_5} while using the fact that $|F_N^-| \leq C N^{\delta(d-1)}$, one obtains for all $h\leq h_*$ and 
$\varepsilon,\delta > 0$,
	\begin{equation}\label{eq:lower_sup_6}
	\begin{split}
	&\liminf_{N \to \infty} \frac{1}{\mathrm{cap}(T_N)}\log \P\big[\lr{}{\varphi\geq h}{0}{\partial B_N},\, \nlr{}{\varphi\geq h}{0}{\infty}\big]\geq \\	&\liminf_{N \to \infty} \frac{1}{\mathrm{cap}(T_N)}\log \P\big[A_N^h\cap C_N^h\big]\geq -\frac{1}{2}\Big((h_* - h)^2(1+\Cr{thick-tube}\delta^{\Cr{thick-tube'}})+\Cr{thick-tube}\delta^{-1}\varepsilon^2\Big).
	\end{split}
	\end{equation}
	The result \eqref{eq:main_lbsupercrit_d=3} now follows from \eqref{eq:lower_sup_6} by 
	first letting $\varepsilon\to0$ and then $\delta\to 0$ (recall \eqref{eq:cap_line_asymp}).
	
	It remains to argue that \eqref{eq:main_lbsupercrit_d=3} continues to hold for the event $ 
	\text{2-arms}(N,h)$. Then the claim automatically follows for $\textnormal{LocUniq}(N,h)^{c}$ as the former is 
	contained in the latter (recall their definitions in~\eqref{eq:intro2arm} and 
	\eqref{eq:def_locuniq}). To this end, let %, $H_N^+ = \{x \in \Z^d : N \leq x^1 \leq 2N\}$, 
	$H^- = \{x \in B_{2N} : x^1 \leq -N\}$ and
	\begin{equation}
		\label{eq:lower_sup_501}
	A_{N,\text{left}}^h \coloneqq \big\{  \lr{}{\{\varphi\geq h\} \cap H^-}{B_N}{\partial B_{2N}}\big\}, \quad A_{N,\text{right}}^h\coloneqq (C_N^h \cap A_N^h) \circ \tau_{Ne_1}
	\end{equation}
with $A_N^h$ and $C_N^h$ given by \eqref{eq:LBpf2} and \eqref{eq:lower_sup_1} respectively. 
Note that the latter implicitly depends on $\delta$, cf.~\eqref{eq:LBpf1}, \eqref{eq:LBpf2} 
and above \eqref{eq:lower_sup_1}. Since $A_{N,\text{left}}^h \supset A_N^h \circ 
\tau_{-2Ne_1}$, combining \eqref{eq:lower_sup_500}, Lemma~\ref{lem:entrop_lower} and 
translation invariance of $\P$, one deduces in particular that
\begin{equation}
		\label{eq:lower_sup_501.1}
\liminf_{N} \frac1{\mathrm{cap}(T_N)}\log \P\big[A_{N,\text{left}}^h \big]\geq -\frac{1}{2}((h-h_*)_+)^2, \text{ for all } h \in \R.
	\end{equation}
The choices \eqref{eq:lower_sup_501} imply that
	\begin{equation}
	\label{eq:lower_sup_502}
(A_{N,\text{left}}^h \cap A_{N,\text{right}}^h) \subset \text{2-arms}(N,h).%\textnormal{LocUniq}(N,h)^{c}.
	\end{equation}
Indeed, the events $A_{N,\text{left}}^h$ and $A_N^h \circ \tau_{Ne_1}$ each yield a crossing cluster in $B_{2N}\setminus B_N$, and $C_N^h  \circ \tau_{Ne_1}$ implies that these are not connected in $\{\varphi \geq h \}$, whence %$B_{2N} \setminus B_N$ has at least two crossing clusters in $\{\varphi\geq h\}$ and 
\eqref{eq:lower_sup_502}. Now, with a similar decoupling argument as around \eqref{eq:lower_sup_3}, one obtains that for all $\varepsilon>0$,
\begin{equation}
	\label{eq:lower_sup_503}
	\P[A_{N,\text{left}}^h \cap A_{N,\text{right}}^h] \geq \P[A_{N,\text{left}}^{h+ \varepsilon}] \, \P[A_{N,\text{right}}^h] - 2 e^{-c\varepsilon^2 N^{d-2}},
\end{equation}		
where, in the notation of \eqref{eq:decomp}, the exponential error term corresponds to twice the 
probability that $|\xi_x^{H^+}|> \frac{ \varepsilon}{2}$ for some $x \in H^-$, with $H^+ 
\coloneqq \{x \in B_{2N} :  x^1 \geq 0\}$. Taking logarithms on the both sides of 
\eqref{eq:lower_sup_503}, %using concavity of $\log(\cdot)$, 
multiplying by $\frac{\log N}{N}$ and letting $N \to \infty$, the desired lower bound for $\P[\text{2-arms}(N,h)]$ 
with $h \leq h_*$ follow on account of \eqref{eq:lower_sup_502}, \eqref{eq:lower_sup_501.1} 
and the second line of \eqref{eq:lower_sup_6} upon letting $\delta,\varepsilon \to 0^+$.
\end{proof}

	\subsection{Lower bound for $d \geq 4$}\label{sec:lower_sup}
	
	We now supply the proof of \eqref{eq:main_lb_dgeq4}. A straightforward strategy consists of opening all the vertices in the line segment $T_N=T_N(0)$ (recall \eqref{eq:def_TN} for notation) and closing all vertices in its outer boundary $\partial_{\text{out}} T_N$. However, the Gaussian free field does not satisfy a uniform finite-energy property, which would make this easy to implement. Moreover, unlike in the subcritical case, one cannot immediately apply the FKG-inequality as the event in question is not monotonic. In order to deal with these issues, we use the midpoint extension $\tilde{\varphi}$ introduced at the end of Section~\ref{sec:preliminaries}, see in particular \eqref{eq:phi_extend1} and \eqref{eq:phi_extend2}.
	
	\begin{proof}[Proof of \eqref{eq:main_lb_dgeq4}]
Recall the decomposition \eqref{eq:LB1_dim4} of $\tilde{\varphi}$. Now let $\overline{T}_N\coloneqq T_N\cup\partial_{{\rm out}}T_N$ (where $T_N$ is viewed as a subset of $\Z^d$, whence $\overline{T}_N \subset \Z^d$) and let $M_N \coloneqq \{m\in\mathbb{M}^d:~ m\sim x, \text{ for some } x\in \overline{T}_N\}$. Using the fact that the absolute value of the Gaussian free field on any transient graph (and thus in particular of $\tilde{\varphi}$ on $\tilde{\Z}^d$) satisfies the FKG-inequality, see e.g.~(1.3) and Corollary 1.3 in \cite{eisenbaum2014}, one obtains that
	\begin{equation}\label{eq:GFFsquare_FKG}
	\P[|\hat{\xi}_{x}| \leq 1, ~ \forall x\in \overline{T}_N] \geq
	\P[|\tilde{\varphi}_m|\leq 1, ~ \forall m\in M_N]\geq \P[|\tilde{\varphi}_{m_0}|\leq 1]^{|M_N|}\geq e^{-CN},
	\end{equation}
	where the first inequality follows because $\hat{\xi}_{x}$ is the mean of $\tilde{\varphi}$ evaluated at its neighbors and $m_0$ refers to an arbitrary reference point in $\mathbb{M}^d$.
 As a consequence, one has, for all $h \in \R$,
	\begin{align*}
	&\P[\lr{}{\varphi\geq h}{0}{\partial B_N},\, \nlr{}{\varphi\geq h}{0}{\infty}]\geq \P[\varphi_x\geq h, ~\forall x\in T_N, ~\varphi_y<h,~\forall \partial_{\text{out}}T_N] \\[-0.3em]
	&\qquad \stackrel{\eqref{eq:LB1_dim4}}{\geq}
\P[|\hat{\xi}_{x}| \leq 1, ~ \forall x\in \overline{T}_N, ~\hat{\psi}_x\geq h +1, ~\forall x\in T_N, ~\hat{\psi}_y<h-1, ~\forall y \in \partial_{\text{out}}T_N] \\ &\qquad \stackrel{\eqref{eq:LB2_dim4}}{=} \P[|\hat{\xi}_{x}| \leq 1, ~ \forall x\in \overline{T}_N] \, P[X\geq h+1]^{|T_N|} \, P[X< h-1]^{|\partial_{\text{out}}T_N|} \stackrel{\eqref{eq:GFFsquare_FKG}}{\geq} e^{-C(h)N},
	\end{align*}
	where in the penultimate step, we also used independence of $\hat{\psi}$ and $\hat{\xi}$ and $X$ is a $N(0,1/2)$-distributed random variable. One easily adapts the above argument in order to create two `insulated' paths in $\{ \varphi \geq h \}$ joining $B_N$ to $\partial B_{2N}$ at an exponential cost in $N$, thus obtaining the lower bound \eqref{eq:main_lb_dgeq4} for the event $ \textnormal{LocUniq}(N,h)^{c}$ instead. This completes the proof of \eqref{eq:main_lb_dgeq4} and with it that of Theorem \ref{P:lb}.
	\end{proof}

\section{Coarse-graining} \label{sec:coarsegrain}

We now prepare the ground for the upper bounds that will be derived in the next section. The main result of this section is a certain coarse-graining scheme for paths, see Proposition~\ref{prop:coarse_paths} below. Its proof is split over Sections~\ref{sec:coarse3d} and \ref{sec:coarse4d}, which deal with the case $d=3$ and $d\geq 4$, respectively.
The key effect of the scheme is to keep the capacity of the `coarse-grained path' above a certain threshold, see \eqref{eq:cg99d=3}, \eqref{eq:cg_capd=3}%and \eqref{eq:cg_capd=4}
, while maintaining good control on the entropy factor for its possible choices. The notion of `coarse-grained path' is formalized in Definition~\ref{def:admissible}. It depends on a function $\Gamma(\cdot)$ parametrizing this entropy, see \eqref{def:admissible3} and \eqref{eq:cg_gamma}.

We now describe the precise setup. We consider positive integers $L\geq 1$ and $K\geq 100$ and introduce the lattice 
	\begin{equation}
	\label{eq:LL}
	\mathbb L = \mathbb L(L) \coloneqq L\Z^d
	\end{equation}
	along with the boxes
	\begin{equation}\label{eq:boxes}
	 C_z \coloneqq z+[0,L)^d ,\; \
	D_z \coloneqq z+[-3L,4L)^d,\; \ U_z \coloneqq z+[-KL+1, L+KL-1)^d
	\end{equation}
	as well as
	\begin{equation}
	\label{eq:C_boxes}
	\tilde{C}_z \coloneqq z+[-L,2L)^d,\; \ \tilde{D}_z  \coloneqq z+[-2L,3L)^d
	\end{equation}
attached to $z \in \mathbb L$. Notice that $ C_z \subset \tilde{C}_z \subset \tilde{D}_z \subset D_z\subset U_z$.  When considering more than one scale in a given context, we will sometimes explicitly refer to the associated length scale $L$ by writing $ C_{z,L}= C_z, \tilde{C}_{z,L}=\tilde{C}_z$ etc. %The set $ C_{z,L}$ shouldn't be confused with $B(z,L)$ as defined below \eqref{eq:def_h**}. 
Using the notation of \cite[Section 4]{Sz15}, for any $z \in  \mathbb L$, we introduce the decomposition
	\begin{equation}\label{eq:decomp_z}
	\varphi= \xi^z + \psi^z
	\end{equation}
where $\xi^z_x\coloneqq \xi^{U_z}_x =  E_x[\varphi_{X_{T_{U_z}}}]$ for all $x \in \mathbb{Z}^d$, cf.~\eqref{eq:decomp}, and $\psi^z= \psi^{U_z}$, with $U_z$ as in \eqref{eq:boxes}. Recall that $\psi^z$ has law $\mathbb{P}_{U_z}$. Letting $\tilde{U}_z\coloneqq (U_z\cap\Z^d) \cup \{ \tilde{x} \in \tilde{\Z}^d : \exists x \in (U_z \cap \Z^d) \text{ s.t. } |x-\tilde{x}|= \frac12\}$, we further define for $z \in  \mathbb L$ the extended harmonic average
\begin{equation}
\label{eq:xi_tilde_z}
\tilde{\xi}^z \coloneqq  \tilde{\xi}^{\tilde{U}_z}  \ \big(\stackrel{ \eqref{eq:Markov_GFFtilde}}{=} \tilde{E}_{\cdot}\big[\tilde{\varphi}_{\tilde{X}_{T_{\tilde{U}_z}}}\big] \big), \quad \tilde{\psi}^z \coloneqq \tilde{\varphi} -  \tilde{\xi}^{z}, 
\end{equation}	
where $\tilde{\varphi}$ refers to the extension of $\varphi$ to the graph $\tilde{\Z}^d$, see the discussion leading to \eqref{eq:phi_extend1}. On account of \eqref{eq:phi_extend2} and \eqref{eq:phi_extend4}, one thus has $\tilde{\xi}^{z}\vert_{\Z^d}= \xi^z$ and $\tilde{\psi}^{z}\vert_{\Z^d}= \psi^z$.
%The field $\xi^z$ will be referred to as the harmonic average of $\varphi$ in $U_z$ and $\psi^z$ as the local field in $U_z$. %{\color{red} It seems convenient, for the purpose of measurability in Defintion~\ref{def:goodevent}, to define the harmonic average in terms of the free field values at $\mathbb M^d$.} 
%Plainly, $\xi^z_x=\varphi_x$ (and therefore $\psi^z_x=0$) for all $x\in \Z^d\setminus U_z$. As in \cite[Lemma~1.2]{RoS13}, one observes that $\xi^z$ is independent $\psi^z$ and that $(\psi^z_x)_{x\in U_z}$ has law $\mathbb{P}_{U_z}$.
 Moreover, if
	\begin{equation}
	\label{eq:C:cond1}
\text{	 $\mathcal{C}\subset\mathbb L$ is a non-empty collection of sites with mutual $|\cdot|_{\infty}$-distance at least $2KL+L$,}
\end{equation}
 then denoting by $\tilde{\psi}^z=(\tilde{\psi}^z_{\tilde{x}})_{\tilde{x} \in \tilde{\Z}^d}$, for $z\in \mathcal{C}$, %and %$\tilde{\xi}= (\tilde{\xi}^z_{\tilde{x}})_{z\in \mathcal{C},\; \tilde{x}\in \tilde{U}_z},$ 
 one has that
	\begin{align}\label{eq:decomp_boxes}
	\begin{split}
	&\text{the Gaussian fields $\{ \tilde{\xi}, \, \tilde{\psi}^z, z\in \mathcal{C}\}$, where $ \tilde{\xi}= (\tilde{\xi}^z_{\tilde{x}})_{z\in \mathcal{C},\, \tilde{x}\in \tilde{U}_z}$, are independent.}
	\end{split}
	\end{align}
	Given $\mathcal{C}$ as in \eqref{eq:C:cond1}, we write 
	\begin{equation}
	\label{eq:cg_C}
	\Sigma \coloneqq \Sigma(\mathcal{C}) \coloneqq \bigcup_{z\in \mathcal{C}}  C_z.
	\end{equation}
The following precise result will be useful.
\begin{lemma}[Control of harmonic average]\label{lem:BTIS_Szn} There exist $\Cl{c:btis}(d) 
\in [1, \infty)$ as well as, for all $K \geq 100$, $\alpha(K)>1$ with $\lim_{K \to \infty} \alpha(K)= 1$ such that for every $a>0$, one has
\begin{equation}
\label{eq:BTIS_Szn}
		\limsup_{L} \sup_{\mathcal{C}} \bigg\{ \log\P\Big[\bigcap_{z\in\mathcal{C}} \{\sup_{ D_z} \tilde{\xi}^z\geq a\}\Big] +  \frac{1}{2}\bigg(a-\frac{\Cr{c:btis}}{K}\sqrt{\frac{|\mathcal{C}|}{\mathrm{cap}(\Sigma)}}\bigg)^2_+ \,\frac{\mathrm{cap}(\Sigma)}{\alpha(K)} \bigg\} \leq 0,
\end{equation}
		where the supremum over $\mathcal{C}$ runs over the sets satisfying \eqref{eq:C:cond1}, $\sup_{ D_z}$ refers to the supremum over all points in $D_z \cap \tilde{\mathbb{Z}}^d$ and $r_+ = r\vee 0$ for any $r \in \R$.
	\end{lemma}
\begin{proof} With $\xi^z$ in place of $\tilde{\xi}^z$, \eqref{eq:BTIS_Szn} is proved in \cite[Corollary~4.4]{Sz15}. To extend the result to $\tilde{\xi}^z$, observe that, for any mid-point $\tilde{x} \in (D_z \cap \tilde{\mathbb{Z}}^d)\setminus \Z^d$, the neighbors $x_1$ and $x_2$ of $\tilde{x}$ (in $\tilde{\Z}^d$) both belong to $D_z \cap \mathbb{Z}^d$. Furthermore, by harmonicity of $\tilde{\xi}^z_{\cdot}$ in $D_z \cap \tilde{\mathbb{Z}}^d$, cf.~\eqref{eq:boxes} and \eqref{eq:xi_tilde_z}, one has
$
 \tilde{\xi}^z_{\tilde{x}}= \frac12(\tilde{\xi}^z_{x_1}+ \tilde{\xi}^z_{x_2})=  \frac12({\xi}^z_{x_1}+ {\xi}^z_{x_2}),
$ whence $\tilde{\xi}^z_{\tilde{x}} \leq a$ whenever ${\xi}^z_{x_i} \leq a$ for $i=1,2$. Together, these observations yield that $$\textstyle \{\sup_{ D_z \cap \tilde{\Z}^d} \tilde{\xi}^z\geq a\} \subset \{\sup_{ D_z \cap \Z^d} {\xi}^z\geq a\},$$ and the claim follows.
\end{proof}
We will be interested in families of collections $\mathcal{C}$ satisfying \eqref{eq:C:cond1} with certain finer properties. In what follows, let $\Lambda_N $ be any of the elements in
\begin{equation}
\label{eq:scriptS_N}
\text{$\mathcal{S}_N \coloneqq \{  B_N, \, B_{2N} \setminus  B_N, \, \tilde{D}_{0,N}\setminus \tilde{C}_{0,N}\}$, for $N \geq 1$ (see below \eqref{eq:C_boxes} for notation).}
\end{equation}
In line with the wording below \eqref{eq:def_locuniq}, for $U\subset V \subset \subset \Z^d$, we say that a $*$-path $\gamma$ in $\Z^d$ (see Section \ref{sec:preliminaries} for its definition) \textit{crosses} $V\setminus U$ if $\gamma$ intersects both $U$ and $\partial V$. If $U=\{0\}$ we omit the reference to $U$; e.g.~when $\gamma$ crosses $B_N$ we mean that $\gamma$ intersects both $0$ and $\partial B_N$.
\begin{defn}\label{def:admissible} $(K \geq 100, \, L \geq 1, \, N \geq 10KL).$
Let $\Gamma:[1,\infty)\mapsto [0,\infty)$. A family $\mathcal{A} = \mathcal A_{N, L}^K(\Lambda_N)$ of collections $\mathcal{C} \subset \mathbb{L}$ is called $\Gamma$-admissible if, for some $\Cr{c:nLB}(d) \in (0,1) $,
\begin{align}
%\item[i)] $|\mathcal{C}| = 2^{n - m - k}$,
&\begin{array}{l}
\text{all $\mathcal C \in \mathcal A$ have equal cardinality $n \coloneqq |\mathcal C|$ satisfying $n \in [\frac{\Cl[c]{c:nLB}N}{u(KL)}, \frac{N}{ u(KL)}]$, where}\\
\text{$u(x)=x$ if $d=3$ and $u(x)=x(\log x)^2$ if $d \geq 4$, \eqref{eq:C:cond1} holds for all $\mathcal C \in \mathcal A$} \\
\text{and $\tilde{D}_z=\tilde{D}_{z,L}\subset \Lambda_N$ for all $z \in \mathcal{C}$,} 
\end{array}\label{def:admissible1}\\[0.5em]
&\begin{array}{l}
\text{for any $*$-path $\gamma$ %from $0$ to $\partial B_{N}$, 
crossing $\Lambda_N$, there exists $\mathcal{C} \in \mathcal{A}$ such that $\gamma$}\\ 
\text{crosses $\tilde{D}_{z}\setminus C_z$ (with $ C_z= C_{z,L}$, $\tilde{D}_z=\tilde{D}_{z,L}$ as in \eqref{eq:boxes}, \eqref{eq:C_boxes}) for all $z \in \mathcal{C}$,}
\end{array} \label{def:admissible2}\\[0.5em]
% i.e. it connects $B_{L}(z_i)$ to $\partial B_{2L}(z_i)$ inside the annulus $B_{2L}(z_i)\setminus B_{L -1}(z_i)$ 
&\begin{array}{l}\text{$\log |\mathcal{A}| \leq \Gamma(N/L).$} \end{array} \label{def:admissible3} %\\[0.5em]
%&\label{eq:cap_coarse_general}
%\begin{array}{l}
%\text{$\displaystyle \inf_{\mathcal{C}\in \mathcal{A}} \; \inf_{\substack{\tilde{\mathcal{C}}\subset \mathcal{C} \\ |\tilde{\mathcal{C}}|\geq (1-\rho)|\mathcal{C}|}} \; {\mathrm{cap}(\tilde{\Sigma})} \geq \lambda \, (1 - \rho) \, \mathrm{cap}(T_{N}),\ $  for all $\rho \in (0,1)$, with $\tilde{\Sigma} =\Sigma(\tilde{\mathcal{C}})$} %as defined in \eqref{eq:cg_C}.}
%\end{array}
\end{align}
%(see  \eqref{eq:cg_C} for the definition of $\Sigma(\cdot)$). {\color{red}FOR PIERRE: I changed to $\partial{\tilde \tilde{C}_z}$ since that is what we need to use for the supercritical upper bound. However since essentially all but one of $ C_z$ will be away from the boundary of $ B_N$ or $\tilde C_0(N)$, this should be immediate for the proof. We will say that we are going to prove it for $ C_z$ to $\partial \tilde{C}_z$, although it immediately gives crossing between $ C_z$ and $\partial {\tilde \tilde{C}_z}$.}
%\begin{equation}
%\label{eq:holes}
% \mathcal{H} \coloneqq [0,L_n] \setminus \big( \bigcup_{z \in \mathcal{C}} |z|_{\infty} + [0, L_{m + k}) \big)
% \end{equation}
%can be written as the disjoint union $\H=\bigcup_{m+k \leq \ell < n}\H_{\ell}$, where for each $\ell$, $\H_\ell$ has the following properties:
%the set $\H_{\ell}$ is the union of disjoint intervals of length $L_{\ell+1}-2L_{\ell}$ each; moreover denoting by $\partial \coloneqq \{0,L_n\}$,
%none of the intervals comprised in $\H_{\ell}$ are at distance $\leq L_{\ell}$ from $\partial$, at most $2$ of them are at distance $\delta$ from $\partial$
%satisfying $L_{\ell} \leq \delta < L_{\ell+1}$, and\ for every $1 \leq  i < n-\ell $, at most $2^i$ intervals are at distance $\delta$ from $\partial$
%satisfying $L_{\ell+i} \leq \delta < L_{\ell+i+1}$.
%The set of $(n,k, K)$-admissible collections will be denoted by $\mathcal{A}^K_{N,k}$.
\end{defn}
%Whenever clear from the context we will omit the reference to $N,L$ and $K$ from a 
%$(f,\lambda)$-admissible collection.

The main result of this section is the following:
\begin{proposition}[Coarsening of crossing paths]\label{prop:coarse_paths} There exists $\Cr{C:cg_complexity}(d) \in [1,\infty)$ such that, for all $K \geq 100$, 
	$L \geq 1$, $N \geq 10KL$ and $\Lambda_N \in \mathcal{S}_N$, one can construct a $\Gamma$-admissible collection $\mathcal{A}=\mathcal{A}_{N,L}^K(\Lambda_N)$ with the 
	following properties. 

\begin{itemize}
\item[i)] If $d=3$,
	%there exist $L_0=L_0( \rho) \geq 1 $ and $\Cl{C:NLratio}=\Cr{C:NLratio}(K,\rho)$ such that, for all integers $L$, $N$ satisfying $L \geq L_0$ and  $N/L \geq \Cr{C:NLratio}$, 
	\begin{align}
	%\label{eq:cg99}
	%(\lambda, f(r))=
	%\begin{cases}
%(\Cl[c]{C:cg}, \Cl{C:cg_complexity}r\log r)%\frac{ \Cl{C:cg_complexity}}{K}r\log r)
%, &\text{ if } d=3\\
%(\Cr{C:cg}, \Cr{C:cg_complexity}r),  &\text{ if } d\geq 4.
%	\end{cases}
\label{eq:cg_gamma}
& \Gamma(r)= {\Cl{C:cg_complexity}}{K}^{-1} r \log (r\vee 2),%\text{ where } \ \log^* r\coloneqq\begin{cases} K^{-1} \log r, &\text{if $d=3$}\\ 1,&\text{if $d\geq 4$} \end{cases}%\quad \\
%& \Gamma(r)= \Cr{C:cg_complexity}r, %\qquad \quad \ \, \inf_{K} \lambda(K)>0,  \ \, \text{ if }d\geq 4. \label{eq:cg99d=4}
	\end{align}%\marginpar{no divided by K in dim 4}
%with $\Cr{C:cg}= \Cr{C:cg}(K, \rho) \to 1$ as $K \to \infty$ and $\rho \to 1$ when $d=3$.  
and there exist positive numbers $ \lambda (K) \in (0,1]$ satisfying
\begin{equation}
\textstyle \lim_{K \to\infty} \lambda(K)=1, \label{eq:cg99d=3}
\end{equation}
such that, for all $\rho \in (0,1)$,
\begin{align}
\label{eq:cg_capd=3}
\liminf_{ N \to \infty} \; \inf_{L_0(N) \leq L \leq L_1(N)}  \; 
&\inf_{\mathcal{C}\in \mathcal{A}} \; \inf_{\substack{\tilde{\mathcal{C}}\subset \mathcal{C} \\ |\tilde{\mathcal{C}}|\geq (1-\rho)|\mathcal{C}|}} \; \frac{\mathrm{cap}(\tilde{\Sigma})}{\mathrm{cap}(T_{N})} \geq \lambda(K) \, (1 - \rho) % , \text{ if } d=3,\\[0.5em]
%\label{eq:cg_capd=4}
%\inf_{N\geq 1} \; \inf_{L \leq N/10K} \; &\inf_{\mathcal{C}\in \mathcal{A}} \; \inf_{\substack{\tilde{\mathcal{C}}\subset \mathcal{C} \\ |\tilde{\mathcal{C}}|\geq (1-\rho)|\mathcal{C}|}} \; \frac{\mathrm{cap}(\tilde{\Sigma})}{\mathrm{cap}(T_{N})} \geq \lambda(K) \, (1 - \rho) \, , \text{ if } d \geq 4,
\end{align}
with $\tilde{\Sigma} =\Sigma(\tilde{\mathcal{C}})$ $($see \eqref{eq:cg_C} for $\Sigma(\cdot))$ %, to be precise, the $\liminf$ in \eqref{eq:cg_capd=3} regards any sequence $(N_k, L_k)_{k \geq 0}$ subject to the constraints $N_k /L_k \geq 10K$ and $N_k, L_k \to \infty$ as $k \to \infty$. 
 and $(L_0(N), L_1(N))$ any sequence with $\lim_N L_0(N) = \infty$ and $L_1(N) \leq N/10K$ satisfying $\lim_N \tfrac{\log L_1(N)}{\log N} = 0$. 
\item[ii)] If $d \geq 4$, %\eqref{eq:cg_capd=3} remains valid upon replacing \eqref{eq:cg_gamma} by
 \begin{equation}
\label{eq:cg_gamma4}
\tag{\ref{eq:cg_gamma}'}
 \Gamma(r)= {\Cr{C:cg_complexity}} r
 \end{equation}
and \eqref{eq:cg_capd=3} remains valid with $L_0(N)=1$, $L_1(N)=N/10K$ and for some $\lambda(K) \in (0,1]$.
\end{itemize}
Moreover, the above conclusions continue to hold %in case $d=3$ 
with $\Lambda_N=B_N \setminus B_{\varepsilon N}$, for any $\varepsilon \in (0,\frac13)$, upon replacing $T_{N}$ by $T_{(1-\varepsilon)N}$ in \eqref{eq:cg_capd=3}.   %{\color{red}%Should it not be even better when $L_k$ is large?}
\end{proposition}

\begin{remark}
\label{remark:coarsening}
\begin{enumerate}[label={\arabic*)}]

\item The statement of Proposition~\ref{prop:coarse_paths} could even be generalized to the case $\Lambda_N=B_N \setminus B_{\varepsilon N}$, for any $\varepsilon \in (0,1)$, with suitable modifications. Specifically, the condition in \eqref{eq:cg_capd=3} relating $L$ and $N$ would involve $K$ and $\varepsilon$, and $\Cr{c:nLB}$ in \eqref{def:admissible1} would depend on~$\varepsilon$. We refrain from such extensions since we will only be interested in the limit $\varepsilon \downarrow 0$ in the sequel.

\item The differing complexities $\Gamma(r)$ in \eqref{eq:cg_gamma}, \eqref{eq:cg_gamma4}, reflect different procedures in implementing the coarse-graining depending on whether $d=3$ or $d\geq 4$. An interesting question this brings about is the following: can one devise a coarsening scheme in dimension $3$ (for instance the strategy employed for $d \geq 4$, or a variation thereof), which would be more cost-effective, i.e. reduce $\Gamma(r)$, ideally getting rid of the logarithmic factor in \eqref{eq:cg_gamma}, while retaining the controls \eqref{eq:cg99d=3}, \eqref{eq:cg_capd=3} on the capacity?

\item The coarse-graining scheme used in case $d=3$ can also be employed in dimensions $d\geq4$ in such a way that \eqref{eq:cg_capd=3} continues to hold. The resulting higher combinatorial complexity, see \eqref{eq:cg_gamma}, already yields near-optimal asymptotic upper bounds for the quantities of interest. We refer to Remarks~\ref{R:cg3to4} and~\ref{Rk:onionsd=4}~below for further details. 
\end{enumerate}

\end{remark}

\medskip

The proof of Proposition \ref{prop:coarse_paths} is split over Sections \ref{sec:coarse3d} and \ref{sec:coarse4d}, which deal separately with the cases $d=3$ and $d \geq 4$.

\subsection{Proof of Proposition~\ref{prop:coarse_paths} for $d=3$}\label{sec:coarse3d}
	Let $L\geq1$ and $N\geq 10KL$ be integers. We consider the case $\Lambda_N=B_N$. The small adaptations needed to account for the remaining cases in $\mathcal{S}_N $, see \eqref{eq:scriptS_N}, as well as  for $\Lambda_N = B_N \setminus B_{\varepsilon N}$, are indicated at the end of the proof. For each $i=1,\dots, n$ with $n\coloneqq \lfloor N/3KL \rfloor-1$ (note that $n \geq1$), define the concentric shells $S_i \coloneqq \partial B_{3KLi}$.

	By paving $S_i$ with boxes of the form $C_z= z+[0,L)^d$ for $z \in \mathbb{L}$, and considering the successive first exits of the path $\gamma$ from each of the sets {\em enclosed} by $S_i$, one finds for each $i =1,\dots, n$ a point $z_i \in \mathbb{L}$ such that 
		\begin{equation}
		\label{eq:cg100}
		C_{z_i} \cap S_i \neq \emptyset \text{ and $\gamma$ crosses $  \tilde{D}_{z_i}\setminus C_{z_i}$}. 
		\end{equation}
Note that $\tilde{D}_{z_i} \subset B_N$ for all $i=1,\dots,n$. We then define $\mathcal{A}=\mathcal{A}_{N,L}^K(B_N)$ as the family consisting of all collections $\mathcal{C}\coloneqq \{z_i : 1\leq i \leq n\} $ that can be obtained in this way. 
		
		We proceed to verify the conditions of Definition \ref{def:admissible} for $\Gamma(\cdot)$ as in \eqref{eq:cg_gamma} (with $d=3$). Properties \eqref{def:admissible1} and \eqref{def:admissible2} are immediate by construction. Regarding the cardinality of $\mathcal{A}$, one notes that the number of choices for $z_i$ is bounded by $(C\frac{N}{L})^{d-1}$, whence $|\mathcal{A}| %\leq \prod_{i=1,\dots,n} (CKi)^{d-1} 
\leq (C\frac{N}{L})^{n(d-1)},$ from which \eqref{def:admissible3} follows with $\Gamma(\cdot)$ given by \eqref{eq:cg_gamma}, for suitable choice of $\Cr{C:cg_complexity}$. 
		
		It remains to argue that \eqref{eq:cg_capd=3} holds, with $\lambda(K)$ satisfying \eqref{eq:cg99d=3}. This will be done in two steps, corresponding to Lemmas~\ref{Claim:cap1} and  \ref{Claim:cap2} below. We begin with a reduction step (Lemma~\ref{Claim:cap1}), consisting of replacing the set $\tilde{\Sigma}$ appearing in \eqref{eq:cg_capd=3} (recall \eqref{eq:cg_C}) by the `porous line'
 \begin{equation}
		\label{eq:capcompa-1}
		\widetilde{T}_N \coloneqq \bigcup_{i=1}^{\lceil (1-\rho)n \rceil} \big(\tilde{z}_i+ [0,L)\times \{0\}^2\big), \quad  \text{with} \quad \tilde{z}_i \coloneqq  S_i \cap ( [0,\infty)\times \{0\}^2) \, \big(=   ( 3LKi,0,0)\big).
		\end{equation}
Lemma~\ref{Claim:cap2} then compares the capacity of porous and non-porous lines. 
		\begin{lemma} 
			\label{Claim:cap1}
			For all $\rho \in (0,1)$, %$L \geq L_0$ and $N \geq 4KL$, all $\mathcal{C} \in \mathcal{A}$ and any sub-collection $\tilde{\mathcal{C}}$ satisfying $|\tilde{\mathcal{C}}|\geq (1-\rho)|\mathcal{C}|$, 
			\begin{equation} 
			\label{eq:capcompa1}
			%\liminf_{N,L\to\infty} \;
			\liminf_{ N \to \infty} \; \inf_{L_0(N) \leq L \leq N/10K}  \;  \inf_{\mathcal{C}\in \mathcal{A}} \; \inf_{\substack{\tilde{\mathcal{C}}\subset \mathcal{C} \\ |\tilde{\mathcal{C}}|\geq (1-\rho)|\mathcal{C}|}} \;\frac{\mathrm{cap}(\tilde{\Sigma})}{\mathrm{cap}(\widetilde{T}_N)} \geq   \lambda(K),
			\end{equation}
			for suitable $\lambda(K)>0$ satisfying \eqref{eq:cg99d=3}  and any $L_0(N)$ with $\lim_N L_0(N)=\infty$.%(cf.~also below \eqref{eq:cg_capd=3} for the meaning of the $\liminf$). 
		\end{lemma}
		
				We prove \eqref{eq:capcompa1} using a projection argument, first by rigidly displacing the boxes in $\tilde{\Sigma}$ onto $\R\times \{0\}^2$, then by `packing them together' along this axis to obtain the `homogenous porosity' of $\widetilde{T}_N$. Since these operations essentially reduce the ($\ell^2$-)distances between pairs of points in $\tilde{\Sigma}$, the capacity expectedly decreases. This intuition is formalized below.

	\begin{proof}[Proof of Lemma~\ref{Claim:cap1}]
	In view of \eqref{eq:Greenasympt}, for any $\varepsilon\in(0,1)$ we can choose $\widetilde{L}_0(\varepsilon) \geq 100$ such that
		\begin{equation}
		\label{eq:capcompa0}
		\Cr{c:green}^{-1} |x| g(0,x)  \in \big[1-\varepsilon, 1/(1-\varepsilon)\big], \text{ whenever }|x|_{\infty} \geq \widetilde{L}_0.
		\end{equation}
		Let $L\geq \widetilde{L}_0(\varepsilon)$. Notice that $\tilde{z}_i$ in \eqref{eq:capcompa-1} is such that $C_{\tilde{z}_i}$, cf. \eqref{eq:boxes}, is the (unique) box among those paving $S_i$ intersecting the positive half-line $\mathbb{Z}_+ \times \{0\}^2$. Now for $\mathcal{C} \in \mathcal{A}$ and $z_i \in \mathcal{C}$ as in \eqref{eq:cg100}, let $\tau(z_i)\coloneqq \tilde{z}_i$.  Importantly, $\tau$ is an ($\ell^{2}$-)projection, in that	
\begin{equation}
		\label{eq:capcompa-0.5}
		|\tau(z)-\tau(z')|%= \big||\tau(z)|_{\infty}- |\tau(z')|_{\infty} \big| 
		\leq |z-z'|, \text{ for all } z,z' \in \mathcal{C} \text{ and } \mathcal{C}\in \mathcal{A}.
		\end{equation}
Indeed, by construction, \eqref{eq:capcompa-0.5} holds with $|\cdot|_{\infty}$ in place of $|\cdot|$ and \eqref{eq:capcompa-0.5} follows because $|\tau(z)-\tau(z')|= |\tau(z)-\tau(z')|_{\infty}$ and $|\cdot|_{\infty} \leq |\cdot|$. The map $\tau$ extends naturally to a bijection defined on the set $\Sigma$ (cf. \eqref{eq:cg_C}) by setting $\tau(y)= \tau(z)+ y-z$, if $y \in C_z$ for some $z \in \mathcal{C}$. In words, $\tau$ sends any point in $C_z$ to the corresponding point in $C_{\tau(z)}$.

		Recalling the notation from \eqref{eq:cap_var}, for any probability measure $\mu$ supported on $\tilde{\Sigma}$, with $\mu' =\tau \circ \mu$ and $\tilde{\Sigma}_{\tau}\coloneqq \{ \tau(z): z \in \tilde{\Sigma}\}$, it follows that
		\begin{equation}
		\label{connecLDP5}
		E(\mu)=\sum_{x,y \in  \tilde{\Sigma}_{\tau}} \mu'(x) g\left(\tau^{-1}(x), \tau^{-1}(y)\right) \mu'(y) \leq \kappa \cdot E(\tau \circ \mu)
		\end{equation}
		with 
		\begin{equation}
		\label{connecLDP6}
		\kappa \coloneqq \sup_{x,y \in \tilde{\Sigma}_{\tau}}\frac{g\left(\tau^{-1}(x), \tau^{-1}(y)\right)}{g(x,y)}.
		\end{equation}
The quantity $\kappa$ is suitably bounded as follows. By translation invariance of the Green's function, if $x,y$ belong to the same box $C_{\tau(z)}$ for some $z\in \tilde{\mathcal{C}}$, then $g(\tau^{-1}(x),\tau^{-1}(y))=g(x-y)$. Otherwise, i.e. if $x \in C_{\tau(z)}$ and $y \in C_{\tau(z')}$ for $z,z' \in \mathcal{C}$ with $z \neq z'$, using the triangle inequality and \eqref{eq:capcompa-0.5}, one readily infers that $|x-y| \leq |\tau^{-1}(x)- \tau^{-1}(y)|+ 4\sqrt{3}L,$ for all $x,y \in \tilde{\Sigma}_{\tau}$. With this observation, returning to \eqref{connecLDP6} and using \eqref{eq:capcompa0}, which is in force as $|x-y| \wedge |\tau^{-1}(x) -\tau^{-1}(y)|\geq KL \geq \widetilde{L}_0$ whenever $x$ and $y$ belong to different boxes, one finds that
		\begin{equation*}
		%\label{connecLDP7}
		\kappa \leq 1 \vee \bigg( \frac{1}{(1-\varepsilon)^2} \sup_{\substack{x \in  C_{\tau(z)}, y \in  C_{\tau(z')} \\z,z' \in \mathcal{C},\,  z \neq z'}} \frac{|x-y|}{|\tau^{-1}(x)-\tau^{-1}(y)|} \bigg) \leq \frac{1+C/K}{(1-\varepsilon)^2},
		\end{equation*}
		using in the last step that $|\tau^{-1}(x)-\tau^{-1}(y)| \geq KL$. Substituting this bound into \eqref{connecLDP5}, taking an infimum over $\mu$, noting that $\mu \mapsto \tau \circ \mu$ is a bijection between probability measures with support on $\tilde{\Sigma}$ and $\tilde{\Sigma}_{\tau}$, and applying the variational formula \eqref{eq:cap_var}, one obtains the lower bound 
		\begin{equation} 
		\label{eq:capcompa20}
		\mathrm{cap}(\tilde{\Sigma}) \geq  (1-C/K)(1- \varepsilon)^2 \mathrm{cap}(\tilde{\Sigma}_{\tau}).
		\end{equation}
		
		In view of \eqref{eq:capcompa1}, in order to produce the set $\widetilde{T}_N$, cf. \eqref{eq:capcompa-1}, one trims $\tilde{\Sigma}_{\tau}$ as follows. First, let $\widetilde{T}_N':=\bigcup_{i=1}^{\lceil (1-\rho)n \rceil} C_{\tilde{z}_i}$ and note that $\widetilde{T}_N\subset \widetilde{T}_N'$, whence $\text{cap}(\widetilde{T}_N) \leq \text{cap}(\widetilde{T}_N' )$. Observe that the elements of $\mathcal{C}$, and hence of $\tilde{\mathcal{C}}$, are  naturally ordered according to increasing index $i$, cf. below \eqref{eq:cg100}. Now one only retains the boxes in $\tilde{\Sigma}_{\tau}$ corresponding to the first $\lceil (1-\rho)n \rceil$ elements of $\tilde{\mathcal{C}}$ in this ordering (recall that $\tilde{\mathcal{C}}$ has at least this many elements), and only keeps the intersection of the resulting set of boxes with the line $\mathbb{Z}\times \{0\}^2 $, thus obtaining overall a smaller set $\tilde{\Sigma}_{\tau}' \subset \tilde{\Sigma}_{\tau}$, whence $\text{cap}(\tilde{\Sigma}_{\tau}' ) \leq \text{cap}(\tilde{\Sigma}_{\tau} ) $. The resulting set $\tilde{\Sigma}_{\tau}'$ is in natural bijection with $\widetilde{T}_N'$, essentially by removing the gaps, one by one by rigidly shifting all the intervals to the (say) right of the gap by a suitable constant amount to the left. This operation either leaves the relative position between two points $x,y \in \tilde{\Sigma}_{\tau}' $ unchanged or \textit{reduces} their Euclidean norm, but in the latter case never as to fall below $KL \geq \widetilde{L}_0$ (a lower bound on the gap size in $\widetilde{T}_N$). With this observation, a similar argument as above, using \eqref{eq:cap_var} and \eqref{eq:capcompa0} yields that 
		\begin{equation}
		\label{eq:capcompa21}
		\text{cap}(\tilde{\Sigma}_{\tau} ) \geq   \text{cap}(\tilde{\Sigma}_{\tau}' ) \geq (1-\varepsilon)^2 \text{cap}(\widetilde{T}_N') \geq (1-\varepsilon)^2 \text{cap}(\widetilde{T}_N).
		\end{equation}
		By letting $\varepsilon\to0$ (and therefore $L\geq \widetilde{L}_0\to\infty$ as well as $N \geq 10KL \to \infty$), \eqref{eq:capcompa20} and \eqref{eq:capcompa21} imply \eqref{eq:capcompa1} with $\lambda(K)=1-C/K$.
	\end{proof}	
				
		We proceed with
		\begin{lemma}
			\label{Claim:cap2}
			For all $K \geq 100$, $\rho \in (0,1)$,
			\begin{equation}
			\label{eq:capcompa22}
			\liminf_{\frac{\log N}{ \log L} \to \infty} \  \frac{ \mathrm{cap}(\widetilde{T}_N)}{ \mathrm{cap}\big({T}_{N}\big)} \geq 1-\rho,
			\end{equation}
where the $\liminf$ regards any sequence $(N_k, L_k)_{k \geq 0}$ satisfying $\frac{\log N_k}{\log L_k} \to \infty$ as $k \to \infty$.
		\end{lemma}
			\begin{proof} 
		Let $\widehat{T}_N \coloneqq T_{\lceil (1-\rho)n \rceil 3KL +L}$ (recall that $n= \lfloor N/3LK \rfloor-1$ and \eqref{eq:def_TN} for notation). On account of \eqref{eq:cap_line_asymp}, it suffices to show that $\liminf_{N/L \to \infty} \  \frac{ \mathrm{cap}(\widetilde{T}_N)}{ \mathrm{cap}(\widehat{T}_{N})} \geq 1$ and \eqref{eq:capcompa22} directly follows. By definition, see \eqref{eq:capcompa-1}, $\widetilde{T}_N \subset \widehat{T}_N$ and 
		\begin{equation}
		\label{eq:capcompa23}
		\text{for all $x \in  \widehat{T}_N$, $d(x, \widetilde{T}_N) \leq 3KL$}.
		\end{equation}
		Hence, by \eqref{eq:sweeping_2}, \eqref{eq:capcompa22} follows at once if we argue that
		\begin{equation}
		\label{eq:capcompa3}
		\liminf_{N/L \to \infty} \  \inf_{x \in  \widehat{T}_N} P_x[H_{\widetilde{T}_N}< \infty]=1.
		\end{equation}
In view of \eqref{eq:capcompa23}, Lemma~\ref{lem:visibility} applied to $T= \widetilde{T}_N \subset \widehat{T}_N $ and $x \in  \widehat{T}_N$ with the choices $r=6KL$ and $\gamma = \frac{1}{13K}$, by which \eqref{eq:vis1} is satisfied, yields that $$P_x[H_{\widetilde{T}_N}= \infty] \leq %C(K) 
C \Big(\frac{\log L}{\log N}\Big)^{c(K)}\to 0 \text{ as $\log N/ \log L \to \infty$. }$$ Thus, \eqref{eq:capcompa3} follows, which completes the proof.
				\end{proof}
		
With Lemmas~\ref{Claim:cap1} and~\ref{Claim:cap2} at hand, we can conclude the proof of \eqref{eq:cg_capd=3} (and \eqref{eq:cg99d=3}) for $\Lambda_N =B_N$. %First suppose that $L \geq N/ (\log N)^{\alpha}$, for some $\alpha <1$ (say $\alpha=\frac12$ for instance). Then for all $\mathcal{C}\in \mathcal{A}$ and $\emptyset \neq \tilde{\mathcal{C}}\subset \mathcal{C}$, using \eqref{eq:capball} and \eqref{eq:cap_line_asymp}, one sees that\begin{equation}\label{eq:capcompa324} \frac{\mathrm{cap}(\tilde{\Sigma})}{\mathrm{cap}(T_{N})} \geq  \frac{\mathrm{cap}(B_L)}{\mathrm{cap}(T_{N})}\geq (\log N)^{1-\alpha} \to \infty, \text{ as }N \to \infty.\end{equation}In particular, by \eqref{eq:capcompa324} one may assume in addition that $L=o(N)$ as $N \to \infty$ in~\eqref{eq:cg_capd=3}. %(i.e.~it suffices to consider sequences $(N_k,L_k)$, $k \geq 0$, satisfying the additional constraint $N_k/L_k \to \infty$). Under these assumptions,\eqref{eq:cg99d=3}, \eqref{eq:cg_capd=3}, follow directly from \eqref{eq:capcompa1} and \eqref{eq:capcompa22}. 
The remaining cases in $\mathcal{S}_N $, see \eqref{eq:scriptS_N}, i.e.~$\Lambda_N=B_{2N} \setminus B_N$, resp.~$\Lambda_N=\tilde{D}_{0,N}\setminus \tilde{C}_{0,N}$ are dealt with by considering instead $S_i =  \partial B_{N+ 3KLi}$, resp.~$S_i = \partial ([-3KLi-N, 2N+ 3KLi)^d)$ (cf.~\eqref{eq:C_boxes}), for $1 \leq i \leq n$, and adapting the subsequent arguments accordingly. In particular, the bound for $|\mathcal{A}|$ remains valid for these choices. The set $\widetilde{T}_N$ in \eqref{eq:capcompa-1} changes accordingly whence $\tilde{z}_i = (N+ 3LKi,0,0)$, resp.~$\tilde{z}_i = (2N+ 3LKi-1,0,0)$. The statements and proofs of Lemmas~\ref{Claim:cap1} and \ref{Claim:cap2} then remain valid.

In the case $\Lambda_N = B_N \setminus B_{\varepsilon N}$ for a given $\varepsilon \in 
(0,\frac13)$, one considers the shells $S_i \coloneqq \partial B_{\lceil\varepsilon N\rceil + 
3KLi}$, for $1\leq i \leq n\coloneqq \lfloor (1-\varepsilon)N/3KL \rfloor-1$ (note that $n 
\geq 1$ when $N \geq 10KL$). Then, defining $\widetilde{T}_N$ as in \eqref{eq:capcompa-1} (whence $\tilde{z}_i = (\lceil \varepsilon N\rceil + 3LKi,0,0)$), Lemma~\ref{Claim:cap1} 
remains valid and Lemma~\ref{Claim:cap2} as well upon replacing $T_N$ by $T_{(1-\varepsilon)N}$ 
in \eqref{eq:capcompa22}. The lower bound \eqref{eq:cg_capd=3} with $T_{(1-\varepsilon)N}$ 
instead of $T_N$ then follows as above. This concludes the proof of 
Proposition~\ref{prop:coarse_paths} in the case $d= 3$. \hfill $\square$

\begin{remark}
\label{R:cg3to4} The above coarse-graining scheme also applies when $d \geq 4$. As a result, one obtains a
$\Gamma$-admissible collection ${\mathcal{A}}'={\mathcal{A}}_{N,L}^K(\Lambda_N)'$, with $\Gamma(\cdot)$ as defined in \eqref{eq:cg_gamma} for $d=3$, i.e.~$\Gamma(r)=\frac{\Cr{C:cg_complexity} }{K}r \log r$, and with $u(x)=x$ in \eqref{def:admissible1}, so that the statement~\eqref{eq:cg_capd=3} for $d \geq 4$ holds with ${\mathcal{A}}'$ in place of ${\mathcal{A}}$. Only \eqref{eq:cg_capd=3} requires an explanation. Repeating the steps leading up to \eqref{eq:capcompa-1}, one shows an analogue of Lemma~\ref{Claim:cap1} %with $\widetilde{T}_N'=\bigcup_{i=1}^{\lceil (1-\rho)n \rceil} C_{\tilde{z}_i}$ (cf. also below~\eqref{eq:capcompa20}) in place of $\widetilde{T}_N$. 
without the condition \eqref{eq:cg99d=3}. The proof essentially remains the same except that one simply uses $c \leq |x|^{d-2} g(0,x)\le C$ for all $x \in \Z^d$ instead of \eqref{eq:capcompa0}, which is sufficient since one only aims at a $\lambda(K) > 0$ in \eqref{eq:capcompa1}. Then Lemma~\ref{Claim:cap2} gets replaced by the statement that 
$ \frac{ \mathrm{cap}(\widetilde{T}_N)}{ \mathrm{cap}({T}_{N})} \geq cK^{-1} (1-\rho)$ for all $L \geq 1$ and $N \geq 10KL$, which follows by covering ${T}_{N}$ with at most $CK$ shifted copies of the set $\widetilde{T}_N$ and using monotonicity and subadditivity of $\mathrm{cap}(\cdot)$ (see e.g.~Proposition 2.2.1(b) in \cite{La91} regarding the latter).
%$\frac{ \mathrm{cap}(\widetilde{T}_N')}{ \mathrm{cap}(\widehat{T}_{N}')}\geq cK^{-(d-2)}$ for all $L \geq 1$ and $N \geq 10KL$, where $\widehat{T}_N' \coloneqq T_{\lceil (1-\rho)n \rceil 3KL +L}(L)$. To see this one simply uses \eqref{eq:sweeping_2}, observes that for all $x \in  \widehat{T}_N'$, $d(x, \widetilde{T}_N') \leq 3KL$ by construction (cf.~\eqref{eq:capcompa23}), whence
%$$
% \inf_{x \in  \widehat{T}_N'} P_x[H_{\widetilde{T}_N'}< \infty] \geq \inf_{x: d(x, [0,L)^d) \leq CKL } P_x\big[H_{[0,L)^d}< \infty\big] \stackrel{\eqref{eq:lastexit}, \eqref{eq:capball}}{\geq} cK^{-(d-2)}.
%$$
%The lower bound~\eqref{eq:cg_capd=3} for $d \geq 4$ then readily follows from the above adaptations of Lemmas~\ref{Claim:cap1} and~\ref{Claim:cap2} since $\frac{\mathrm{cap}({T}_{N})}{\mathrm{cap}({T}_{(1-\rho)N})} \geq 1$.
The slightly higher combinatorial complexity of the collection $\mathcal{A}'$, reflected by the logarithmic factor in $\Gamma(\cdot)$, cf.~\eqref{eq:cg_gamma}, only yields near-optimal upper bounds for $d \geq 4$, see Remark~\ref{Rk:onionsd=4} below. The presence of the additional logarithm is remedied by the approach of Section~\ref{sec:coarse4d}.
\end{remark}
\medskip
		\subsection{Proof of Proposition~\ref{prop:coarse_paths} for $d\geq 4$}	\label{sec:coarse4d}

			We introduce the length scales 
	\begin{equation}
	\label{eq:coarsegrain_r_n}
	L_0 \coloneqq 1, \quad L_{m+1} \coloneqq  \lceil 2(1+\varepsilon_m)L_m \rceil , \text{ with } \varepsilon_m \coloneqq (m+1)^{-2}, \, m \geq 0.
	\end{equation}
Note that $2^m \leq L_m \leq C2^m$ for all $m \geq 0$, which will be used frequently below.
Throughout the proof we use $C_{z,k}, \tilde{C}_{z,k}, D_{z, k}$ and $U_{z, k}$ to denote the boxes $ C_z, \tilde{C}_z, 
	D_z$ and $U_z$, see~\eqref{eq:boxes}, corresponding to the length $L = L_k$ for any $z \in \Z^d$. Also let $\hat{C}_{z,k}\coloneqq z + [-L_{k} + L_{k-1}, 2L_k - 
	L_{k-1})^d$ and observe that $C_{z,k}\subset \hat{C}_{z,k}\subset \tilde{C}_{z,k}$. 
	
We first define a certain coarse graining of paths crossing generic shapes, see Lemma~\ref{L:CGd=4} below, which will later be applied inductively to define an admissible 
collection with the desired properties. Roughly speaking, the idea is to implement a cascading 
scheme on the path $\gamma$ of diameter~$N$, thus only retaining its trace in a system of 
well-separated boxes at scale $L$ (naturally indexed by the leaves of a binary tree of depth 
$\approx \log_2(N/L)$), see for instance \cite{DrRaSap14}, Section 8.1 for a gentle 
introduction to this circle of ideas in a related model. 

The precise recursive scheme 
underlying the proof of Lemma~\ref{L:CGd=4} %refines 
builds upon the ideas from \cite{RoS13, PT12}. %, thereby improving the amount of information kept when iterating the construction from a scale $L_{n_0} \approx N$ to $L_{k_0} \approx L$ for suitable $0 \leq k_0 \leq n_0$.
In particular, we modify the sub-division scheme of \cite{PT12} in order to ensure that the 
smaller annuli obtained after each iteration lie inside $\Lambda_N$ (see the proof of 
Lemma~\ref{L:CGd=4} below). This is to guarantee that the annuli obtained at the very end are 
also contained in $\Lambda_N$ as required by Definition~\ref{def:admissible}. However, doing 
so causes the annuli to overlap  and we need to use more restricted regions than full annuli 
in order to avoid that. A general notion of regions with which we can accomplish this is that 
of {\em shapes} which we introduce now.\\ %In particular, retaining merely the information that \textit{full} annuli are crossed at smaller scales (as done in~\cite{PT12}) makes it difficult to derive the lower bound \eqref{eq:cg_capd=3}. 
%This leads to the notion of {\em shapes} which we introduce now.
A {\em shape} at level $k$ anchored at $z \in \Z^d$ is any $*$-connected subset $S$ of $\tilde{C}_{z,k} \setminus C_{z,k}^-$, where $C_{z,k}^-= C_{z,k}\setminus \partial C_{z,k}$, %such that every $*$-connected component of $S$ connects 
intersecting both %$\partial_{{\rm out}}C_{z,k}\marginpar{check!}
$\partial C_{z,k}$
 and $\partial \tilde{C}_{z,k}$. %Here $\partial_{{\rm out}}^{*}A= \{ x \in A^c : \exists y \in A \text{ s.t. } |x-y|_{\infty}=1\}$ refers to the $*$-neighborhood of a set $A \subset \Z^d$. 
 The collection of all shapes at level $k$ anchored at $z$ will be denoted by $\mathcal S_{z, k}$ and $\mathcal{S}_{k} \coloneqq \bigcup_{z \in \Z^d} \mathcal S_{z, k}$ % \bigcup_{z \in L_k\Z^d} \mathcal S_{z, k}
 is the set of all shapes at level $k$. %Any face $F$ of $C_{z,k}$ (i.e.~any set of the form $F=( \partial C_{z,k}) \cap \{ x^i = z^i+ a\}$ for some $1\leq i \leq d$, $a \in \{ 0,L_k-1\}$)
 % whose 1-neighborhood in $*$-sense %is connected to (
%intersects %?)\marginpar{check} 
%$S$ intersecting $S$ will be called an {\em exposed face} (with respect to $S$). %The reference to the shape $S$ will often be implicit in the forthcoming discussion but always be clear from the context.
For all practical purposes except the one mentioned in the previous paragraph, the reader can 
safely picture a shape in $\mathcal{S}_{k}$ as an annulus $\tilde{C}_{z,k} \setminus 
C_{z,k}^-$.

In what follows, let $\mathbb{T}_n$, $n \geq 0$, denote the (rooted) complete binary tree of 
depth $n$ and let $\mathcal{L}(\mathbb{T}_n)= \{0,1\}^n$ (with $\mathcal{L}(\mathbb{T}_0)=\{ 
\varnothing \}$, the root) be its set of leaves. The leaves of $\mathbb{T}_n$ provide a 
natural indexing set due to the recursive dyadic manner in which the coarsening scheme 
operates, cf.~\eqref{eq:proper2}.

\begin{lemma}[Coarse-graining of shapes] \label{L:CGd=4}

For any integers $n \geq  k > 0$, $z \in \Z^d$ and all shapes $S \in \mathcal S_{z,n}$, there exists a family $\mathcal A_S = 
\mathcal A_{S, k}$ with
\begin{equation}
\label{eq:complexity}
\log |\mathcal A_{S}| \le C L_{n - k}
\end{equation}
 of collections $\mathcal D  = \{(z(\ell), S(\ell)): \ell \in \mathcal{L}(\mathbb 
T_{n-k})\} \subset \Z^d \times \mathcal{S}_{k}$ satisfying the following three properties:
\begin{align}
&\label{eq:proper1}\begin{array}{l}\text{$S(\ell) \subset S$ and $S(\ell) \in \mathcal S_{z(\ell), k}$ for all $\ell \in\mathcal{L} (\mathbb T_{n-k})$,}\end{array}\\[0.5em]
&\label{eq:proper2}\begin{array}{l}
\text{if $n> k$, the sub-collections $\mathcal D_i \coloneqq \{(z(\ell), S(\ell)): \ell \in  \mathcal{L}(\mathbb T_{n-k, i})\}$, $i=1,2$,}\\
	\text{of $\mathcal D$, where $\mathbb T_{n-k, i}$, $i=1,2$, denote the two binary sub-trees of $\mathbb T_{n-k}$ with 
	}\\
	\text{depth $n-k-1$, belong to $\mathcal A_{S_1}$ and $\mathcal 
	A_{S_2}$ for some shapes $S_1,S_2 \in \mathcal{S}_{n-1}$}\\
	\text{such that $d_{\infty}(S_1, S_2) \ge 2\varepsilon_{n-1}L_{n-1}$,}
\end{array}\\[0.5em]
&\label{eq:proper3}\begin{array}{l} 
\text{for any $*$-path $\gamma \subset S$ crossing $\tilde{C}_{z, n} \setminus C_{z,n}$, there exists $\mathcal D \in \mathcal A_S$ such that}\\
\text{for all $\ell \in \mathcal{L}(\mathbb T_{n-k})$, $\gamma$ 
	induces a $*$-path $\gamma' \subset S(\ell)$ crossing $  \tilde{C}_{z(\ell), k} \setminus C_{z(\ell), k}$.}
	\end{array}
\end{align}
\end{lemma}

\begin{proof} Fix $k > 0$. We proceed by induction over $n \geq k$. When $n = k$, we simply choose $\mathcal A_S$ to be the singleton set consisting of $\mathcal D \coloneqq \{(z, S)\}$ whence \eqref{eq:complexity}, \eqref{eq:proper1} and \eqref{eq:proper3} hold.

Suppose now that for some $n \geq 0$ and each $S \in \mathcal{S}_{n}$, there exists a family $\mathcal{A}_{S}$ satisfying  \eqref{eq:proper1}-\eqref{eq:proper3}, and such that, for some $b_{n-k} \in (0,\infty)$,
\begin{equation}
\label{eq:complexity:C_n}
\sup_{S \in  \mathcal{S}_{n}} \log |\mathcal A_{S}| \le b_{n-k} L_{n - k}
\end{equation}
(note that \eqref{eq:complexity:C_n} holds for $n=k$ with $b_0=0$). Consider a shape $S \in \mathcal{S}_{z,n+1}$ for some $z \in \Z^d$. Choose a fixed set of vertices $T_1 \subset C_{z, n+1}$ of cardinality $|T_1| \le \Cl{C:boundT1}$ for suitable $\Cr{C:boundT1} \in [1,\infty)$ such that $C_{w, n} \subset C_{z, n+1}$ for all $w \in T_1$ and the faces of the boxes $\{C_{w, n}; w \in T_1\}$ form a cover of the  %exposed 
faces of $C_{z,n+1}$. Now, since $C_{w, n} \subset C_{z, n+1}$ for all $w \in T_1$, given any $*$-path $\gamma$ 
crossing $\tilde{C}_{z, n+1} \setminus C_{z, n+1}$, one finds $z_1 \in T_1$ such that $\gamma$ induces a path $\gamma_1$ crossing $ \tilde{C}_{z_1, n} \setminus C_{z_1, n}$. Furthermore if 
$\gamma \subset S$, it follows that $\gamma_1 \subset S_{z_1}$ where $S_{z_1} \subset S$ 
is %the union of all $*$-connected components of $S \cap \tilde{C}_{z_1, n}$ 
the $*$-connected component of $S \cap (\tilde{C}_{z_1, n} \setminus C_{z_1, n}^-)$ containing  $\gamma_1$%intersecting both $\partial C_{z_1, n}$ and $\partial\tilde{C}_{z_1, n}$
. It is clear from this definition that $S_{z_1} \in \mathcal S_{z_1, n}$.

By a similar reasoning, one finds a set $T_2  \subset \hat{C}_{z, n+1}$, see below \eqref{eq:coarsegrain_r_n} for notation, with $|T_2| \le  \Cl{C:boundT2}$ for suitable $\Cr{C:boundT2} \in [1, \infty)$ such that for some $z_2 \in T_2$, $\gamma$ exits $\hat C_{z, n+1}$ for the last time through $C_{z_2, n} \subset \hat C_{z, n+1}$. Since $z_2 \in \hat C_{z, n+1}$ and consequently $\tilde C_{z_2, n} \subset \tilde{C}_{z, n+1}$, we deduce from these definitions that $\gamma$ induces a path $\gamma_2$ crossing $ \tilde C_{z_2, n} \setminus C_{z_2, n}$. Now, let $S_{z_2} \in \mathcal S_{z_2, n}$ be defined as the $*$-connected component of $S \cap (\tilde{C}_{z_2, n} \setminus  C_{z_2, n}^-)$ containing $\gamma_2$. %, where $C_{z, n+1}^- = C_{z, n+1} \setminus \partial C_{z, n+1}$. 
Note that $\gamma_2  \subset S_{z_2} \subset S \setminus \hat C_{z, n+1}$. Now recalling that $S_{z_1} \subset \tilde C_{z_1, n} \subset z + [-L_n, L_{n+1} + L_{n})^d$, it follows that 
\begin{equation}
\label{eq:cgd=4.103}
d_{\infty}(S_{z_1}, S_{z_2}) \ge d_{\infty}([-L_n, L_{n+1} + L_{n})^d, \partial_{\text{out}}\hat C_{0, n+1}) \geq L_{n+1}-2L_n\stackrel{\eqref{eq:coarsegrain_r_n}}{\geq} 2\varepsilon_nL_n.
\end{equation}
Therefore, upon defining $\mathcal A_S$ to be the collection of all $\mathcal D$ such that the restriction of $\mathcal D$ to the leaves of the left and right sub-trees of $\mathbb T_{n+1-k}$ 
of depth $n-k$ correspond to some $\mathcal D_1 \in \mathcal A_{S_{z_1}}$ and $\mathcal 
D_2 \in \mathcal A_{S_{z_2}}$, for some $z_1 \in T_1$ and $z_2\in T_2$ respectively, then the 
properties \eqref{eq:proper1}-\eqref{eq:proper3} follow as an immediate consequence of the above construction and the induction hypothesis. In particular, the distance constraint in \eqref{eq:proper2} (at level $n+1$) is exactly \eqref{eq:cgd=4.103}.

Finally, observe that 
\begin{equation}
\label{eq:recursion}
|\mathcal 
A_{S}| \le |T_1|\cdot|T_2| \,\sup_{S' \in \mathcal{S}_{n}}|\mathcal A_{S'}|^2 \leq   \Cr{C:boundT1}  \Cr{C:boundT2}\sup_{S' \in \mathcal{S}_{ n}}|\mathcal A_{S'}|^2. %\Cl{C:treerecursion} \sup_{S' \in \mathcal{S}_{n}}|\mathcal A_{S'}|^2.
\end{equation}
Together, \eqref{eq:complexity:C_n}, \eqref{eq:recursion} and \eqref{eq:coarsegrain_r_n} readily imply %$\sup_{S \in  \mathcal{S}_{n+1}} \log |\mathcal A_{S}| \le C_{n+1-k} L_{n+1 - k}$ 
that a bound similar to \eqref{eq:complexity:C_n} holds with $n+1$ in place of $n$ and $b_{n+1-k}\coloneqq b_{n-k} + \frac{\log ( \Cr{C:boundT1}  \Cr{C:boundT2})}{L^{n-k}}$%$b_{n+1-k}\coloneqq b_{n-k} + \frac{\log \Cr{C:treerecursion}}{L^{n-k}}$
, from which \eqref{eq:complexity} follows with the choice $C= \lim_{n \to \infty} b_n (< \infty)$.
\end{proof}

The next result entails a capacity estimate which will be key in deducing \eqref{eq:cg_capd=3}.
For a given shape $S \in \mathcal{S}_{n}$ and a collection $\mathcal D  = \{(z(\ell), S(\ell)): \ell \in \mathcal{L}(\mathbb 
T_{n-k})\} \in \mathcal A_{S, k}$ (with $\mathcal A_{S, k}$ as given by Lemma~\ref{L:CGd=4}), consider a collection $\mathcal{B}=\{ {B}(\ell): \ell \in \mathcal{L}(\mathbb 
T_{n-k})\}$ of boxes such that, for all $\ell \in  \mathcal{L}(\mathbb 
T_{n-k}) $,
\begin{equation}
\label{eq:cgd=4.1}
{B}(\ell)= z + [0,r)^d, \text{ for some } z \in \Z^d \text{ s.t. } B(\ell) \cap S(\ell) \neq \emptyset \text{ and }1 \leq r \leq \frac{1}{2}\varepsilon_k L_k.
\end{equation}
Now, define the set
\begin{equation}
\label{eq:cgd=4.2}
S_\mathcal{B}(\tilde{\mathcal D}) \coloneqq \bigcup_{\substack{ \ell \in\mathcal{L}(\mathbb T_{n-k}): \\  (z(\ell), S(\ell)) \in \tilde{\mathcal{D}}}} B(\ell), \quad  \text{ for $\tilde{\mathcal{D}} \subset \mathcal{D}$}%C_{z(\ell), k}%S(\ell)
\end{equation}
%where $C_{z(\ell), k}$ is the box to which the shape $S(\ell) \in \mathcal S_{z(\ell), k}$ is attached,
and 
\begin{equation}
\label{eq:kappa_cap}
\kappa_{n,k} \coloneqq \inf_{\rho \in (0,1] }\inf_{S \in \mathcal{S}_{n}}\inf_{\mathcal D \in \mathcal A_{S, k}} \inf_{\mathcal{B}} \inf_{\substack{ \tilde{\mathcal D} \subset \mathcal D: \\|\tilde{\mathcal D}| \geq \rho |\mathcal D |}} \rho^{-1}{\rm cap}\,(S_{\mathcal{B}}(\tilde{\mathcal D})),
\end{equation}
%In plain words, $C_{n , k}$ is the minimum capacity of sets of the form $S_{\mathcal{B}}(\mathcal D)$ 
%where $\mathcal D \in \mathcal A_{S, k}$ for some shape $S$ at scale $n$. 
with the infimum over $\mathcal{B}$ ranging over all collections of boxes satisfying \eqref{eq:cgd=4.1}.
The following lemma supplies suitable lower bounds on the quantity $\kappa_{n,k}$.

\begin{lemma} $(d \geq 4)$
	\label{lem:coarse_capacity}
	For any $n \ge k > 0$, one has %and $a\in [0,1]$, one has
	\begin{equation}
	\label{eq:coarse_capacity}
	\kappa_{n+1,k} \ge 2^{n+1}% (\kappa_{k,k})^a 
	\wedge \frac{2\kappa_{n,k}}{1 + C\frac{2^{n+1}% (\kappa_{k,k})^a 
	}{(\varepsilon_n L_n)^{d-2}}}\,.
	\end{equation}
As a consequence, for all for all $n \ge k > 0$, one has%$d \ge 4$ and $k \ge k_0$ depending solely on $d$, we have
	\begin{equation}
\label{eq:coarse_capacity2}
\kappa_{n,k} \ge 2^{n} \wedge \Cl[c]{c:kappaLB} 2^{n-k}\kappa_{k,k}. %{\prod_{j = k+1}^{n}\Big(1 + C\frac{2^{j}}{(2\varepsilon_{j-1} L_{j-1})^{d-2}}\Big)}.
\end{equation}
\end{lemma}

\begin{proof}
Consider $S \in \mathcal{S}_{n+1}$, $\mathcal D \in \mathcal A_{S} = \mathcal A_{S, k}$, a collection $\mathcal{B}$ satisfying \eqref{eq:cgd=4.1} and  $\tilde{\mathcal D} \subset \mathcal D$ with $|\tilde{\mathcal D}| \geq \rho \mathcal D$ for some $\rho \in (0,1]$. Consider the sub-collections $\mathcal D_1$ 
and $\mathcal D_2$ of $\mathcal D$ given by \eqref{eq:proper2}, define $\tilde{\mathcal D}_i 
\coloneqq \tilde{\mathcal D} \cap \mathcal D_i$ as well as the sets  $S_{\mathcal{B}}(\tilde{\mathcal D}_i)$ for $i = 1, 2$ in similar fashion as \eqref{eq:cgd=4.2}, so that $S_{\mathcal{B}}(\tilde{\mathcal{D}})=S_{\mathcal{B}}(\tilde{\mathcal D}_1) \cup S_{\mathcal{B}}(\tilde{\mathcal D}_2)$. By \eqref{eq:proper2} and \eqref{eq:cgd=4.1} and since $n \geq k$, one has
	\begin{equation}
	\label{eq:coarse_capacity_separation}
		d_{2}(S_{\mathcal{B}}(\tilde{\mathcal D}_1), S_{\mathcal{B}}(\tilde{\mathcal D}_2)) \ge  \varepsilon_{n}L_{n}.
		\end{equation}
Now, since $S_{\mathcal{B}}(\tilde{\mathcal D}) \supset S_{\mathcal{B}}(\tilde{\mathcal D}_i)$ for $i=1,2$, using the sweeping identity \eqref{eq:sweeping}, one bounds
\begin{align}
\label{eq:coarse_capacity_sweep}
{\rm cap}(S_{\mathcal{B}}(\tilde{\mathcal D}_i)) &\le \sum_{z \in S_{\mathcal{B}}(\tilde{\mathcal D}_i)}e_{S_{\mathcal{B}}(\tilde{\mathcal D})}(z) + \max_{z \in S_{\mathcal{B}}(\tilde{\mathcal D}_{3-i})}P_z[H_{S_{\mathcal{B}}(\tilde{\mathcal D}_i)} < \infty]\sum_{z \in S_{\mathcal{B}}(\tilde{\mathcal D}_{3-i})}e_{S_{\mathcal{B}}(\tilde{\mathcal D})}(z).%, \mbox{ and }\nonumber \\
%{\rm cap}(S_{\mathcal{B}}(\tilde{\mathcal D}_2)) &\le \sum_{z \in S_{\mathcal{B}}(\tilde{\mathcal D}_2)}e_{S_{\mathcal{B}}(\tilde{\mathcal D})}(z) +  \max_{z \in S_{\mathcal{B}}(\tilde{\mathcal D}_1)}P_z[H_{S_{\mathcal{B}}(\tilde{\mathcal D}_2)} < \infty]\sum_{z \in S_{\mathcal{B}}(\tilde{\mathcal D}_1)}e_{S_{\mathcal{B}}(\tilde{\mathcal D})}(z).
\end{align}
Using \eqref{eq:lastexit}, one finds, for $i=1,2$,
\begin{align*}
\max_{z \in S_{\mathcal{B}}(\tilde{\mathcal D}_{3-i})}P_z[H_{S_{\mathcal{B}}(\tilde{\mathcal D}_i)} < \infty] &\le {\rm cap}(S_{\mathcal{B}}(\tilde{\mathcal D}_i))\max_{z \in S_{\mathcal{B}}(\tilde{\mathcal D}_{3-i}), y \in S_{\mathcal{B}}(\tilde{\mathcal D}_i)}g(z, y),% \mbox{ and }\nonumber\\
%\max_{z \in S_{\mathcal{B}}(\tilde{\mathcal D}_1)}P_z[H_{S_{\mathcal{B}}(\tilde{\mathcal D}_2)} < \infty] &\le {\rm cap}(S_{\mathcal{B}}(\tilde{\mathcal D}_2))\max_{z \in S_{\mathcal{B}}(\tilde{\mathcal D}_2), y \in S_{\mathcal{B}}(\tilde{\mathcal D}_1)}g(z, y).
	\end{align*}
The maximum of Green's functions on the right-hand side is bounded by $C(\varepsilon_nL_n)^{2-d}$ in view of 
\eqref{eq:coarse_capacity_separation} and \eqref{eq:Greenasympt}. Substituting the resulting bound into \eqref{eq:coarse_capacity_sweep} yields the estimates 
$$
{\rm cap}(S_{\mathcal{B}}(\tilde{\mathcal D}_i)) \leq  \sum_{z \in S_{\mathcal{B}}(\tilde{\mathcal D}_i)}e_{S_{\mathcal{B}}(\tilde{\mathcal D})}(z) +  C\,\frac{{\rm cap}(S_{\mathcal{B}}(\tilde{\mathcal D}_1)) \vee {\rm cap}(S_{\mathcal{B}}(\tilde{\mathcal D}_2))}{(\varepsilon_n L_n)^{d-2}}\sum_{z \in S_{\mathcal{B}}(\tilde{\mathcal D}_{3-i})}e_{S_{\mathcal{B}}(\tilde{\mathcal D})}(z),
$$
valid for $i=1,2$. Adding these and solving for ${\rm cap}(S_{\mathcal{B}}(\tilde{\mathcal D}))$ gives, in view of \eqref{eq:kappa_cap},
\begin{equation}
\label{eq:kappa_cap2}
{\rm cap}(S_{\mathcal{B}}(\tilde{\mathcal D})) \ge \frac{{\rm cap}(S_{\mathcal{B}}(\tilde{\mathcal D}_1)) + {\rm cap}(S_{\mathcal{B}}(\tilde{\mathcal D}_2))}{1 + C\,\frac{{\rm cap}(S_{\mathcal{B}}(\tilde{\mathcal D}_1)) \vee {\rm cap}(S_{\mathcal{B}}(\tilde{\mathcal D}_2))}{(\varepsilon_n L_n)^{d-2}}} \ge \frac{|\tilde {\mathcal D}|}{|\mathcal D|}\,\frac{2\kappa_{n, k}}{1 + C\,\frac{{\rm cap}(S_{\mathcal{B}}(\tilde{\mathcal D}_1)) \vee {\rm cap}(S_{\mathcal{B}}(\tilde{\mathcal D}_2))}{(\varepsilon_n L_n)^{d-2}}}\,.
\end{equation}
However, since ${\rm cap}(S_{\mathcal{B}}(\tilde{\mathcal D})) \ge {\rm cap}(S_{\mathcal{B}}(\tilde{\mathcal D}_1)) \vee {\rm cap}(S_{\mathcal{B}}(\tilde{\mathcal D}_2)) =: \Xi$ due to the monotonicity of $\text{cap}(\cdot)$, see \eqref{eq:sweeping_2}, 
distinguishing whether $ \Xi \geq 2^{n+1}%(\kappa_{k, k})^a
$, in which case ${\rm cap}(S_{\mathcal{B}}(\tilde{\mathcal D}))$ inherits this lower bound, or $ \Xi < 2^{n+1}%(\kappa_{k, k})^a
$, in which case one applies \eqref{eq:kappa_cap2}, it follows that
\begin{equation*}
	{\rm cap}(S_{\mathcal{B}}(\tilde{\mathcal D})) \ge \frac{|\tilde {\mathcal D}|}{|\mathcal D|}\Big(2^{n+1}%(\kappa_{k, k})^a 
	\wedge \frac{2\kappa_{n, k}}{1 + C\,\frac{2^{n+1}%(\kappa_{k, k})^a
	}{(\varepsilon_n L_n)^{d-2}}}\Big),
\end{equation*}
yielding \eqref{eq:coarse_capacity}. The lower bound \eqref{eq:coarse_capacity2} follows from \eqref{eq:coarse_capacity} and a straightforward induction argument, with $ \Cr{c:kappaLB} = \prod_{n = 0}^{\infty}(1 + C\frac{2^{n+1}}{(2\varepsilon_{n} L_{n})^{d-2}})^{-1}  > 0$ (see \eqref{eq:coarsegrain_r_n}).
\end{proof}

We now complete the proof of Proposition~\ref{prop:coarse_paths} for $d\geq 4$. Let $L \geq 1$ $K \geq 100$ and $N \geq 10KL$. We first introduce the collection $\mathcal{A}=\mathcal{A}_{N,L}^K(\Lambda_N)$ for any $\Lambda_N \in \mathcal{S}_N$ (recall its definition from \eqref{eq:scriptS_N}) and verify 
that it is $\Gamma$-admissible. Let $n_0$ be maximal such that $L_{n_0} \leq N/10$, which by \eqref{eq:coarsegrain_r_n} implies that
 \begin{equation}
 \label{eq:cgd=4.3}
N/30 \leq L_{n_0} \leq N/10, \text{ and let $k_0$ be such that } L_{k_0-1} \leq 5L \leq L_{k_0}.
 \end{equation}
 Note that $k_0 \leq n_0$. The set $\mathcal{A}$ will be defined in terms of the coarse grainings $\mathcal{A}_{S,k_0}$ of a fixed number of shapes $S \in \mathcal{S}_{n_0}$, which we now introduce. Let $V^0 \subset \Z^d$ consist of $0$ and any point $z \in L_{n_0}\Z^d$ such that $C_{z, n_0}= z + [0,L_{n_0})^d$ intersects $\partial B_{3N/2}$ or $\partial [-\frac{N}{2}, \frac{3N}{2})^d$. Define $\mathcal{S}^0 =\{ \tilde{C}_{z,n_0} \setminus C_{z,n_0} : z \in V^0\}  \subset \mathcal{S}_{n_0}$. Note that $|\mathcal{S}^0| \leq C$ by choice of $n_0$ in \eqref{eq:cgd=4.3}. Moreover, in view of \eqref{eq:scriptS_N}, any $*$-path $\gamma$ crossing $\Lambda_N$ induces a $*$-path crossing $S$ for some $S \in \mathcal{S}^0$.

Let us briefly pause to describe in plain words how $\mathcal{A}$ will be constructed. With $L_{n_0} \approx N$ and to a given shape $S\in \mathcal{S}_{n_0}$ among the finite set $\mathcal{S}^0$, we can associate by means of Lemma~\ref{L:CGd=4} (applied with $n=n_0$ and $k=k_0$)
a collection of shapes at scale $L_{k_0} \approx L$ indexed by the leaves of a binary tree of depth $n_0-k_0$. We can attach to this setup a system of $L$-boxes which are simply the translates of $[0,L)^d$ anchored at a point $y \in \mathbb{L}$ intersecting the resulting shapes. The important requirement \eqref{def:admissible2} that these boxes witness the macroscopic crossing will essentially follow from \eqref{eq:proper3}. One issue with this is that the resulting system of boxes is only separated by a distance roughly $\varepsilon_{k_0} L_{k_0} \ll L$, which does not satisfy \eqref{eq:C:cond1}. To fix this, we `trim' the collection by going down the tree until reaching an intermediate scale $L_k$ of order $KL$, and then only retain a fixed leave (say the smallest one in lexicographic order) in each subtree of depth $k-k_0$ attached to a vertex at depth $n_0 -k$. We now proceed to formalize this and subsequently verify $\Gamma$-admissibility (see Definition~\ref{def:admissible} and \eqref{eq:cg_gamma4}) of the resulting collection of boxes along with the required estimate \eqref{eq:cg_capd=3} on their capacity.

\smallskip

Let $k \in [k_0,n_0]$ be such that
\begin{equation}
 \label{eq:cgd=4.4}
2\varepsilon_{k-2} L_{k-2} \leq 2KL + 3L \leq 2\varepsilon_{k-1} L_{k-1}
\end{equation}
and %define the subset of leaves  $\mathcal{L}'(\mathbb T_{n_0-k_0})\subset \mathcal{L}(\mathbb T_{n_0-k_0})$ 
consider the subset of leaves $\ell' \in \mathcal{L}(\mathbb 
T_{n_0-k_0})$ of the form $\ell' = \ell \times(0,\dots,0)=: \ell_0$, where $\ell \in \mathcal{L}(\mathbb 
T_{n_0-k})$ is arbitrary and $(0,\dots,0) \in  \mathcal{L}(\mathbb 
T_{k-k_0})$ is fixed. For a given shape $S \in \mathcal{S}^0$ and a collection $\mathcal{D}_0 \in \mathcal{A}_{S,k_0}$, `prune' $\mathcal{D}_0 = \{(z(\ell'), S(\ell')): \ell' \in \mathcal{L}(\mathbb 
T_{n_0-k_0})\}$ (and forget the anchor point $z(\ell')$) to obtain the collection $\mathcal{D} = \{S(\ell_0): \ell \in \mathcal{L}(\mathbb 
T_{n_0-k})\}$. Let $\mathcal{A}'$ denote the collections $\mathcal{D}$ thereby obtained as $\mathcal{D}_0$ ranges over $\mathcal{A}_{S,k_0}$ and $S \in \mathcal{S}^0$. The collection $\mathcal{A}=\mathcal{A}_{N,L}^K(\Lambda_N)$ is then defined as (recall $\mathbb{L}$ from \eqref{eq:LL})
\begin{equation}
 \label{eq:cgd=4.5}
 \begin{split}
\mathcal{A}= \big\{ \mathcal{C}= ( y(\ell) )_{ \ell \in  \mathcal{L}(\mathbb 
T_{n_0-k})} \ :  \  & y(\ell) \in \mathbb{L}, \, \exists \, \mathcal{D} = \{S(\ell_0): \ell \in \mathcal{L}(\mathbb 
T_{n_0-k})\} \in \mathcal{A}' \\
&\text{s.t. }  C_{y(\ell),L} \cap S(\ell_0) \neq \emptyset \text{ for all } \ell \in  \mathcal{L}(\mathbb 
T_{n_0-k})  \big\}.
\end{split}
\end{equation}%\newpage
We now verify that $\mathcal{A}$ defined in \eqref{eq:cgd=4.5} is $\Gamma$-admissible. The fact that $|\mathcal{A}|$ satisfies \eqref{def:admissible3} with $\Gamma(\cdot)$ as in \eqref{eq:cg_gamma4} follows from \eqref{eq:complexity}, the fact that $ |\mathcal{S}^0|\leq C$, whence $\log | \mathcal{A}'|\leq CL_{n_0-k_0}$, and since the choice of points in $\mathcal{C}$ for a given $\mathcal{D} \in  \mathcal{A}' $ is bounded by $C^{|\mathbb{T}_{n_0-k}|}$. Overall, this gives $\log | \mathcal{A}|\leq CL_{n_0-k_0}$, as desired (note that $L_{n_0-k_0} \leq C N/L$ by \eqref{eq:cgd=4.3}).

The crossing property \eqref{def:admissible2} can be seen as follows. Let $\gamma$ be a $*$-path crossing $\Lambda_N$. As noted above, $\gamma$ induces a crossing for one of the annuli shapes $S \in \mathcal{S}^0$. By \eqref{eq:proper3} there exists $\mathcal{D} \in \mathcal{A}_{S,k_0}$ such that the following holds for every $\ell \in  \mathcal{L}(\mathbb 
T_{n_0-k})$. The $*$-path $\gamma$ induces a $*$-path $\gamma'$ crossing $\tilde{C}_{z(\ell), k_0} \setminus C_{z(\ell), k_0}$ with $\gamma' \subset S(\ell_0)$. In particular, by paving the part of $\partial S(\ell_0)$ adjacent to $C_{z(\ell), k_0}$ by boxes $C_{y,L}$, for $y \in \mathbb{L}$ and by choice of $L_{k_0}$ in \eqref{eq:cgd=4.3}, one finds a point $y(\ell)$ such that $\tilde{D}_{y(\ell),L}\setminus C_{y(\ell),L} $ is crossed by $\gamma'$. The resulting collection $\mathcal{C}$ belongs to $\mathcal{A}$ and \eqref{def:admissible2} follows.

Regarding \eqref{def:admissible1}, observe that $n=|\mathcal{C}|= |\mathcal{L}(\mathbb{T}_{n_0-k})|$, whose logarithm is comparable to $ L_{n_0}/L_k$, hence to $N/ (KL (\log KL)^2)$ using \eqref{eq:cgd=4.3}, \eqref{eq:cgd=4.4} and the fact that $c  (\log KL)^{-2} \leq \varepsilon_k \leq C  (\log KL)^{-2} $. The required separation property \eqref{eq:C:cond1} then follows from  \eqref{eq:proper1} and \eqref{eq:proper2}. Indeed, the latter (applied inductively) implies that any two shapes $S(\ell_0)$, $S(\ell_0')$ with $\ell \neq \ell' \in  \mathcal{L}(\mathbb 
T_{n_0-k})$ are each subsets of two shapes $S(\ell),S(\ell')\in \mathcal{S}_{k}$ separated by $2\varepsilon_{k-1}L_{k-1} \geq  2KL + 3L$, see \eqref{eq:cgd=4.4}. Hence $S(\ell_0)$, $S(\ell_0')$ inherit this separation. On account of \eqref{eq:cgd=4.5}, the resulting points $y(\ell)$, $y(\ell')$ are then at $\ell^{\infty}$-distance at least $2KL + L$.

Thus $\mathcal{A}$ is $\Gamma$-admissible. To see that the capacity lower bound \eqref{eq:cg_capd=3} holds, first observe that $L \geq \frac{\Cr{c:rLB}}{K} \varepsilon_k L_{k}$ by \eqref{eq:cgd=4.4} upon choosing $\Cr{c:rLB}$ small enough and choose $r(\geq 1)$ satisfying
\begin{equation}
\label{eq:condition_r}
\frac{\Cl[c]{c:rLB}}{2K} \leq \frac{ r}{\varepsilon_k L_k} \leq \Big(\frac{1}{2} \wedge \frac{\Cr{c:rLB}}{K} \Big).
\end{equation}
In particular $r \leq L$. Now consider an arbitrary collection $\mathcal{C} \in \mathcal{A}$ and note that
\begin{equation}
\label{eq:capLBfinal}
\text{each box $C_{y(\ell),L}$, $y(\ell) \in {\mathcal{C}}$, contains a box $B(\ell)$ satisfying \eqref{eq:cgd=4.1} (with $n=n_0$);}
\end{equation} 
indeed this follows immediately by construction of $C_{y(\ell),L}$, which intersects $S(\ell)$ by definition, see \eqref{eq:cgd=4.5}, and the fact that $r \leq L$.

Together, \eqref{eq:capLBfinal}, \eqref{eq:cg_C} and \eqref{eq:cgd=4.2} imply that for any $\mathcal{C} \in \mathcal{A}$ and any sub-collection $\tilde{\mathcal{C}}\subset \mathcal{C}$ with $|\tilde{\mathcal{C}}| \geq (1-\rho)|{\mathcal{C}}|$ for some $\rho \in (0,1)$, ${\rm cap}(\Sigma(\tilde{\mathcal{C}})) \geq {\rm cap}(S_{\mathcal{B}}(\tilde{\mathcal D}))$, for some family $\mathcal{B}$ satisfying \eqref{eq:cgd=4.1}, where $\mathcal{D}$ refers to the collection generating $\mathcal{C}$, see \eqref{eq:cgd=4.5}, and $\tilde{\mathcal D} \subset \mathcal{D}$ is the sub-collection of $\mathcal{D}$ corresponding to the indices $\ell \in  \mathcal{L}(\mathbb 
T_{n_0-k}) $ appearing in $\tilde{\mathcal{C}}\subset \mathcal{C}$. It follows that
\begin{equation}
\label{eq:cgfinald=4.1}
(1-\rho)^{-1} {\rm cap}(\Sigma(\tilde{\mathcal{C}})) \stackrel{\eqref{eq:kappa_cap}}{\geq} \kappa_{n_0,k} \stackrel{\eqref{eq:coarse_capacity2}}{\geq} c K^{-(d-2)} N,
\end{equation}
where the last inequality is obtained by combining the fact that $2^{n_0} \geq cN$ due to \eqref{eq:cgd=4.3} (see also the note following \eqref{eq:coarsegrain_r_n}) and observing that 
\begin{equation}
\label{eq:cgfinald=4.2}
\begin{split}
2^{-k}\kappa_{k,k} 
&\stackrel{\eqref{eq:kappa_cap}}{\geq} c L_{k}^{-1} \text{cap}(B_{r})  \stackrel{\eqref{eq:condition_r}}{\geq } c K^{-(d-2)} \inf_{m \geq 0} L_m^{-1}\text{cap}(B_{\varepsilon_m L_m}) \\
&\stackrel{\eqref{eq:capball}}{\geq } c K^{-(d-2)} \inf_{m\geq 0} L_m^{d-3}\varepsilon_m^{d-2} \stackrel{\eqref{eq:coarsegrain_r_n}}{\geq} c'K^{-(d-2)}.
\end{split}
\end{equation}
The bound \eqref{eq:cg_capd=3} follows immediately from \eqref{eq:cgfinald=4.1}. If $\Lambda_N=B_N \setminus B_{\varepsilon N}$ for some $\varepsilon \in (0,\frac13)$, one simply sets $\mathcal{A}_{N,L}^{K}(\Lambda_N)\coloneqq \mathcal{A}_{N,L}^{K}(B_N \setminus B_{N/2})$, which has the desired properties. This completes the proof of Proposition~\ref{prop:coarse_paths} for $d\geq 4$.  \hfill $\square$

%\medskip

%\begin{remark} The above fails in case $d=3$ notably because in \eqref{eq:cgfinald=4.2}, one obtains $2^{-k}\kappa_{k,k} \geq C \varepsilon_k^{-1}/K$ and because $ \Cr{c:kappaLB} =0$ (see the end of the proof of Lemma~\ref{lem:coarse_capacity}).
%\end{remark}

\section{Upper bounds}\label{sec:upper}
Using the coarse-graining scheme developed in the last section, see in particular Proposition~\ref{prop:coarse_paths}, we  now derive companion upper bounds to the lower bounds obtained in Theorem~\ref{P:lb}.  
The main result of this section is:
\begin{theorem}[Upper bounds] \label{P:ub} $\quad$
	\begin{enumerate}
		\item[i)] If $d=3$, then 
		\begin{align}
		&\text{for all $h > h_{*}$, }\limsup_{N\to\infty}\, \frac{\log N}{N} \log \P[\lr{}{\varphi\geq h}{0}{\partial B_N}%,\, \nlr{}{\varphi\geq h}{0}{\infty}
		] \leq -\frac{\pi}{6}(h - h_{*})^2,  \label{eq:main_ubsubcrit_d=3} \\
		&\text{for all $h< h_*$, }\limsup_{N\to\infty}\, \frac{\log N}{N} \log \P[\lr{}{\varphi\geq h}{0}{\partial B_N},\, \nlr{}{\varphi\geq h}{0}{\infty}
		] \leq -\frac{\pi}{6}(h_* - h)^2. \label{eq:main_ubsupercrit_d=3}
		\end{align}
		\item[ii)] If $d \geq 4$, then
	\begin{align}
		&\text{for all $h > h_{*}$, }\limsup_{N\to\infty}\, \frac{1}{N} \log \P[\lr{}{\varphi\geq h}{0}{\partial B_N}%,\, \nlr{}{\varphi\geq h}{0}{\infty}
		] < 0,  \label{eq:main_ubsubcrit_d>4} \\
		&\text{for all $h< h_*$, }\limsup_{N\to\infty}\, \frac{1}{N} \log \P[\lr{}{\varphi\geq h}{0}{\partial B_N},\, \nlr{}{\varphi\geq h}{0}{\infty}
		] < 0. \label{eq:main_ubsupercrit_d>4}
	\end{align}
	\end{enumerate}
		Moreover, the bounds \eqref{eq:main_ubsupercrit_d=3} and \eqref{eq:main_ubsupercrit_d>4} also hold for the events $ \textnormal{LocUniq}(N,h)^{c}$ and $\text{2-arms}(N,h)$ (see~\eqref{eq:def_locuniq} and \eqref{eq:intro2arm}) in place of $ \displaystyle \{ \lr{}{\varphi\geq h}{0}{\partial B_N},\, \nlr{}{\varphi\geq h}{0}{\infty}\}$.
\end{theorem}
In spite of a common thread, the treatment of the subcritical ($h>h_*$) and supercritical 
($h<h_*$) regimes involve significantly different ideas. The supercritical case is more 
involved, mostly due to the additional disconnection constraint present in the events. 
Correspondingly, the upper bounds of Theorem~\ref{P:ub} are furnished separately in two 
subsections. Section~\ref{sec:upper_sub} contains the proof of \eqref{eq:main_ubsubcrit_d=3} 
and \eqref{eq:main_ubsubcrit_d>4}, Section~\ref{sec:upper_sup} that of 
\eqref{eq:main_ubsupercrit_d=3} and \eqref{eq:main_ubsupercrit_d>4}.
 
\subsection{Upper bounds for the subcritical phase}\label{sec:upper_sub}
We start by giving an overview of the proof strategy leading to \eqref{eq:main_ubsubcrit_d=3} and \eqref{eq:main_ubsubcrit_d>4}. To any path connecting $0$ to $\partial 
B_N$ in $\{\varphi \ge h\}$, one associates, in view of 
Proposition~\ref{prop:coarse_paths}, a collection of well-separated boxes of carefully chosen size $L \ll N$, %(as will turn out $L$ needs to grow poly-logarithmically in $N$, see e.g.~\eqref{eq:bootstrap10})
each containing a box-to-box crossing at scale $L$. By the decomposition \eqref{eq:decomp_z} of $\varphi$ into the sum of 
$\psi$ and $\xi$ within each such box, it follows that when $h > h_*$, either $\xi$ is atypical for all but a small proportion of the boxes, or the localized version of the event, involving crossings in $\{\psi \ge h_{*} + \varepsilon\}$, behaves atypically for the remaining boxes. The corresponding events $E_{N, L}$ and $F_{N, L}$ (for $\psi$ and $\xi$, respectively), are defined in \eqref{eq:EandF} below. Together, they yield the central estimate \eqref{eq:bnd_psi_xi}, which drives the subsequent upper bounds. The key control on the event $F_{N, L}$ involving the harmonic average, derived in Lemma~\ref{lem:bnd_xibad} for $d=3$ and Lemma~\ref{lem:bnd_xibad4} for $d \geq 4$, is obtained by combining Lemma~\ref{lem:BTIS_Szn} and the capacity estimates of \eqref{eq:cg_capd=3}. %, \eqref{eq:cg_capd=4}. 
The resulting bound ends up carrying the leading order in \eqref{eq:main_ubsubcrit_d=3}. The localized event $E_{N, L}$ is dealt with in Lemmas~\ref{lem:bnd_psibadstep1} and~\ref{lem:bnd_psibadd4}, and essentially inherits a given a-priori estimate (for instance \eqref{eq:def_h**}). Pitting the resulting bounds against the entropy factor \eqref{eq:cg_gamma} coming from the choice of coarse-grainings in Proposition~\ref{prop:coarse_paths} leads to an improved bound on the one-arm probability, for suitably chosen box sizes $L$. This scheme can be applied as a bootstrapping mechanism, see
Proposition~\ref{prop:bootstrap} below, thus yielding the desired bound \eqref{eq:main_ubsubcrit_d=3} starting from \eqref{eq:def_h**} in several steps (in contrast, a single step suffices when $d \geq 4$).

\medskip

We now render the above precise. Let $h > h'$ and $\varepsilon \in (0, h - h')$. Referring to the notations from \eqref{eq:LL}--\eqref{eq:decomp_z}, given $L \ge 1$, $K \geq 100$ and a vertex $z\in \mathbb L = \mathbb L(L)$, we introduce the events
\begin{align}
	&\{\text{$z$ is $\psi$-bad}\} \coloneqq \{\lr{}{\psi^z\geq h' + \frac{\varepsilon}{4}}{C_z}{\partial \tilde C_z}\}, \mbox{ and} \label{eq:psibadsub}\\
	&\{\text{$z$ is $\xi$-bad}\} \coloneqq\big\{ \sup_{D_z}\xi^z \geq h-h' -\textstyle\frac{\varepsilon}{4}\big\}.\label{eq:xibadsub}
\end{align}
We also refer to the box $C_z$ as $\psi/\xi$-bad whenever $z$ is $\psi/\xi$-bad. Next, for any $N \ge 4K L$ and and $\rho \in (0, 1)$, consider the events
\begin{equation}
\label{eq:EandF}
\begin{split}
	&E_{N, L} = E_{N, L}^K(\rho, h, h', \varepsilon):=\left\{
	\begin{array}{c}
		\text{$\exists\,\mathcal{C}\in \mathcal A_{N, L}$ and $\tilde{\mathcal{C}}\subset \mathcal{C}$ with  $|\tilde{\mathcal{C}}| = \lceil\rho 
|\mathcal C|\rceil$} \\ 
		\text{such that all the sites $z \in \tilde{\mathcal{C}}$ are $\psi$-bad}
	\end{array}
	\right\}, \\
	&F_{N, L} = F_{N, L}^K(\rho, h, h', \varepsilon):=\left\{
	\begin{array}{c}
		\text{$\exists\,\mathcal{C}\in \mathcal{A}_{N, L}$ and $\tilde{\mathcal{C}}\subset \mathcal{C}$ with $|\tilde{\mathcal{C}}|=|\mathcal{C}|- \lceil\rho 
|\mathcal C|\rceil$} \\ 
		\text{such that all the sites $z\in \tilde{\mathcal{C}}$ are $\xi$-bad},
	\end{array}
	\right\}
\end{split}
\end{equation}
where $\mathcal A_{N, L} = \mathcal{A}^K_{N,L}(B_N)$ is the admissible collection given by Proposition~\ref{prop:coarse_paths}.
In applications below, the events $E_{N, L}$ and $F_{N, L}$ will typically be `bad', i.e.~have low probability, and $\rho$ will be close to $0$.  Since $\varphi = \xi^z + \psi^z$ on $U_z \supset 
\tilde{C}_z$ and $D_z \supset \tilde{C}_z$, it is then a consequence of the property~\eqref{def:admissible2} of $\mathcal A_{N, L}$ and \eqref{eq:psibadsub}, \eqref{eq:xibadsub}, \eqref{eq:EandF} that
\begin{equation}\label{eq:bnd_psi_xi}
\P[\lr{}{\varphi\geq h}{0}{\partial B_N}]\leq \P[E_{N, L}]+ \P[F_{N, L}].
\end{equation} 
%We will now bound each of $\P[E_{N, L}]$ and $\P[F_{N, L}]$ separately.

The following (a-priori) bound will be useful in dealing with \eqref{eq:psibadsub} and the event $E_{N, L}$ in \eqref{eq:EandF}. It will also apply to a different notion of $\psi$-badness in the next subsection, hence the general formulation. Consider an arbitrary increasing set $A \in \mathcal B(\R^{\tilde C_z})$. Then, in the notation of \eqref{eq:LBevent}, for all $\varepsilon >0$ there exists $\Cr{c:cap}(\varepsilon) \in (0,1)$ increasing in $\varepsilon$ with the property that, if for some $h'\in \R$ and $L \geq 1$,
\begin{equation}
\label{eq:inputbnd}
\P[A^{h'}(\varphi)] \le e^{-2f(L)}\,\,\mbox{with } \log 2 \le f(L) \le \Cr{c:cap}(\varepsilon) L,
\end{equation}
then
\begin{equation}
\label{eq:psibad}
\P\big[A^{h' + \frac\varepsilon{4}} (\psi^z)\big] \le e^{-f(L)}.
\end{equation}
Indeed, \eqref{eq:psibad} follows immediately from the decomposition $\varphi = \xi^z + \psi^z$ valid on $U_z$, whence
\begin{align*}
	\P\big[A^{h' + \frac\varepsilon{4}} (\psi^z)\big] 	\le\P[A^{h'}(\varphi)] + \P\big[\,\inf_{D_z} \xi^z \le -\varepsilon/4\,\big] 
	\le  \P[A^{h'}(\varphi)] + e^{-2\Cl[c]{c:cap}(\varepsilon) L} \stackrel{\eqref{eq:inputbnd}}{\le}  e^{-f(L)},
\end{align*}
where in the penultimate step we used Lemma~\ref{lem:BTIS_Szn} for the singleton $\mathcal C 
\coloneqq \{z\}$ along with the lower bound  ${\rm cap}(C_z) \ge c L$ from \eqref{eq:capball} (valid for all $d \ge 3$).

\medskip

At this point we consider the cases $d=3$ and $d \geq 4$ separately.

\medskip

\noindent {\bf Upper bound for $d = 3$.} Recall from \eqref{eq:def_h**} that the quantity $\P[\lr{}{\varphi\geq 
h}{0}{\partial B_N}]$ decays stretched exponentially in $N$ for every $h >h_{*}$ with some exponent $\beta = \Cr{c:strexp}(h) \in (0, 1)$. In what follows, we will bootstrap this decay to 
the one asserted by \eqref{eq:main_ubsubcrit_d=3} in -- as will soon turn out to be 
necessary -- two steps. This is encapsulated in the following proposition, from which the upper bound \eqref{eq:main_ubsubcrit_d=3} will quickly follow.
\begin{proposition}[Bootstrap]\label{prop:bootstrap}
	Let $h' \in \R$ and $\beta' \in (0, 1)$ be such that
	\begin{equation}
	\label{eq:bootstrapinput}
	%\P[\lr{}{\varphi\geq h'}{0}{\partial B_L}] \le e^{-c(h')\,L^{\beta'}}, \text{ for all $L \ge 1$.}
    \limsup_{N \to \infty}\frac{1}{N^{\beta'}}\log \P[\lr{}{\varphi\geq h'}{0}{\partial B_N}] < 0.
	\end{equation}
	Then for all $h > h'$, the following improved bounds hold, depending on the value of $\beta'$. If $\beta' \le 1/2$, then
	\begin{equation}
	\label{eq:bootstrapoutput1}
	%\P[\lr{}{\varphi\geq h}{0}{\partial B_L}] \le e^{ - \frac{\pi}{6}(h - h' - \varepsilon )^2 \frac{L}{(\log L)^{\beta}}}
	\limsup_{N \to \infty}\frac{1}{N^\beta}\log \P[\lr{}{\varphi\geq h}{0}{\partial B_N}] < 0 \mbox{ for every }\beta < 1,
	\end{equation}
    whereas if $\beta' > 1/2$, then
		\begin{equation}
	\label{eq:bootstrapoutput2}
	%\P[\lr{}{\varphi\geq h}{0}{\partial B_L}] \le e^{ - \frac{\pi}{6}(h - h' - \varepsilon )^2 \frac{L}{(\log L)^{\beta}}}
	\limsup_{N \to \infty}\frac{\log N}{N}\log \P[\lr{}{\varphi\geq h}{0}{\partial B_N}] \le \frac{\pi}{6}(h - h')^2.
	\end{equation}
	%where $\beta = \beta(\beta') > 1$ for $\beta' \le 1/2$ and $=1$ otherwise.% is such that $\beta = 1$ whenever $\beta' > 1/2$.
\end{proposition}

Assuming Proposition~\ref{prop:bootstrap} to hold, we first give the short:

\begin{proof}[Proof of \eqref{eq:main_ubsubcrit_d=3}]
Let $h > h_*$. By \eqref{eq:def_h**} we have \eqref{eq:bootstrapinput} at any height $h' > h_*$ with exponent $\beta' =  \Cr{c:strexp}(h') > 0$. Therefore, by \eqref{eq:bootstrapoutput1}, we obtain that \eqref{eq:bootstrapinput} holds for every $h' > h_*$ and $\beta' < 1$ -- in particular for $\beta' = 3/4$ (say). Consequently, we obtain the bound in \eqref{eq:bootstrapoutput2} for any $h' \in (h_*, h)$. The result now follows by sending $h' \to h_*$.%Let $h>h_{*}$. For any $\delta\in (0, (h-h_{*})/3)$, pick $h_1\coloneqq h_{*}+\delta$, $h_2\coloneqq h_{*} +2\delta$, so that $h_{*}< h_1<h_2< h$.
%	By \eqref{eq:def_h**} we have \eqref{eq:bootstrapinput} at height $h'=h_1$ for some value of $\beta'$ and $c$ (both depending on $\delta$). Therefore, by applying Proposition~\ref{prop:bootstrap} we get the improved bound \eqref{eq:bootstrapoutput} at every $h_2 > h_1$ and consequently the bound \eqref{eq:bootstrapinput} with $h'=h_2$ for any 
%	$\beta' < 1$. Thus choosing for instance $\beta' =3/4$ (any value larger than $1/2$ would do), we can deduce from \eqref{eq:bootstrapoutput}
%	\begin{equation*}
%		\limsup_{N\to\infty}\, \frac{\log N}{N} \log \P[\lr{}{\varphi\geq h}{0}{\partial B_N}] \le - \frac{\pi}{6}(h - h_{*} - 2\delta - \varepsilon )^2
%	\end{equation*}
%	for every value of $\varepsilon < \delta$. The result then follows by sending $\varepsilon \to 0$ and subsequently $\delta \to 0$.
\end{proof}
\begin{remark}
	\label{remark:bootstrap}
	%\begin{enumerate}[label={\arabic*)}]
%\item	
Proposition~\ref{prop:bootstrap}  highlights in a transparent form the paradigm underlying our strategy to obtain sharp upper bounds. Indeed, a similar (but considerably more involved) bootstrapping mechanism is at work in the supercritical regime; see  
Section~\ref{sec:upper_sup}. The choice \eqref{eq:bootstrapinput} as a starting point for the 
bootstrap reflects the fact that  stretched exponential estimates naturally come out of the 
static renormalization arguments leading to the existence of a non-trivial subcritical regime, 
see~\cite{RoS13}. One could forego one step in deducing \eqref{eq:main_ubsubcrit_d=3} as 
\eqref{eq:bootstrapoutput1} is implied by the strongest available results \cite{PR15,PT12}, but 
our findings do not rely on these. Moreover, we will face similar issues in the supercritical regime, where such results are not available {\em a-priori}. In fact, one could even deduce the desired bound 
\eqref{eq:bootstrapoutput2} from a much weaker a-priori estimate than a stretched-exponential bound by bootstrapping a few more times, see Remark~\ref{remark:stronger_bootstrap} below.
%\end{enumerate}
	\end{remark}
	\medskip
We now aim at showing Proposition~\ref{prop:bootstrap}. Its proof combines individual estimates for $\P[E_{N, L}]$ and
$ \P[F_{N, L}]$, cf.~\eqref{eq:EandF} and \eqref{eq:bnd_psi_xi}, which are supplied in the following two lemmas.

	\begin{lemma}\label{lem:bnd_psibadstep1} $(\rho \in (0, 1), \, K \geq 100, \, h > h', \, \varepsilon \in (0, h - h')).$ %$(\rho \in (0, 1), L \ge L_0(\rho), \, K \geq 100).$
If \eqref{eq:inputbnd} holds with $A^{h'} (\varphi)= \{\lr{}{\varphi \ge h'}{C_0}{\partial {\tilde C}_0}\}$ for some $L \geq 1$, then with $E_{N, L}=  E_{N, L}^K(\rho, h, h', \varepsilon)$, for all $N \ge 10K L$ one has
\begin{equation}
\label{eqn:bnd_psibadstep1}
\log \P[E_{N, L}] \le n(C \log (nK) - \rho f(L)),
	\end{equation}	
	 where $n = |\mathcal C|$ for any $ \mathcal{C} \in \mathcal A$ (cf. \eqref{def:admissible1} in Definition~\ref{def:admissible}).
\end{lemma}

\begin{proof} On account of \eqref{eq:EandF} and by a union bound, one obtains
	\begin{equation}\label{eq:pf_bnd_psibad_1}
		\begin{split}
			\P[E_{N, L}]&\leq |\mathcal{A}_{N, L}| {n\choose \lceil \rho n \rceil} \sup_{\mathcal{C}\in \mathcal{A}_{N, L}} \sup_{\substack{\tilde{\mathcal{C}}\subset \mathcal{C} \\ |\tilde{\mathcal{C}}|= \lceil \rho n \rceil}}  \P[\text{$z$ is $\psi$-bad},\, \forall z\in \tilde{\mathcal{C}}]\\
			&\leq e^{Cn\log (nK)} \sup_{\mathcal{C}\in \mathcal{A}_{N, L}} \sup_{\substack{\tilde{\mathcal{C}}\subset \mathcal{C} \\ |\tilde{\mathcal{C}}|=\lceil \rho n \rceil}}  \P[\text{$z$ is $\psi$-bad},\, \forall z\in \tilde{\mathcal{C}}]
		\end{split}
	\end{equation}
where the second line follows using \eqref{def:admissible3}, \eqref{eq:cg_gamma} and the lower bound on $n$ from \eqref{def:admissible1} in order to bound $|\mathcal A_{N, L}|$. Now, due to the independence property \eqref{eq:decomp_boxes} and by translation invariance, one has, for any $\tilde{\mathcal{C}}\subset\mathcal{C}\in \mathcal{A}$ with $|\tilde{\mathcal{C}}|= \lceil \rho n \rceil$, since \eqref{eq:inputbnd} holds,
	\begin{equation*}%\label{eq:pf_bnd_psibad_2}
	\P[\text{$z$ is $\psi$-bad},\, \forall z\in \tilde{\mathcal{C}}] = \P[\text{$0$ is $\psi$-bad}]^{\lceil\rho n \rceil} \stackrel{\eqref{eq:psibadsub}, \eqref{eq:psibad}}{\le} \ e^{-\rho n f(L)},
	\end{equation*}
	which together with \eqref{eq:pf_bnd_psibad_1} gives \eqref{eqn:bnd_psibadstep1}.
\end{proof}

\medskip

Next we present the relevant bound for $\P[F_{N, L}]$. Fix any function $L_1(N)$ such that $(\log L_1(N )/\log \log N) \to \infty$ and $(\log L_1(N) / \log N) \to 0$ (cf.~below \eqref{eq:cg_capd=3} in Proposition~\ref{prop:coarse_paths}). As will become clear, as soon as $L$ grows fast enough with $N$, the `energy term' stemming from $F_{N,L}$ will dominate the `entropy term' arising from the relevant union bound (similar to \eqref{eq:pf_bnd_psibad_1} above).

\begin{lemma}\label{lem:bnd_xibad}$(\rho \in (0, 1), \, K \geq 100, \, h > h', \, \varepsilon \in (0, h - h')).$ 
For any $\theta > 0$, one has
	\begin{equation}\label{eq:bnd_xibad} 	
\limsup_{N \to \infty}\,\sup_{L \in [ (\log N)^{2+\theta}, L_1(N)]}\frac{\log N}{N}\log \P[F_{N, L}] \leq% \frac{\Cl{C:offset}}{\lambda K} 
- \frac{\pi}{6} \frac{\lambda(K)(1-\rho)}{\alpha(K)} (h - h' - \varepsilon/4)^2,
	\end{equation}
with $\alpha(K)$ and $\lambda(K)$ as appearing in Lemma~\ref{lem:BTIS_Szn} and Proposition~\ref{prop:coarse_paths}, respectively.
\end{lemma}

\begin{proof}
	Recalling \eqref{eq:xibadsub} and \eqref{eq:EandF}, and proceeding as in \eqref{eq:pf_bnd_psibad_1}, one finds that
	\begin{align*}%\label{eq:pf_bnd_xibad_1}
		\begin{split}
		\log	\P[F_{N, L}]%&\leq |\mathcal A_{N, L}| {n\choose \lceil \rho n \rceil} \sup_{\mathcal{C}\in \mathcal{A}_{N, L}} \sup_{\substack{\tilde{\mathcal{C}}\subset \mathcal{C} \\ |\tilde{\mathcal{C}}| = n - \lceil \rho n \rceil}}  \P[\text{$z$ is $\xi$-bad}\;\, \forall z\in \tilde{\mathcal{C}}]\\
			&\leq C n\log (nK) + \sup_{\mathcal{C}\in \mathcal{A}_{N, L}} \sup_{\substack{\tilde{\mathcal{C}}\subset \mathcal{C} \\ |\tilde{\mathcal{C}}| = n - \lceil \rho n \rceil}} \log \P\Big[\bigcap_{z\in\tilde{\mathcal{C}}} \big\{\sup_{x\in D_z} \xi_x^z\geq h - h' - \varepsilon/4\big\}\Big].
		\end{split}
	\end{align*}
Denoting the event on the right hand side above by $F(\tilde{\mathcal C})$, we get, combining Lemma~\ref{lem:BTIS_Szn}, the capacity lower bound \eqref{eq:cg_capd=3} from Proposition~\ref{prop:coarse_paths} and the fact that $n \leq N/LK$ by \eqref{def:admissible1},
\begin{align*}
%\label{eq:Fbndbtis}
&\sup_{L \in [(\log N)^{2+\theta}, L_1(N)]}\,\sup_{\mathcal{C}\in \mathcal{A}_{N, 
L}} \sup_{\substack{\tilde{\mathcal{C}}\subset \mathcal{C} \\ |\tilde{\mathcal{C}}| = 
n - \lceil \rho n \rceil}} \log \P[F(\tilde{\mathcal C})] \\
&\qquad \qquad \le -\frac{1}{2}(h - h' - \varepsilon/4 - \delta)_+^2 \frac{\lambda(K)(1-\rho)}{\alpha(K)}\mathrm{cap}(T_N)(1+o_N(1)),
\end{align*}
where $$\delta = \delta(N, K, \rho, \theta) \coloneqq \displaystyle C\sqrt{\frac{N}{\lambda(K)(1-\rho)K^3(\log N)^{2+\theta}\mathrm{cap}(T_N)}} \longrightarrow 0 ,\text{ as $N \to \infty$},$$ 
using the asymptotics for $\mathrm{cap}(T_N)$ from \eqref{eq:cap_line_asymp} in the last step. Finally notice that
$$n \log (nK) \le \frac{N}{ K (\log N)^{1+\theta}}\, \text{ for $L \ge  (\log N)^{2+\theta}$}.$$
The lemma follows by combining the previous displays along with the asymptotic of ${\rm cap}(T_N)$.
\end{proof}
With Lemmas~\ref{lem:bnd_psibadstep1} and \ref{lem:bnd_xibad} at hand, we proceed to the
\begin{proof}[Proof of Proposition~\ref{prop:bootstrap}] 
%Let us begin by 
%
%define the different values of 
%
%define $\lambda = \lambda(\beta) > 0$ as follows
%\begin{equation*}
%%\label{eq:lambda}
%\lambda \coloneqq
%\begin{cases}
%	1/\beta - 2%\frac{ \Cl{C:cg_complexity}}{K}r\log r)
%	, &\text{ if } \beta \in (0, 1/2)\\
%	1/4,  &\text{ if } \beta = 1/2\\
%	1/ (1 -\beta) - 2, &\text{ if } \beta \in (1/2, 1).
%\end{cases}
%\end{equation*}
Let $h' \in \R$ be such that \eqref{eq:bootstrapinput} holds and consider $h> h'$ and $\varepsilon \in (0,h-h')$.
First choose $K \geq 100$ large enough and $\rho \in (0,1)$ close enough to $0$, both 
depending on $h$, $h'$ and $\varepsilon$, such that, applying Lemma \ref{lem:bnd_xibad}, one 
obtains, for all $\theta > 0$, $N \geq C(\varepsilon, h,h' ,\theta)$ and $L \in  [ (\log 
N)^{2+\theta}, L_1(N)]$,
\begin{equation}
\label{eq:bootstrapFNL}
\log \P[F_{N, L}] \leq - \frac{\pi}{6} (h-h'- \varepsilon/2)^2 \frac{N}{\log N}
\end{equation}
(recall to that effect that both $\alpha(K)$ and $\lambda(K)(1-\rho)$ converge to 1 in the limit $\rho \to 0$ and $K \to \infty$ by Lemma~\ref{lem:BTIS_Szn} and \eqref{eq:cg99d=3} in Proposition \ref{prop:coarse_paths}, respectively). Now for any  $N \geq 10^3$, let
\begin{equation}
\label{eq:bootstrap10}
L = L(N, \theta)\coloneqq (\log N)^{2+\theta},
\end{equation}
for some $\theta >0$ to be chosen. Notice that with \eqref{eq:bootstrap10} and by \eqref{eq:bootstrapinput}, the condition \eqref{eq:inputbnd} holds with $A^{h'} (\varphi)= \{\lr{}{\varphi \ge h'}{C_0}{\partial {\tilde C}_0}\}$ and $f(L) = c(h',\beta')L^{\beta'}$ whenever $N \ge C(\varepsilon, h,h' ,\theta, \beta')$. With this choice of $f(\cdot)$ and since $nK \leq N/L$
by \eqref{def:admissible1}, it follows that
 $$C\log (n K) - \rho f(L) \le - \rho f(L)/2, \text{ for all $N \ge C(\varepsilon, h,h' ,\theta, \beta')$}$$
as soon as $\theta$ is chosen such that
\begin{equation}
\label{eq:lambdabnd1}
(2 + \theta) \beta' > 	1.
\end{equation}
Hence applying Lemma~\ref{lem:bnd_psibadstep1} and using the lower bound on $n$ from \eqref{def:admissible1}, one gets (with $L$ as in \eqref{eq:bootstrap10})
\begin{equation}
\label{eq:bootstrapENL}
\log \P[E_{N, L}] \le -\frac{c(h', \beta')\rho N}{ K (\log N)^{(2 + \theta)(1 - \beta')}},
\end{equation}
provided \eqref{eq:lambdabnd1} is satisfied and $N$ is sufficiently large. Plugging the bounds from \eqref{eq:bootstrapFNL} and \eqref{eq:bootstrapENL} into \eqref{eq:bnd_psi_xi} we immediately deduce, letting $N\to \infty$ and then $\varepsilon \to 0$, that
\begin{equation}
\label{eq:subcrit_final1}
\lim_{N \to \infty}\frac{(\log N)^{\beta}}{N}\log \P[\lr{}{\varphi\geq h}{0}{\partial B_N}] \le - \frac{\pi}{6} (h-h')^2
\end{equation}
%\begin{equation}
%\label{eq:subcrit_final1}
%	\P[\lr{}{\varphi\geq h}{0}{\partial B_N}] \le e^{-  \frac{\pi}{6} (h-h'- \varepsilon)^2 \frac{N}{(\log N)^{\beta}}}
%\end{equation}
%for all $N \ge C(\varepsilon, h,h', \theta, \beta')$ and 
for any value of $\beta$ satisfying %$\beta > (2 + \theta)(1 - \beta')$ when $(2 + \theta)(1 - \beta') \geq 1$ and $1$ otherwise.
\begin{equation}
\label{eq:subcrit_beta'}
  \begin{array}{ll}
	\beta > (2 + \theta)(1 - \beta'), &\text{ if } (2 + \theta)(1 - \beta') \geq 1\\
	\beta = 1, &\text{ otherwise}
	\end{array}
\end{equation}
and any choice of $\theta >0$ such that \eqref{eq:lambdabnd1} is satisfied.
If $\beta' \leq 1/2$, the conditions $\theta > 0$ and $ (2 + \theta)(1 - \beta') < 1$ cannot 
simultaneously hold. Hence, in this case, choosing for example $\theta=\theta(\beta')$ so that 
$(2 + \theta) \beta' = 2$, whence \eqref{eq:lambdabnd1} is satisfied, \eqref{eq:subcrit_final1} 
yields the bound \eqref{eq:bootstrapoutput1}. %for some exponent $\beta>1$. 
On the other hand when $\beta' > 1/2$, the conditions \eqref{eq:lambdabnd1} and $ (2 + \theta)(1 - \beta') < 1$ 
can be recast as
$$
 \frac{1}{\beta'}-2<\theta< \frac{1}{1-\beta'}-2 
$$
(note that the interval of admissible values for $\theta$ is non-degenerate because 
$\beta'>1/2$). So choosing for instance $\theta = \frac{1}{2\beta'(1-\beta')}-2$, we obtain \eqref{eq:subcrit_final1} with $\beta=1$ (since $(2 + \theta)(1 - \beta') < 1$ holds), which is \eqref{eq:bootstrapoutput2}.%and setting $(2 + \theta) \beta' = 2$ (so that \eqref{eq:lambdabnd1} is satisfied), we can immediately deduce
%\begin{equation*}
%\P[\lr{}{\varphi\geq h}{0}{\partial B_N}] \le e^{-  \frac{\pi}{6} (h-h'- \varepsilon)^2 \frac{N}{(\log N)^{\beta}}}
%\end{equation*}
%for all $N \ge C(\varepsilon, h,h',\beta')$ where we can choose any $\beta > \max(1, (2 + \theta)(1 - \beta'))$. When $\beta' > 1/2$, on the other hand, we set $$\theta \coloneqq \frac{1}{2\beta'(1 - \beta')} - 2 > 0$$
%which satisfies \eqref{eq:lambdabnd1} as well as $(2 + \theta)(1 - \beta') < 1$ and consequently we get the previous bound with $\beta = 1$.
\end{proof}
\begin{remark}\label{remark:stronger_bootstrap}
A careful examination of the proof of Proposition~\ref{prop:bootstrap} reveals that a stretched exponential a-priori bound
such as \eqref{eq:bootstrapinput} is not required to arrive at \eqref{eq:main_ubsubcrit_d=3}. Indeed one could for instance obtain the same result by means of a few additional bootstrapping steps starting from a much weaker estimate of the type $ \limsup_L \frac{\P[\lr{}{\varphi \geq h }{C_{0, L}}{\partial \tilde C_{0, L}}]}{ (\log L)^{\beta'(h)}}< 0$ for some $\beta'(h) > 0$ and all $h> h_*$ (or even a $k$-fold composition of $\log$, for some fixed integer $k\geq 1$). Combining with other existing methods, see e.g.~\cite{PR15}, one would further obtain that \eqref{eq:main_ubsubcrit_d=3} holds as soon as $\P[\lr{}{\varphi \geq h }{C_{0, L}}{\partial \tilde C_{0, L}}]$ is bounded from above by a suitable $c(d)\in (0,1)$ uniformly along a diverging subsequence of scales $L$. Similar conclusions could be drawn in the supercritical regime, cf. Remark~\ref{R:final},1).

\end{remark}

\bigskip

\noindent {\bf Upper bound for $d \geq 4$.} We now supply the proof of  
\eqref{eq:main_ubsubcrit_d>4}. Throughout the remainder of Section~\ref{sec:upper_sub}, for an arbitrary level $h > h_{*}$ (as appearing in \eqref{eq:main_ubsubcrit_d>4}), we simply fix $h'= 
(h_{*}+h)/2$, $\varepsilon = (h - h') / 8$, $\rho = 1/2$ and $K = 100$ in \eqref{eq:psibadsub}--\eqref{eq:EandF}. %The relation \eqref{eq:bnd_psi_xi} remains valid with these choices. 
The events $E_{N,L}$ and $F_{N,L}$ thus effectively depend on the sole parameter $h$. The following two results replace Lemmas~\ref{lem:bnd_psibadstep1} and \ref{lem:bnd_xibad}, respectively.

\begin{lemma}\label{lem:bnd_psibadd4}
$(d \geq4, \, h>h_*)$	For all $L \geq C(h)$ and $N \ge 10^3L$ one has
%any $L \ge L_0 = L_0(1/2)$ and $N \ge 4K L$ one has,
	\begin{equation}
	\label{eq:psibadd=4subcrit}
		%\P[E_{N, L}] \le  e^{-c(h, h') N / L^{1- \beta'}}.
		\P[E_{N, L}] \le  e^{-c(h) N / L}.
	\end{equation}	
\end{lemma}
\begin{proof}
The proof mimics that of Lemma~\ref{lem:bnd_psibadstep1}, with small modifications. Proceeding as in \eqref{eq:pf_bnd_psibad_1}, using \eqref{def:admissible3}, \eqref{eq:cg_gamma4} and the bound $\tfrac{cN}{L (\log L)^2} \leq n$ from \eqref{def:admissible1} to bound $|\mathcal{A}_{N,L}^{100}|$, one finds,
\begin{align}\label{eq:pf_bnd_psibad_1.11}
	\begin{split}
		\P[E_{N, L}] %&\leq |\mathcal{A}_{N, L}| {n\choose \lceil \rho n \rceil} \sup_{\mathcal{C}\in \mathcal{A}_{N, L}} \sup_{\substack{\tilde{\mathcal{C}}\subset \mathcal{C} \\ |\tilde{\mathcal{C}}|= \lceil \rho n \rceil}}  \P[\text{$z$ is $\psi$-bad}\;\, \forall z\in \tilde{\mathcal{C}}]\\
		&\leq e^{Cn(\log L)^2}\P[\text{$0$ is $\psi$-bad}]^{\lceil n/2 \rceil} \le e^{n(C(\log L)^2 - c(h)L^{\Cr{c:strexp}})}, %\le e^{-c'(h)N/L},% \sup_{\mathcal{C}\in \mathcal{A}_{N, L}} \sup_{\substack{\tilde{\mathcal{C}}\subset \mathcal{C} \\ |\tilde{\mathcal{C}}|=\lceil \rho n \rceil}}  \P[\text{$z$ is $\psi$-bad}\;\, \forall z\in \tilde{\mathcal{C}}]
	\end{split}
\end{align}
for all $L \geq 1$, $N \geq 10^3L$, where the first inequality also relies on the independence property \eqref{eq:decomp_boxes} and the second one on the fact that \eqref{eq:def_h**} and \eqref{eq:inputbnd}--\eqref{eq:psibad} combine to give a suitable bound on $\P[\text{$0$ is $\psi$-bad}]$. Using the lower bound on $n$ yet again, \eqref{eq:psibadd=4subcrit} readily follows from \eqref{eq:pf_bnd_psibad_1.11}.
	\end{proof}
%\noindent The proof is exactly same as that of Lemma~\ref{lem:bnd_psibadstep1} in view of \eqref{eq:psibad} with appropriate application of Proposition~\ref{prop:coarse_paths} to bound the cardinality of $\mathcal A_{N, L}$. The improvement compared to $d = 3$ is due to the absence of the factor $\log n$ in the (log-)complexity of $\mathcal A_{N, L}$ for $d \ge 4$. 
\begin{remark}\label{remark:stronger_bootstrapdge4}
The conclusions of Lemma~\ref{lem:bnd_psibadd4} would remain unaltered if one replaced \eqref{eq:def_h**} by the (weaker) assumption that  $\P[\lr{}{\varphi \geq h }{C_{0, L}}{\partial \tilde C_{0, L}}] \leq e^{ -c(h)(\log L)^{2+\varepsilon}}$, for some $\varepsilon=\varepsilon(h) >0$ and all $L \geq 1$, $h>h_*$. This is related to the power in the definition of $\varepsilon_m$ in \eqref{eq:coarsegrain_r_n}, and could be relaxed to a `$1+\varepsilon$'-condition by suitable modification of \eqref{eq:coarsegrain_r_n} and the subsequent arguments of Section~\ref{sec:coarse4d}, which would lead to a corresponding improvement of the lower bound on $n$ in~\eqref{def:admissible1}.
%Observe that the lemma also holds if $f(L) > (\log L)^{C'}$ for some $C' > C$.
	\end{remark}
The analogue of Lemma~\ref{lem:bnd_xibad} is
\begin{lemma}\label{lem:bnd_xibad4} $(d \geq 4, \, h > h_*)$
	For some $C(h) > 0$,
	\begin{equation}
		\label{eq:xibadd=4subcrit}	
		\limsup_{N \to \infty}\,\sup_{L \in [ C(h), N / 10K]}\frac{1}{N}\log \P[F_{N, L}] <% \frac{\Cl{C:offset}}{\theta K} 
		0.
	\end{equation}
\end{lemma}
\begin{proof} For arbitrary $L \geq 1$, $N \geq 10^3L$ and any given collection $\mathcal{C}\in \mathcal{A}= \mathcal{A}_{N,L}^{100}$ and $\tilde{\mathcal{C}}\subset \mathcal{C}$ with $|\tilde{\mathcal{C}}| = n - \lceil n/2 \rceil$, one obtains by virtue of \eqref{eq:cg_capd=3} and \eqref{eq:cap_line_asymp4d} that $\mathrm{cap}(\tilde{\Sigma}) \geq c N$. Together with Lemma~\ref{lem:BTIS_Szn}, this is seen to imply that, 
\begin{equation}
\label{eq:xibadd=4subcrit.1}
\P\Big[\bigcap_{z\in\tilde{\mathcal{C}}} \big\{\sup_{x\in D_z} \xi_x^z\geq h - h' - \textstyle\frac{\varepsilon}{4} \displaystyle\big\}\Big] \leq e^{-c(h)N}
\end{equation}
(with $h'$, $\varepsilon$ as defined above Lemma~\ref{lem:bnd_psibadd4}) whenever $L \geq C(h)$ and $N \geq 10^3L$, noting that $\frac{|\tilde{\mathcal{C}}|}{\mathrm{cap}(\tilde{\Sigma})} \leq cL^{-1}$ becomes suitably small for such $L$, cf.~\eqref{eq:BTIS_Szn}. In view of \eqref{eq:EandF}, applying a union bound over the choices of $\mathcal{C}\in \mathcal{A}_{N, L}$ and $\tilde{\mathcal{C}}\subset \mathcal{C}$, \eqref{eq:xibadd=4subcrit} is easily seen to follow since 
$\log |\mathcal{A}| \leq Cn (\log L)^2 \leq C\frac{N}{L} $, see~\eqref{eq:pf_bnd_psibad_1.11} and the upper bound on $n$ from \eqref{def:admissible1}. Thus, the resulting combinatorial complexity doesn't spoil the upper bound in \eqref{eq:xibadd=4subcrit.1} whenever $L \geq C'(h)$.
\end{proof}

\begin{proof}[Proof of \eqref{eq:main_ubsubcrit_d>4}]
The upper bound \eqref{eq:main_ubsubcrit_d>4} follows immediately by combining \eqref{eq:bnd_psi_xi}, \eqref{eq:psibadd=4subcrit} and \eqref{eq:xibadd=4subcrit} upon choosing $L=C(h)$ large enough to ensure both Lemmas~\ref{lem:bnd_psibadd4} and \ref{lem:bnd_xibad4} are in force. \end{proof}

% by combining the previous two lemma.%\begin{remark}\label{remark:stronger_bootstrapdge4}
%In the spirit of Remark~\ref{remark:stronger_bootstrap}, we can verify that the proof works 
%even if we started with any $f(L)$ satisfying $\lim_{L \to \infty} f(L) = \infty$, i.e. if 
%$\lim_{N \to \infty}\P[\lr{}{\varphi \geq h }{C_{0, N}}{\partial \tilde C_{0, N}}] = 0$ 
%throughout the subcritical regime. The improvement compared to $d = 3$ is due to the absence of the factor $\log n$ in the (log-)complexity of $\mathcal A$ (recall Lemma~\ref{lem:bnd_psibadstep1})
%
%is the absence of the factor $\log n$ from the 
%	\end{remark}
\begin{remark}
\label{Rk:onionsd=4}
%Explain that coarse-graining scheme for $d=3$ leads to near-optimal results. Refer to Remark~\ref{remark:coarsening},3) which entails that coarse-graining with onions still work and give the same capacity lower bound.
Following up on Remark~\ref{R:cg3to4}, we describe which upper bounds can be derived for $d \geq 4$ using the collection $\mathcal{A}'$ (obtained by following the coarse-graining scheme used for $d=3$). For $h > h_*$, fix $h'$, $\varepsilon$, $K$ and $\rho$ as above \eqref{lem:bnd_psibadd4} and define $E_{N, L}' = E_{N, L}^K(\rho, h, h', \varepsilon)'$, $F_{N, L}'$ as in \eqref{eq:EandF}, but with $\mathcal{A}'$ in place of $\mathcal{A}$. Similarly as in \eqref{eq:bnd_psi_xi}, by admissibility of $\mathcal{A}'$, one sees that $\P[\lr{}{\varphi\geq h}{0}{\partial B_N}]\leq \P[E_{N, L}']+ \P[F_{N, L}']$. Using 
\eqref{eq:cg_capd=3}, \eqref{eq:xibadd=4subcrit.1} and recalling the larger combinatorial complexity of $|\mathcal{A}'|$ (as in \eqref{eq:pf_bnd_psibad_1} for instance), one finds that
\begin{equation}
\label{eq:subcritconcl1}
\P[F_{N, L}'] \leq e^{C n\log (nK)}e^{-c(h)N} \stackrel{\eqref{def:admissible1}}{\leq} e^{CN \frac{\log N}{L} - cN} \leq e^{-c'(h)N}, \text{ if $C(h)\log N \leq L \leq cN$.}
\end{equation}
Regarding $\P[E_{N, L}']$, observe that the bound \eqref{eqn:bnd_psibadstep1} from Lemma~\ref{lem:bnd_psibadstep1} holds with $E_{N, L}'$ instead of $E_{N, L}$. This crucially uses the fact that $n=|\mathcal{C}|$ for $\mathcal{C} \in \mathcal{A}'$ satisfies the lower bound \eqref{def:admissible1} with $u(x)=x$, which is used in the proof of  Lemma~\ref{lem:bnd_psibadstep1}. As a consequence, one obtains
\begin{equation}
\label{eq:subcritconcl2}
 \P[E_{N, L}'] \le e^{ - c \frac{f(L)N}{L}}, \text{ if $f(L) \geq C \log N$, $L \leq cN$,}
	\end{equation}	
assuming \eqref{eq:inputbnd} holds with $A^{h'} (\varphi)= \{\lr{}{\varphi \ge h'}{C_0}{\partial {\tilde C}_0}\}$.

As we now explain, one can deduce from \eqref{eq:subcritconcl1} and \eqref{eq:subcritconcl2} that for any integer $k \geq 1$
\begin{equation}
\label{eq:subcritconcl3}
 \P[\lr{}{\varphi \ge h}{0}{\partial B_{N}}] \le e^{ - c(h) N/(\log^{(k)}\hspace{-0.3ex}N)^{\Cl{C:logk}}}, \text{ for all $N \geq 1$, $h > h_*$,}
\end{equation}	
where $\Cr{C:logk} = \Cr{C:logk}(d) \in [1, \infty)$ and $\log^{(k)}(\cdot)$ denotes the $k$-fold composition of $\log(\cdot)$.
To obtain \eqref{eq:subcritconcl3} one proceeds as follows: in a separate (first) step, starting from \eqref{eq:def_h**}, one chooses $f(L)=c(h)L^{\Cr{c:strexp}}$ and $L=C(h)(\log N)^{1+\Cr{c:strexp}^{-1}}$, whence \eqref{eq:subcritconcl1} and \eqref{eq:subcritconcl2} apply and yield \eqref{eq:subcritconcl3} for $k=1$ with $\Cr{C:logk}= \Cr{c:strexp}^{-1}-\Cr{c:strexp}$. Now, assuming \eqref{eq:subcritconcl3} to hold for some $k \geq 1$, one chooses $f(L)=c(h) L/(\log^{(k)}\hspace{-0.3ex}L)^{\Cr{C:logk}}$ and $L = C(h) \log N (\log^{(k+1)}\hspace{-0.3ex}N)^{\Cr{C:logk}}$, whence \eqref{eq:subcritconcl1} and  \eqref{eq:subcritconcl2} apply (in particular \eqref{eq:inputbnd} holds with this choice of $f(\cdot)$ due to \eqref{eq:subcritconcl3}, and the conditions on $L$ and $f(L)$ in  \eqref{eq:subcritconcl1}, \eqref{eq:subcritconcl2} are met) and readily yield \eqref{eq:subcritconcl3} with $k+1$ instead of $k$.

Note that \eqref{eq:subcritconcl3} is nearly \eqref{eq:main_ubsubcrit_d>4} and that the obstruction to obtaining the desired result comes from competition between the entropy of $\mathcal{A}'$ and the local fields leading to \eqref{eq:subcritconcl2}, not the contribution \eqref{eq:subcritconcl1} from the harmonic average, which exhibits the desired exponential decay.
 
 \end{remark}
\subsection{Supercritical phase}\label{sec:upper_sup}
%In addition to \eqref{eq:main_ubsupercrit_d=3} and \eqref{eq:main_ubsupercrit_d>4}, we will also prove the upper bound on the probability of a local uniqueness event which might be of independent interest. To this end let us define:
%\begin{equation*}
%{\rm LocUniq}(L, h) \coloneqq \big\{\exists \text{ a unique cluster in $\{\varphi \ge h\}$ crossing }\tilde C_{0, L} \setminus C_{0, L}\big\}.
%\end{equation*}
%\begin{proposition}\label{prop:locuniq_ub} If $d = 3$ and $h < \bar h$, we have
%	\begin{equation}
%	\label{eq:ub_locuiq_d=3}
%	\limsup_{N\to\infty}\, \frac{\log N}{N} \log \P[{\rm LocUniq}(N, h)^c] \leq -\frac{\pi}{6}(\bar h - h)^2,
%	\end{equation}
%	whereas for $d \ge 4$ and $h < \bar h$, we have
%		\begin{equation}
%	\label{eq:ub_locuiq_d>4}
%\limsup_{N\to\infty}\, \frac{1}{N} \log \P[{\rm LocUniq}(N, h)^c] < 0.
%	\end{equation}
%\end{proposition}
%
We now proceed to the proofs of \eqref{eq:main_ubsupercrit_d=3} and \eqref{eq:main_ubsupercrit_d>4} in Theorem~\ref{P:ub}, along with the corresponding statements for $ \textnormal{LocUniq}(N,h)^{c}$, $h< h_*$, which we will actually prove first. In all cases, our argument revolves around a notion 
of good event $G_N$, see \eqref{def:goodevent}, which will allow us to construct ambient clusters with 
certain desirable properties. The bottom line is that it will be costly for any large connected set to avoid connecting to any such ambient cluster. This is quantified in Lemma~\ref{lem:truncated}. The desired estimate for $\P[G_N]$ is then arrived to by means of a renormalization scheme, starting from a certain localized good event $\mathcal{G}_z$, see Definition~\ref{def:goodevent2}, which satisfies a suitable a-priori estimate, see Lemma~\ref{lem:loc_amb_cluster1}. Importantly, the scheme, whose essential features are captured by Proposition~\ref{lem:inclusion_fxn_goodevent}, improves not only probabilistic estimates but also the number $a$ of contact points inherent to the definition of $\mathcal{G}_z$, see \eqref{def:goodevent2.2}, in each step of the iteration. The proofs of \eqref{eq:main_ubsupercrit_d=3} and \eqref{eq:main_ubsupercrit_d>4} for all events of interest follow by combining Lemma~\ref{lem:truncated}, Proposition~\ref{lem:inclusion_fxn_goodevent} and Lemma~\ref{lem:loc_amb_cluster1} and are presented at the end of this section.

For $\tilde{f}:\tilde{\Z}^d\to \R$ (cf.~above \eqref{eq:phi_extend1} regarding 
$\tilde{\Z}^d\supset \Z^d$), we define the local average $(A\tilde{f})_x= (2d)^{-1} \sum_{m \in \mathbb{M}^d: m\sim x}\tilde 
f_m$, for $x \in {\Z}^d$. For 
integers $L_0 \geq 1$ and $M >1$, we then introduce the set 
		\begin{equation}
		\label{eq:ubsuper1}
%\label{eq:conditional_fin_energy}
\mathcal{M}(\tilde{f}) \coloneqq \big\{ x \in \Z^d: (A\tilde{f})_y
  \geq - M \text{ for all }y \in B_{2L_0}(x)\big\}, \qquad\mathcal{M}\coloneqq\mathcal{M}(\tilde{\varphi})
 %, \text{ with $M= (\log L_0)^2$}. 
%\frac{1}{2d}\,\sum_{\tilde y \in {\mathbb M}^d, \tilde y \sim x} \tilde{\varphi}_{\tilde y} \ge - M.
\end{equation}
(note that $B_{2L_0}(x)\subset \Z^d$ by definition and recall the extension $\tilde{\varphi}$ 
from \eqref{eq:phi_extend2}). Due to the decomposition $\tilde{\varphi}=\hat{\xi}+ \hat{\psi}$ 
from \eqref{eq:LB1_dim4}, this means that $\mathcal{M}= \{ x \in \Z^d:  \hat{\xi}_y  \geq - M, 
\, y \in B_{2L_0}(x)\}$. The condition used in the definition \eqref{eq:ubsuper1} provides us with a uniform insertion tolerance bound {\em on} the set $\mathcal M$ which will be used in the proof of Lemma~\ref{lem:truncated} below (see around the display \eqref{eq:insert_tol}) and 
which --- as already noted in the beginning of Section~\ref{sec:lower_sup} --- is not otherwise 
true.

We now introduce a key (good) event, involving various parameters, which roughly ensures the existence of  ambient clusters with desirable properties. In the sequel, $N,L$ and $L_0$ will denote three relevant length scales, $M>1$ controls the midpoint averages from below in \eqref{eq:ubsuper1}, and $b$ will count a certain number of `interfaces', to be introduced shortly, each containing $a$ (so-called) `contact points'. Throughout the rest of this section, we assume that  
\begin{equation}
\label{eq:ubsuper2}
\begin{split}
&\text{$N, L,L_0$ be integers with $N \geq L > 2L_0 \geq1$, $M >1$ and %$a:(0, \infty) \mapsto \N_{>0}$, 
$a,b \in \N_{>0}$,}%\\
%&\text{and a function $g:(0, \infty) \mapsto \N$ such that $a_L \geq 2|B_{3L_0}|$}
\end{split}
\end{equation}
and that $\Lambda_N \in \{ B_{2N} \setminus B_N,\,B_N \setminus B_{\sigma N}, \sigma \in (0,\frac13)\}$ is arbitrary, unless specified otherwise.
%(the latter is a condition on $L_0$ and $L$ for a given $g(\cdot)$). 
To avoid clumsy notation, we will keep dependence on the quantities appearing in \eqref{eq:ubsuper2} and $\Lambda_N$ implicit in the sequel, except for the ones that are subject to change in any given context. The intermediate scale $L$ will first appear in Definition~\ref{def:goodevent2} below. For $h' \in \R$, the event 
$G_{N} = G_{N}(L_0,a, b, h',M)$ is defined as
\begin{equation}\label{def:goodevent}
G_{N} =
\left\{
\begin{array}{c}
\text{$\exists$ disjoint connected sets $\mathscr C_i \subset (\Lambda_N \cap \{\varphi \ge h'\})$, for $1 \leq i \leq b$, such}\\
\text{that any $*$-path $\gamma$ crossing $\Lambda_N$ contains points $x_{i,j} =x_{i,j}(\gamma)$ satisfying}\\
\text{$B_{L_0}(x_{i,j}) \cap \mathscr C_i \cap\mathcal{M} \neq \emptyset$, for all $ 1 \leq i \leq b$ and $1 \leq j \leq a$.}
\end{array}
\right\}.
\end{equation}
(for concreteness, the indexing of the sets $\mathscr C_i $ can be induced by a fixed deterministic ordering of the vertices of $\Z^d$, with $\mathscr C_1 $ containing the smallest vertex in this ordering etc.)
Notice that the sets $\mathscr C_i$ in \eqref{def:goodevent} may very well be connected in $\Lambda_N \cap \{\varphi \ge h'\}$.
In the sequel, we will refer to the sets $\mathscr C_i $ as `interfaces' and to $x_{i,j}$, for $ 1 \leq i \leq b$ and $1 \leq j \leq a$, as the ensuing `contact points'. Thus, there will be a total of $a\cdot b$ contact points whenever $G_{N}$ occurs.

For $h' \in \R$, consider the sigma-algebra $\mathcal{F}_{h'} \coloneqq \sigma(\tilde{\varphi}_{\tilde y}, 1\{\varphi_x \ge h'\};\,  x \in \Z^d, \tilde y \in \mathbb M^d)$ and notice for later reference that $\mathcal{M}$, the event $G_{N}$ as well as the sets $\mathscr C_1, \ldots, \mathscr C_{b}$  are all $\mathcal{F}_{h'} $-measurable. Our interest in $G_{N}$ stems from the following lemma.  For later reference, note that the contribution of $\P[G_{N}^c] $ in \eqref{eq:truncated} and \eqref{eq:locauniq} will eventually turn out to be dominant in generating \eqref{eq:main_d=3}.

\begin{lemma}
	\label{lem:truncated} $(h < h' , \, \eqref{eq:ubsuper2}).$
There exists $c=c(h, h', L_0,M) > 0$ such that, with $G_{N} = G_{N}(L_0, a,b, h')$ as defined in \eqref{def:goodevent}, the following holds when $a \geq C(L_0, h,h',M)$ and $b \geq C(L_0)$. If $\Lambda_N = B_{N} \setminus B_{\sigma N}$, 
	\begin{align}
	&\P[\lr{}{\varphi \ge h}{0}{\partial B_{N}}, \nlr{}{\varphi \ge h}{0}{\infty}] \le \P[G_{N}^c] + \P[\nlr{}{\varphi \ge h'}{B_{\sigma N}}{\infty}] + e^{-cb a }. \label{eq:truncated}
	\end{align}
and if $\Lambda_N = B_{2N} \setminus B_{N}$,
\begin{align}
&\P[{\rm LocUniq}(N, h)^c] \le \P[G_{N}^c] + \P[\nlr{}{\varphi \ge h'}{B_{N}}{\partial B_{2N}}] + e^{-c b a }. 	\label{eq:locauniq}
	\end{align}
	\end{lemma}

Before delving into the proof, let us briefly give some intuition for \eqref{eq:truncated} and \eqref{eq:locauniq} and describe the salient features of the argument. We will focus on \eqref{eq:truncated} (the case of \eqref{eq:locauniq} is similar). In essence, \eqref{eq:truncated} asserts that, whenever the event $G_N$ occurs and $B_{\sigma N}$ connects to infinity at level $h'$ slightly below $h$, the event of interest on the left-hand side of \eqref{eq:truncated} comes at a uniform cost in each of the areas around the $a\cdot b$ contact points generated by the event $G_N$. Indeed, around each of those contact points, the cluster of $0$ in $\{ \varphi \geq h\}$ gets close to an interface, and the occurrence of $\{\lr{}{\varphi \ge h'}{B_{\sigma N}}{\infty} \}$ guarantees that all the interfaces are actually part of the infinite cluster at level $h'$. The cumulated cost of avoiding the interfaces near the contact points gives rise to the term $e^{-c b a }$. However, decoupling these avoidance events is hindered by the long-range correlations, which is dealt with in the proof by exploring the contact points one by one, and working with two sprinkled levels $h$ and $h'$. The controls given by \eqref{eq:ubsuper1}, which are required to hold in the area around contact points by definition of $G_N$, see \eqref{def:goodevent} (in particular, note that every contact point has a nearby point in $\mathcal{M}$) will be used to implement this decoupling.

\begin{proof}
We only give the proof of \eqref{eq:truncated}. The proof of \eqref{eq:locauniq} follows by straightforward modifications of the argument. We begin by introducing an auxiliary event $G'_{N}$ as follows.
With $h'' \coloneqq (h + h') / 2$ and for $\lambda > 0$, the event $G'_{N}(\lambda)$ occurs if all of the following hold:
\begin{itemize}
	\item[i)] the set  $\{\varphi \ge h''\}$ contains an infinite cluster $\mathscr C$ which intersects $B_{\sigma N}$.	
	\item [ii)] Letting $\mathscr S = B_N \cap \mathscr C \cap\mathcal{M}$, 
	there exists --- for any $*$-path $\gamma$ crossing $B_N \setminus 
	B_{\sigma N}$ --- a set $S \subset \gamma$ with $|S| \ge \lambda b a  $ such that $B_{L_0}(x) \cap \mathscr S \neq \emptyset$ for every $x \in S$.
\end{itemize}
The parameter $\lambda$, chosen below in \eqref{eq:explolambda}, will ensure a certain well-separatedness property.
Clearly the event $G'_{N}(\lambda)$ is measurable relative to $\mathcal F_{h''}$ (cf.~below \eqref{def:goodevent}). We will argue that for suitable $c, \lambda > 0$ each depending on $h, h', L_0$ only,
\begin{align}
\label{eq:truncated2}
&\P\big[\,\lr{}{\varphi \ge h}{0}{\partial B_{N}}, \nlr{}{\varphi \ge h}{0}{\infty} \,\big|\, \mathcal F_{h''}\big]1_{G'_{N}(\lambda)} \le e^{-cba}, \text{ and}\\
&\label{eq:G'bnd}
\P\big[\, G'_{N}(\lambda)^c \cap G_{N} \cap \{\lr{}{\varphi \ge h'}{B_{\sigma N}}{\infty}\}\big] \le e^{-cba},
\end{align}
from which \eqref{eq:truncated} readily follows.

The proofs of \eqref{eq:truncated2} and \eqref{eq:G'bnd} both hinge on the following representation of the conditional distribution of $\{\varphi \geq h \}$ under $\P[\, \cdot \, | \mathcal{F}_{h''}]$. In what follows, given $\mathbf{p}= (\mathbf{p}_x)_{x\in\Z^d}$ with $\mathbf{p}_x \in [0,1]$, let $P_{\mathbf{p}}$ denote the corresponding product measure on $\{0,1\}^{\Z^d}$, with canonical coordinates $Y=(Y_x)_{x\in \Z^d}$, so that $P_{\mathbf{p}}[Y_x=1]= \mathbf{p}_x$ for all $x \in \Z^d$. By means of \eqref{eq:LB1_dim4} and \eqref{eq:LB2_dim4}, one infers that for all $ \tilde h> h$,
\begin{equation}
\label{eq:insert_tol}
\begin{split}
&\text{the law of $(1\{\varphi_x\geq h\})_{x\in \Z^d}$ under $\P[\, \cdot \, | \mathcal{F}_{\tilde{h}}]$ is $P_{\mathbf{p}}$ with $\mathbf{p}_x = \P[\varphi_x \geq h | \mathcal{F}_{\tilde{h}}] $, $x\in \Z^d$,}\\
&\mathbf{p}_x=1 \text{ if } x \in \{\varphi \ge \tilde{h}\} \text{ and } 	\mathbf{p}_x\geq \Cl[c]{c:inserttol} = \Cr{c:inserttol}(h,\tilde{h}, M) \text{ if } x \in B_{L_0}(\mathcal{M}).
	\end{split}
\end{equation}
Indeed, regarding the last part, on the event $ \{\varphi_x < \tilde{h}\}$, one has by \eqref{eq:LB1_dim4} and \eqref{eq:LB2_dim4} that $ \P[\varphi_x \geq h | \mathcal{F}_{\tilde{h}}] = 1-\frac{\Phi(\sqrt{2}(h-\hat{\xi}_x))}{\Phi(\sqrt{2}(\tilde{h}-\hat{\xi}_x))} $, where $\Phi(\cdot)$ denotes the distribution function of a standard Gaussian random variable. In case $ x \in B_{L_0}(\mathcal{M})$, one obtains the desired lower bound using that $\hat{\xi}_x \geq - M$ by \eqref{eq:ubsuper1} and $\tilde h > h$ whilst noting that $\lim_{\xi \to \infty} \frac{\Phi(s-{\xi})}{\Phi(t-{\xi})}=0$ for any $s,t \in \R$ with $s< t$. 

With \eqref{eq:insert_tol} at hand, we proceed to show \eqref{eq:truncated2}, \eqref{eq:G'bnd}, starting with \eqref{eq:truncated2} and choosing
\begin{equation}
\label{eq:explolambda}
\text{$\lambda\coloneqq 1/(3|B_{3L_0}|)^2$ (so that $\lceil 2\sqrt\lambda b \rceil \leq b/|B_{3L_0}|$ when $b \geq C(L_0)$).}
%$ \,% \Cr{c:inserttol}' ( h, 
%h'', L_0) = \Cr{c:inserttol}^{|B_{L_0 + 1}|}$ (with $ \Cr{c:inserttol}$ as in \eqref{eq:insert_tol}).}
\end{equation} 
The choice of $\lambda$ in \eqref{eq:explolambda} will soon become clear.
Let $C(0)$ denote the cluster of $0$ in $B_L \cap \{\varphi \ge h\}$, under the measure $\P[\, \cdot \, | \mathcal{F}_{h'}]\stackrel{\text{law}}{=} P_{\mathbf{p}}$. We explore $C(0)$ vertex by vertex starting from $0$ in a canonical way, i.e.~checking at each step the state of the smallest (in a fixed deterministic ordering of the points in $\Z^d$) unexplored vertex in the exterior neighborhood of the currently explored piece of $C(0)$. We 
do so until the first time we discover a vertex $x_1\in C(0)$ (there may not be any) which 
is in the exterior neighborhood of some $B_{L_0}(y_1)$ with $y_1 \in \mathscr S$ (note that $\mathscr{S}$ as defined in ii) above is $\mathcal{F}_{h''}$-measurable). At this 
point, we explore the state of every vertex in $B_{L_0}(y_1)$. By definition, $\mathscr{S} \subset \mathscr{C}$ so $\mathscr{C}$ intersects $B_{L_0}(y_1)$. We stop the exploration if some vertex of $\mathscr{C} \cap B_{L_0}(y_1)$ lies in $C(0)$, which occurs for instance if all the vertices of $B_{L_0}(y_1)$ belong to $C(0)$. Otherwise we continue exploring $C(0)$ until we discover a vertex $x_2\in C(0)$ in the exterior neighborhood of $B_{L_0}(y_2)$ for some $y_2 \in \mathscr S \setminus \{y_1\}$ which was not visited by the exploration yet. As before, we then explore the state of every vertex in $B_{L_0}(y_2)$, stopping the exploration if $C(0)$ intersects $\mathscr{C}$ in that box and continuing otherwise. We proceed like this until either stopping or discovering the whole cluster $C(0)$. By construction, if the algorithm stops prior to discovering $C(0)$, the event $\{ \lr{}{\varphi \ge h}{0}{\infty}\}$ occurs.

Using the above algorithm, one deduces the following. Let $\tau \coloneqq |\{ y_1, 
y_2,\dots, \}|$ denote the number of points $y_i$ discovered until the algorithm stops (possibly~$\tau =0$). Then in view of property ii) above, with $\lambda' = \lambda / |B_{L_0+1}| $,
\begin{equation}
\label{eq:algo111}
 G'_{N}(\lambda) \cap \{\lr{}{\varphi \ge h}{0}{\partial B_{N}}, \nlr{}{\varphi \ge h}{0}{\infty}\} \subset \{ \tau \geq \lambda' ba\}
\end{equation}
(in particular, the right-hand side implies $\{\tau \geq 1\}$ whenever $a \geq C(h, h', L_0)$, 
as stipulated above \eqref{eq:truncated}). Moreover, by means of \eqref{eq:insert_tol} and 
\eqref{eq:explolambda}, one sees that under $\P[\, \cdot \, | \mathcal{F}_{h'}]$ and on the 
event $G'_{N}(\lambda)$, conditionally on  $\{\tau > n\}$ for some $0\leq n < \lambda' ba$, 
the event $\{\tau = n+1 \}$ occurs with probability at least $\Cr{c:inserttol}'$, where 
$\Cr{c:inserttol}' ( h, 
h'', L_0,M) \coloneqq \Cr{c:inserttol}(h,h'',M)^{|B_{L_0 + 1}|}$ (with $\Cr{c:inserttol}$ as in \eqref{eq:insert_tol}). Together with \eqref{eq:algo111}, this readily implies that the left-hand side of \eqref{eq:truncated2} is bounded by $(1-\Cr{c:inserttol}')^{\lambda' ba}$, as claimed.

We now turn to the proof of \eqref{eq:G'bnd}. Recall that $ G_{N} \cap \{\lr{}{\varphi \ge h'}{B_{\sigma N}}{\infty}\}$ is $\mathcal{F}_{h'}$-measurable. We will in fact show that 
\begin{equation}
\label{eq:G'bnd.1}
\P\big[\, G'_{N}(\lambda)^c  \, \big| \, \mathcal{F}_{h'}\big]1 \{  G_{N}, \lr{}{\varphi \ge h'}{B_{\sigma N}}{\infty}\} \le e^{-cba},
\end{equation}
which implies \eqref{eq:G'bnd}. Under $\P[\, \cdot \, | \mathcal{F}_{h'}]$ and on the event $ G_{N} \cap \{\lr{}{\varphi \ge h'}{B_{\sigma N}}{\infty}\}$, fix a path $\gamma_0$ in the infinite cluster of $\{ \varphi \geq h'\}$ crossing $\Lambda_N$. Notice in particular that $\gamma_0 \subset \mathscr{C}$, cf.~i) above. By definition of $ G_{N}$, see \eqref{def:goodevent}, and by suitable labeling (using for instance the above ordering of $\Z^d$), one finds disjoint connected sets $\mathscr C_i \subset (\Lambda_N \cap \{\varphi \ge h'\})$, for $1 \leq i \leq  \lceil 2 \sqrt{\lambda} b \rceil$ (with $\lambda$ as in \eqref{eq:explolambda}) and corresponding points $x_{i,j}(\gamma_0)$, $1 \leq j \leq  \lceil \sqrt{\lambda} a \rceil$, satisfying $B_{L_0}(x_{i,j}(\gamma_0)) \cap \mathscr C_i \cap\mathcal{M} \neq \emptyset$ and $|x_{i,j}(\gamma_0)- x_{i',j'}(\gamma_0)|_{\infty} > 2L_0$ for any $(i,j) \neq (i',j')$ and all $1 \leq  i,i'\leq \lceil 2\sqrt{\lambda}b \rceil$, $1 \leq  j,j'\leq \lceil \sqrt{\lambda}a \rceil$. Thus, letting 
\begin{equation}
\label{eq:G'bnd.2}
I \coloneqq \big\{ i : \lr{}{\varphi \ge h''}{x_{i,j}(\gamma_0)}{\mathscr C_i} \text{ for some } 1 \leq  j\leq \lceil \sqrt{\lambda}a \rceil  \big\}
\end{equation}
it follows that under $\P[\, \cdot \, | \mathcal{F}_{h'}]$ and on the event $\{\lr{}{\varphi \ge h'}{B_{\sigma N}}{\infty}, \,  G_{N}\}$,
\begin{equation}
\label{eq:G'bnd.3}
 G'_{N}(\lambda)^c \subset  \big\{ |I|< \sqrt{\lambda}b \big\};
 \end{equation}
indeed for an arbitrary $*$-path $\gamma$ crossing $\Lambda_N$, the occurrence of $ G_{N}$ guarantees the existence of vertices $x_{i,j}(\gamma)$ such that $B_{L_0}(x_{i,j}(\gamma)) \cap \mathscr C_i \cap\mathcal{M} \neq \emptyset$, for every $1 \leq  i\leq \lceil 2 \sqrt{\lambda}b \rceil$ and $1 \leq  j \leq \lceil \sqrt{\lambda}a \rceil$. Moreover, if $i \in I$, then in fact $\mathscr C_i \subset \mathscr{C}$ on account of \eqref{eq:G'bnd.2}. Hence, if $|I| \geq \sqrt{\lambda} b$, the set $S\coloneqq  \{ x_{i,j}(\gamma): i\in I, \, 1 \leq  j \leq \lceil \sqrt{\lambda}a \rceil \}$ satisfies the properties listed in ii) above and \eqref{eq:G'bnd.3} follows (the occurrence of i) is ensured by that of $\{\lr{}{\varphi \ge h'}{B_{\sigma N}}{\infty}\}$, which is conditioned on).

Finally, using \eqref{eq:insert_tol}, \eqref{eq:explolambda}, \eqref{eq:G'bnd.2} and the separation of the points $x_{i,j}(\gamma_0)$, $1 \leq  i\leq \lceil 2\sqrt{\lambda}b \rceil$, $1 \leq  j\leq \lceil \sqrt{\lambda}a \rceil$, one infers that $|I|$ stochastically dominates (under $\P[\, \cdot \, | \mathcal{F}_{h'}]$ and on $ G_{N}$) a binomial random variable with $\lceil 2\sqrt{\lambda}b \rceil$ trials and success probability $1- (1-\Cl[c]{c:inserttolnew})^{\lceil \sqrt{\lambda}a \rceil}$, where $\Cr{c:inserttolnew}: =\Cr{c:inserttol}' ( h, 
h', L_0,M)$. Thus, 
\begin{equation*}
\P\big[|I|< \sqrt{\lambda}b  \big| \mathcal{F}_{h'}\big]1_{{G_{N}}} \leq 2^{\lceil 2\sqrt{\lambda}b \rceil} (1-\Cr{c:inserttolnew})^{\lceil \sqrt{\lambda}a \rceil \sqrt{\lambda} b} \leq e^{-cb a}
\end{equation*}
for suitable $c=c(L_0,h,h',M)$,
as soon as $\sqrt{\lambda} a \geq C(L_0, h, h', M)$, whence \eqref{eq:G'bnd.1} readily follows from \eqref{eq:G'bnd.3}. This completes the proof of \eqref{eq:G'bnd} and with it that of Lemma~\ref{lem:truncated}.
\end{proof}

The upper bounds \eqref{eq:main_ubsupercrit_d=3} and \eqref{eq:main_ubsupercrit_d>4}, as well as the corresponding  bounds for ${\rm LocUniq}(h, N)^c$, will eventually follow from Lemma \ref{lem:truncated}. As mentioned above that lemma, all estimates will ultimately turn out to be carried by the upper bound for $\P[G_{N}^c]$. %In view of \eqref{eq:truncated}, \eqref{eq:locauniq} and along with the bounds on disconnection probabilities from~\cite{Sz15}, this means that we aim at choosing $a_L$ growing faster than $L / \log L$ when $d = 3$ and linearly in $L$ when $d \ge 4$. 
Our next step is thus to derive suitable upper bounds on $\P[G_{N}^c]$, which we achieve by a renormalization scheme operating as follows. In a single step, the scheme goes from a base scale $L$ to a larger scale $N$, cf.~\eqref{eq:ubsuper2}. Whereas $b$ as appearing in \eqref{def:goodevent} remains effectively fixed through the iteration (as will turn out $b$ will grow roughly like $N/u(KL)$, cf.~\eqref{def:admissible1}), the scheme is designed to \textit{simultaneously} improve on both the growth of $a\coloneqq a_L$ in \eqref{def:goodevent} and the strength of the estimate on $\P[G_{N}^c]$ in each step. Roughly speaking, this will boost $a_L$ from being of order $1$ to growing faster than $L /( \log L)^C$ when $d = 3$ and linearly in $L$ when $d \ge 4$, whence the error terms in \eqref{eq:truncated}, \eqref{eq:locauniq} become sufficiently small.

 The starting point of the argument is a certain good event ${\mathcal G}_z$, which we now introduce, and for which we will later supply a suitable a-priori estimate, see Lemma~\ref{lem:loc_amb_cluster1}. In the sequel, $\tilde\chi = (\tilde\chi^z)_{z \in \mathbb{L}}$, where $\tilde\chi^z$ refers to either $ \tilde\varphi$ or $\tilde \psi^z$, see \eqref{eq:xi_tilde_z}. We write $\chi^z$ for the restriction of $\tilde\chi^z$ to $\Z^d$. Recall \eqref{eq:LL}-\eqref{eq:decomp_z}, \eqref{eq:ubsuper1} and \eqref{eq:ubsuper2} regarding notation.
\begin{defn}(Good event). \label{def:goodevent2}
For $h_1  \le h_2 \leq h_3  \in \R$ and $z \in \mathbb{L}$, %for some $L \geq1$ and $1 \leq L_0 \leq L$. Also for every $w \in \mathbb L$, let $\chi^w$ be a random field defined on $\Z^d \cup 
%\mathbb M^d$ restricted to $U_w$. Finally let $g:(0, \infty) \mapsto \N_{>0}$ be a function and 
%$M > 0$. %, and define $h_1 \coloneqq h + \frac{(\bar h - h)}{4}$ and $h_2 \coloneqq h + \frac{3(\bar h - h)}{4}$. 
the event  ${\mathcal G}_z (\tilde\chi) = {\mathcal G}_z(\tilde\chi, L_0, L, a, h_1, h_2, h_3,M)$ occurs if all of the following hold:
\begin{align}
	&\label{def:goodevent2.1}
	\begin{array}{l}
	\text{$C_z$ is connected to $\partial D_z$ in $\{\chi^z \ge h_3\}$ %with the corresponding cluster being $\mathscr C_z$ (say), 
	and for all $z' \in \mathbb L$ such that $|z - z'|_{\infty} \le L$,}\\
	\text{all clusters of $\{\chi^{z'} \ge h_2 \}$ crossing $\tilde D_z \setminus \tilde C_z$ are connected inside $D_z \cap \{ \chi^z \geq h_1 \}$.} \end{array}\\[0.3em]
	&\label{def:goodevent2.2}
	\begin{array}{l}
	\text{letting $\mathscr S_z = \tilde D_z \cap \mathscr C_z \cap\mathcal{M}( \tilde{\chi}^z)$ where $\mathscr C_z$ denotes the cluster of $C_z$}\\
	\text{in $\{\chi^z \ge h_1\}$, for any $*$-path $\gamma$ crossing $\tilde D_z \setminus 
	\tilde C_z$ there exists a set}\\
	\text{$S \subset \gamma$ with $|S| \ge a$ such that $B_{L_0}(x) \cap \mathscr S_z \neq \emptyset$ for every $x \in S$.}
	\end{array}
	\end{align}
	%
%	\item[ii)] There exists $\mathscr C_z' \subset B_z \cap \mathscr C_z$ with the following two properties,
%	\begin{itemize}
%	\item [a)] For every $x \in \tilde D_z \setminus C_z$ such that $d_\infty(x, \mathscr C_z) \le L_0$,
%	\begin{equation*}
%		%\label{eq:conditional_fin_energy}
%		\frac{1}{2d}\,\sum_{\tilde y \in {\mathbb M}^d, \tilde y \sim x} \chi(\tilde y) \ge - M, \text{ and}
%	\end{equation*}
%\item [b)] for any $*$-connected path $\gamma$ crossing $\tilde D_z \setminus C_z$, there exists a 
%set $S \subset {L_0}\mathbb{Z}^d$ whose $L_0$-neighborhood is contained in $\tilde D_z$ 
%satisfying $|S| \ge a_L$ and that $B_x(L_0)$ intersects $\mathscr C_{z}'$ for every $x \in 
%S$.
%\end{itemize}
\end{defn}
%In the remainder of this section %, except in lemmas and propositions, 
%we will suppress all the parameters involved in an event save the ones that are subject to change in any given context.
%Like in the case of $G$ and $\mathscr C'$, we take $\mathscr C_z'$ to be the maximal set satisfying properties~ii)-a) and b).

%We now set up the bootstrapping scheme that will lead to a suitable control of $\P[G_{N}^c]$. We distinguish two cases, reflected by the index $i\in \{1,2\}$ below. In each intermediate step of the scheme (corresponding to $i=1$ in the choices below), the local uniqueness component of the event $\mathcal{G}_z$, see \eqref{def:goodevent2.1}, needs to be carried from one scale to the next, which requires a two-sided control on the harmonic field. The last step (corresponding to $i=2$)  generates the event $G_{N}$, thus constructing the sets $\mathscr{C}_i$ required by \eqref{def:goodevent}, from a local uniqueness property at smaller scales. This is a pure existence result, which only relies on a one-sided control on the harmonic field, and ultimately yields the precise bound~\eqref{eq:main_ubsupercrit_d=3}.
Henceforth we routinely suppress the dependence of $\mathcal{G}_z$ on parameters which stay fixed.
Note that ${\mathcal G}_z$ indeed depends on $\tilde{\varphi}$, the extension of $\varphi$ to $\tilde{\Z}^d$, through $\mathcal{M}( \tilde{\chi}^z)$ in \eqref{def:goodevent2.2}, see \eqref{eq:ubsuper1}. We now introduce, abbreviating $\tilde{\psi}=(\tilde{\psi}^z)_{z \in \mathbb L}$ and 
for $\varepsilon > 0$,
\begin{align}
&\label{eq:tildeGdef}
\mathcal{H}_z^1 \coloneqq \mathcal G_{z}(\tilde{\psi}) \cap \big\{\sup_{D_z}|\tilde{\xi}^z| \le \varepsilon \big\},\\
&\label{eq:starGdef}
\mathcal{H}_z^2 \coloneqq \mathcal G_{z}(\tilde{\psi}) \cap \big\{\inf_{D_z}\tilde{\xi}^z \ge -\varepsilon \big\},
\end{align}
each inheriting the dependence on parameters from $\mathcal{G}_z$ and where, as in \eqref{eq:BTIS_Szn}, $\sup_{ D_z}$/$\inf_{ D_z}$ refer to the suprema/infima over all points in $D_z \cap \tilde{\mathbb{Z}}^d$. 

The events introduced in Definition~\ref{def:goodevent2} will be the fundamental building blocks of most constructions to follow. Lemma~\ref{lem:loc_amb_cluster1} will soon witness these events to be typical in the supercritical regime when $a=1$. The goal will then roughly consist of improving \textit{both} their likelihood and the value of $a$ simultaneously, a feature which will be achieved by means of the bootstrapping scheme in Proposition~\ref{lem:inclusion_fxn_goodevent} below.

The intuition for the above definition is as follows. Typically, neighboring points $z,z' \in 
\mathbb{L}$ for which $\mathcal{G}_z$ and $\mathcal{G}_{z'}$ occur roughly give rise to 
corresponding macroscopic cluster at scale $L$ which must communicate (i.e.~be glued together) 
due to \eqref{def:goodevent2.1}. Moreover, by means of \eqref{def:goodevent2.2} the resulting 
joint cluster has a desirable number of contact point(s) (parametrized by $S$) for any 
macroscopic annulus crossing $\gamma$ in either of the constituent boxes. Continuing in this 
way by stacking neighboring (good) boxes in which these events occur thus generates large 
clusters with these properties. Here, we are deliberately passing over various technical 
details, including the necessity to work with both the full and localized fields, whence 
$\tilde{\chi}$, and the fact that our construction must operate simultaneaously at various 
sprinkled levels $h_1$--$h_3$, essentially to permit going back and forth between the two 
fields, whenever the relevant harmonic average $\tilde{\xi}^{\cdot}$ is under control, as 
ensured by \eqref{eq:tildeGdef} or \eqref{eq:starGdef}.

We now make these heuristics precise. Let $\varepsilon_0 \coloneqq  \frac1{3}(1 \wedge (h_2-h_1) \wedge (h_3-h_2))$. 
We start by considering two points. The key features of \eqref{def:goodevent2.1}-\eqref{eq:starGdef} are summarized in the following  

\begin{lemma}[Gluing; $h_1  < h_2 < h_3, \,  \varepsilon \in (0, \varepsilon_0),  \eqref{eq:ubsuper2}$] \label{label:gluing}
$\quad$

\medskip
\noindent
For $ z \in \mathbb{L}$, with $(\tilde\chi_1, \tilde\chi_2) \in \{ (\tilde\varphi, \tilde\psi ), (\tilde\psi,\tilde\varphi )\}$ and
$\mathcal{G}^1 \coloneqq \mathcal G_z(\tilde\chi_1, h_1, h_2, h_3,M)$, $\mathcal{G}^2 \coloneqq \mathcal{G}_z(\tilde\chi_2, h_1 - \varepsilon, h_2 + \varepsilon, h_3 - \varepsilon, M+\varepsilon)$, one has
\begin{equation}
\label{eq:locuniqdecomp}
\mathcal{G}^1 \cap \{  \sup_{z'} \sup_{ D_{z'}}|\tilde{\xi}_{\cdot}^{z'}| \le \varepsilon \} \subset \mathcal{G}^2,
\end{equation}
with $z'$ ranging over all points in $\mathbb{L}$ with $|z' - z|_\infty \le L$. In particular,
\begin{equation}
%\label{eq:chi1good}
\label{eq:ellinfapproxsupcrit}
\P[\mathcal{G}^2] \ge 1 - e^{-f(L)} \text{ whenever }	\P[\mathcal{G}^1] \ge 1 - e^{-2f(L)} \text{ with } 4\log(d) \le 4f(L) \le \Cr{c:cap}(4 \varepsilon) L
\end{equation}
where $\Cr{c:cap}$ is as in \eqref{eq:inputbnd}. Moreover, for $z, z' \in \mathbb{L}$ with 
$|z' - z|_\infty = L$, if $\mathcal{H}_z^i \cap \mathcal{H}_{z'}^i$ occurs for some $i \in \{ 
1,2\}$,
\begin{equation}
\begin{split}
&\text{all clusters of $\{\Xi_i^y \ge h_2 + \varepsilon_i\}$ crossing $\tilde{D}_y \setminus \tilde C_y$, for $y \in \{ z,z'\},$ belong}\\
&\text{to a single connected subset $\mathscr{C}_{z,z'}$ of $ \{\varphi \ge h_1 - \varepsilon\} \cap (  D_z\cup   D_{z'})$, which}\\
&\text{crosses both ${D}_z \setminus  C_z$ and ${D}_{z'} \setminus  C_{z'}$,}
\end{split}
\label{property:gluing_neighbor}
\end{equation}
where $\Xi_1^y= \varphi$, $\varepsilon_1=\varepsilon$ and $\Xi_2^y= \psi^y$, $\varepsilon_2=0$, and 
\begin{equation}
\label{property:gluing_neighbor2}
\begin{split}
&\text{for any $*$-path $\gamma$ crossing $\tilde{D}_y \setminus \tilde C_y$, for $y \in \{ z,z'\},$ there exists a set $S \subset \gamma$}\\
&\text{with $|S| \ge a$ such that $B_{L_0}(x) \cap  \tilde D_y \cap \mathscr{C}_{z,z'} \cap\mathcal{M}_{M+\varepsilon}(\tilde\varphi)\neq \emptyset$, for all $x \in S$.}
\end{split}
\end{equation}
\end{lemma}
\begin{proof} We write $\mathcal{M}= \mathcal{M}_M$ throughout the proof to make the dependence on $M$ explicit, cf.~\eqref{eq:ubsuper1}. The two-sided control on the harmonic average in \eqref{eq:locuniqdecomp} and \eqref{eq:decomp_z} imply that for all $h \in \R$ and $i \in \{1,2\}$, any connected subset of $\{\chi_i  \geq h \} \cap {D}_{z'}$ is contained in a connected subset of $\{\chi_{3-i}  \geq h -\varepsilon \} \cap D_{z'}$, for any $z'\in \mathbb{L}$ with $|z' - z|_\infty \le L$. 
Applying this repeatedly in the context of \eqref{def:goodevent2.1} readily yields the corresponding property for the event on the right-hand side of \eqref{eq:locuniqdecomp}. Regarding \eqref{def:goodevent2.2}, recalling the local averaging operator $A$ from above \eqref{eq:ubsuper1}, on $B_{2L_0}(\tilde{D}_z) (\subset \Z^d)$ one has that $A\tilde{\chi}_2= A(\tilde{\chi}_1 \pm \tilde{\xi}^z) \geq A\tilde{\chi}_1 - |A\tilde{\xi}^z| \geq A\tilde{\chi}_1 -\varepsilon$ by linearity and using that $|A\tilde{\xi}^z|\leq A|\tilde{\xi}^z| \leq \varepsilon$ on $B_{2L_0}(\tilde{D}_z)$, whence $\mathcal{M}_M( \tilde{\chi}_1^z) \cap \tilde{D}_z$ is contained in $\mathcal{M}_{M+\varepsilon}( \tilde{\chi}_2^z) \cap \tilde{D}_z$.

The implication \eqref{eq:ellinfapproxsupcrit} follows readily from \eqref{eq:locuniqdecomp} upon applying e.g.~\eqref{eq:BTIS_Szn} for a single box (along with the bound on the capacity of a box as given by \eqref{eq:capball}) to control the tails of the harmonic average.

To obtain \eqref{property:gluing_neighbor}, first observe that regardless of the choice of $i\in \{1,2\}$, by \eqref{def:goodevent2.1}, \eqref{eq:tildeGdef}, \eqref{eq:starGdef} and the fact that any path crossing $D_{z'}\setminus C_{z'}$ also crosses $\tilde D_z \setminus \tilde C_z$ (see \eqref{eq:boxes} and \eqref{eq:C_boxes}), the following holds: all clusters of $\{\psi^y \ge h_2 \}$ crossing $\tilde{D}_y \setminus \tilde C_y$, for $y \in \{ z,z'\},$ belong to a single connected component $\mathscr{C}_{z,z'}$ of $ (\{\psi^z \ge h_1 \} \cap D_z) \cup( \{\psi^{z'} \ge h_1 \} \cap  D_{z'})$, which crosses both ${D}_z \setminus  C_z$ and ${D}_{z'} \setminus  C_{z'}$ (the last part is due to the first item in \eqref{def:goodevent2.1}, which guarantees the existence of such a crossing for $\psi^y$ above level $h_3 > h_2$). The control on the lower tail of $\xi^z$, $\xi^{z'}$ present in \eqref{eq:starGdef} (and \eqref{eq:tildeGdef}) implies that $\mathscr{C}_{z,z'}$ belongs to a connected component of $ \{\varphi \ge h_1 - \varepsilon\}$, thus yielding \eqref{property:gluing_neighbor} in case $i=2$. The upper bound on $\xi^y$, $y \in \{ z,z'\}$, in \eqref{eq:tildeGdef} further implies that any cluster in $ \{\varphi \ge h_2 + \varepsilon\}$ crossing $\tilde{D}_y \setminus \tilde C_y$ is part of a crossing cluster of $ \{\psi^y \ge h_2 \}$, whence \eqref{property:gluing_neighbor} follows for $i=1$.

Finally, the property \eqref{property:gluing_neighbor2} is inherited from \eqref{def:goodevent2.2} when $\mathcal{H}_z^i \cap \mathcal{H}_{z'}^i$ occurs, regardless of $i \in \{1,2\}$. For, by construction, the cluster of $C_y$ in $\mathscr{C}_{z,z'}$ contains the cluster of $C_y$ in $ \{\psi^y \ge h_1 \}$ for $y \in \{z,z'\}$, and $(\mathcal{M}_M( \tilde{\psi}^y) \cap \tilde{D}_y) \subset (\mathcal{M}_{M+\varepsilon}( \tilde\varphi) \cap \tilde{D}_y)$, since $A \tilde{\varphi}= A(\tilde{\psi}^y + \tilde{\xi}^y) \geq A\tilde{\psi}^y -\varepsilon$ in $B_{2L_0}(\tilde{D}_y)$ if the event $ \mathcal{H}_{y}^i$ occurs (this only uses the lower bound on $\tilde{\xi}^y$).
\end{proof}

We now set up the bootstrapping scheme that will lead to a suitable control of $\P[G_{N}^c]$. The index $i \in \{ 1,2 \}$ in \eqref{eq:tildeGdef}, \eqref{eq:starGdef} and below reflects the fact that intermediate steps (corresponding to $i=1$) and the final step (corresponding to $i=2$) of the argument need to be dealt with in distinct ways. %For $h_1  \le h_2 \leq h_3 \in \R$ (and parameters as in \eqref{eq:ubsuper2}), we set 
%\begin{equation}
%\label{eq:bs1}
%G^{\, i}_N  \coloneqq 
%\begin{cases}
%\mathcal{G}_0 (\tilde{\varphi}, L_0, N, a, h_1, h_2, h_3) & \text{ if } i=1\\
%G_{N}(L_0, a, b,  h_1), & \text{ if } i=2.
%\end{cases}
%\end{equation}
%It will be notationally convenient to refer to the event $G^{\, i}_N = G^{\, i}_N (L, a,b, h_1, h_2, h_3)$ (the remaining parameters staying fixed), while keeping in mind that for a specific choice of $i \in \{1,2\}$, the event $G^{\, i}_N $ only depends on a subset of these parameters. Note that $G^{\, 2}_N$ implicitly depends on $\Lambda_N$ as appearing above \eqref{def:goodevent}. 
The coarse-graining scheme developed in Section~\ref{sec:coarsegrain} now enters the picture.
Referring to Proposition~\ref{prop:coarse_paths}, for integers $K \ge 100$, $N \ge 10 K L$, we let $\mathcal A^1 = \mathcal A_{N, L}^{K}(\tilde D_{0,N} \setminus \tilde C_{0,N})$ and $\mathcal A^2 = \mathcal A_{N, L}^K(\Lambda_N)$ (with $\Lambda_N \in \{B_N \setminus B_{\sigma N}, B_{2N} \setminus B_{N}\}$, $\sigma \in (0, 1/3)$, as below \eqref{eq:ubsuper2}), %and either $\sigma \in (0, 1/d)$ and  $\rho_2 \in (0, 1 - d\sigma)$ or $\sigma = 2$ and $\rho_2 \in (0, 1)$. 
and define, for $\rho \in (0, 1)$,
\begin{equation}
\label{def:good_box_event2}
H_{N,L}^{\,i}  %= \mathscr G^\star(\Lambda_N, L_0, L, v, h_1, h_2, M, \varepsilon, \rho_2)  
\coloneqq  \bigcap_{\mathcal C \in \mathcal{A}^i} 
	\,\, \bigcup_{\substack{\tilde{\mathcal{C}}\subset \mathcal{C} \\ |\tilde{\mathcal{C}}|\geq 
			\rho|\mathcal{C}|}}\,\,\bigcap_{z \in \tilde{\mathcal C}} \, \mathcal{H}_z^i.
\end{equation}
The event $H_{N,L}^{\,i}  $ will typically be `good', i.e.~likely, and $\rho$ small, i.e.~close to $0$.
We elaborate a bit more on the central role of the event $H_{N,L}^{\,i}$ in Remark~\ref{R:onestepsupercrit},2) below.
The next result is at the heart of our argument. It shows that $H_{N,L}^{\,i}$ typically reproduces the event from Definition~\ref{def:goodevent2} (implicit in $\mathcal{H}_z^i$) at a higher scale $N$ and with an {\em improved} choice of $a$ (for $i=1$), as well as the target event $G_{N}$ from \eqref{def:goodevent} (for $i=2$).

\begin{proposition}[Bootstrap; $h_1  < h_2 < h_3, \,   \eqref{eq:ubsuper2}, \, \rho \in (0,1), K \ge 100, i=1,2$] 
	\label{lem:inclusion_fxn_goodevent} 
With $H_{N,L}^{\,i} = H_{N,L}^{\,i} ( a,  h_1, h_2, h_3, \varepsilon, M, \rho)$, whenever $\tfrac{\rho N}{u(KL)} \geq C$, one has the inclusion
%For all possible values of the underlying parameters, the following inclusion holds
%	\begin{equation*}
%		\mathscr G_w(N, L_0, L, v, h_1, h_2, M, \varepsilon) \cap \{\lr{}{\varphi \ge h_2 + \varepsilon}{B_w}{D_w}\} \subset \mathcal G_{w}(\tilde \varphi, L_0, N, \tfrac{N}{800L}g, h_1 - \varepsilon, h_2 + \varepsilon , M + \varepsilon).
%		\end{equation*}
			\begin{equation}
			\label{eq:inclusion_goodevent.1}
		\big( H_{N,L}^{\,i} \cap \Omega_N^i \big)\subset 
		\begin{cases}
\mathcal{G}_0 (\tilde{\varphi}, N, a', h_1 - \varepsilon, h_2 + \varepsilon , h_3- \varepsilon, M+ \varepsilon), & \text{ if } i=1, \text{ for $\varepsilon \in (0, \varepsilon_0)$,}\\
G_{N}(a, b,  h_1 -\varepsilon, M+ \varepsilon),& \text{ if } i=2, \text{ for all $\varepsilon >0$},
\end{cases}
%		 G_N^{\,i} (L_0, a_i, b, h_1 - \varepsilon, h_2 + \varepsilon , h_3- \varepsilon),
		\end{equation}
		where   $ \Omega_N^1= \{\lr{}{\varphi \ge h_3 - \varepsilon}{C_{0,N}}{\partial D_{0,N}}\}$, $ \Omega_N^2=\R^{\tilde{\Z}^d}$ and 
\begin{equation}
			\label{eq:inclusion_goodevent.2}
b\coloneqq   \lfloor \tfrac{\Cr{c:nLB}}{10} \,\tfrac{\rho N}{u(KL)} \rfloor, \quad a' \coloneqq % \begin{cases}%\Cl[c]{c:improvedg}
 b a %, & \text{ if }i=1\\
%a(N) & \text{ if }i=2
%\end{cases}
%.% \text{ with } \Cr{c:improvedg}= \begin{cases} 
% \tfrac{\Cr{c:nLB}}{7}(1 - \sigma), & \text{ if $\Lambda_N = B_N \setminus B_{\sigma N}$ and $i=2$,}\\
% \tfrac{\Cr{c:nLB}}{7}, & \text{otherwise}.\end{cases}
\end{equation}
(here $\Cr{c:nLB}$ is as in \eqref{def:admissible1}).
\end{proposition}
%We need, as an input for the proof of this lemma, a combinatorial result about the existence of  ``dual'' surfaces separating two sets. %In the sequel we say $S_1 \subset\subset \Z^d$ is {\em surrounded} by $S_2 \subset \subset \Z^d$ if $S_1$ is contained in some finite connected component of $\Z^d \setminus S_2$. For convenience, we will denote this relation (which is in fact a partial order) by $S_1 \preceq S_2$.
%\begin{lemma}[Existence of blocking layers]
%	\label{lem:blocking_layers}
%Let $\L \subset \Z^d$ be a box and $S \subset \L$. Also let $T \subset \L \setminus S$ be such that any $*$-connected path between $S$ and $\partial \L$ intersects $T$ in at least $k 
%\ge 1$ points. Then there exist $*$-connected subsets $O_1, \ldots, O_{k}$ of $T$ such that $S \preceq O_1 \preceq \ldots \preceq O_{k}$.
%\end{lemma}
%Assuming this lemma let us now return to:

\begin{remark}\label{R:onestepsupercrit}
\begin{enumerate}[label={\arabic*)}]
\item In \eqref{eq:inclusion_goodevent.1}, one also has the inclusion $H_{N,L}^{\,1} \subset  G_{N}( a, b,  h_1 -\varepsilon) $ since $H_{N,L}^{\,1} \subset H_{N,L}^{\,2}$. The weaker condition on the harmonic field entering the definition of $\mathcal{H}_z^2$, cf.~\eqref{eq:tildeGdef} and \eqref{eq:starGdef}, will play a role in obtaining the sharp bound \eqref{eq:main_ubsupercrit_d=3} for $d =3$. The asserted result \eqref{eq:main_ubsupercrit_d>4} for $d\geq 4$ could be obtained by means of 
$\mathcal{H}_z^1$ (and $H_{N,L}^{\,1}$) alone. 
\item (The role of $\rho$). In view of \eqref{eq:truncated}, \eqref{eq:locauniq} and 
\eqref{eq:inclusion_goodevent.1}, our goal in bounding $\P[G_N^c]$ becomes to control the probability of $(H_{N,L}^i)^c$, which by definition, see \eqref{def:good_box_event2}, entails the existence of 
at least one collection $\mathcal C \in \mathcal{A}^i$ containing a large fraction $(1-\rho)$ 
of (bad) points $z$ for which $ \mathcal{H}_z^i$ does not occur. The sacrifice of a small 
fraction $\rho>0$ of (good) points inherent in \eqref{def:good_box_event2} is utilized in the 
proof below in order to create $b$ interfaces (growing linearly with $\rho$) with certain good 
properties, implied locally by the occurrence of $\mathcal{H}_z^i$. This eventually leads to 
the desired improvement \eqref{eq:inclusion_goodevent.2} over $a$. 
\end{enumerate}
\end{remark}

Let us outline the main lines of the proof of Proposition~\ref{lem:inclusion_fxn_goodevent} 
for, say, $i=1$, which is the more complicated case, because this is where we witness the improvement $a \to a'$ in the number of contact points. This will also shed some light on 
\eqref{eq:inclusion_goodevent.1} and \eqref{eq:inclusion_goodevent.2}. A key feature of the 
event $H_{N,L}^{\,1}$ defined in \eqref{def:good_box_event2} is the following: in combination 
with Lemma~\ref{lem:blocking_layers} applied on the coarse-grained lattice $\mathbb{L}$, the 
definition of $H_{N,L}^{\,1}$ will be seen to imply the existence of $b$ interfaces made of 
$L$-boxes, with $b$ as in \eqref{eq:inclusion_goodevent.2} that must \textit{all} be traversed 
by any macroscopic `radial' path $\gamma$, and having roughly the following two properties. 
First, each interface has an associated cluster which is `present in all of its boxes'. 
Second, whenever the path $\gamma$ traverses one interface, it picks up $a$ contact points 
with the corresponding cluster. Now, whenever the event $ \Omega_N^1$ occurs, all of these 
clusters actually communicate (much in spirit as in the proof of Lemma~\ref{lem:truncated}; 
see also the discussion following \eqref{eq:locauniq}). Actually this happens at sprinkled 
levels but let us forego these details. All in all, this mechanism generates a macroscopic 
cluster at scale $N$, now with $a\cdot b = a'$ contact points for any crossing path, whence 
\eqref{eq:inclusion_goodevent.2}.

\begin{proof}%[Proof of Lemma~\ref{lem:inclusion_fxn_goodevent}]
 %We first consider the case $i=1$ and indicate the small modifications needed to handle case $i=2$ at the end. 
For any $A \subset \Z^d$, we use $\mathbb L_A$ to denote the set of all $z \in 
\mathbb L= \mathbb L(L)$ such that $A \cap C_{z,L} \neq \emptyset$ (see \eqref{eq:LL} and \eqref{eq:boxes} for notation).
In what follows, let $U = \mathbb{L}_{\tilde C_{0,N}}$ if $i=1$ and $U = \mathbb{L}_{B_{\sigma N}}$ (resp.~$U = \mathbb{L}_{B_N}$) if $i=2$ and $\Lambda_N = B_N \setminus B_{\sigma N}$ (resp.~$\Lambda_N = B_{2N} \setminus B_{ N}$). In a similar vein, let $V =\mathbb{L}_{\tilde{D}_{0,N}}$ if $i=1$ and $V =\mathbb{L}_{B_N}$ or $\mathbb{L}_{B_{2N}}$ if $i=2$.

 Now let $ \Sigma \subset \mathbb L$ denote the collection of all points $z \in \mathbb L \setminus U$ such that 
$\mathcal{H}_z^i$ occurs and $\tilde D_{z, L} \subset \tilde D_{0, N}$ if $i=1$, resp.~$\tilde D_{z, L} \subset B_{N}$ or $B_{2N}$ if $i=2$. As we now explain, on $H_{N,L}^{\,i}$, 
\begin{equation}
\label{eq:L2.1applied}
\text{$\Sigma \subset \mathbb L$ satisfies the hypotheses of Lemma~\ref{lem:blocking_layers} with $U, V$ as above and $k= \lfloor \tfrac{\Cr{c:nLB}\rho N}{u(KL)} \rfloor -1$}
\end{equation}
(it is understood that Lemma~\ref{lem:blocking_layers} is applied here to the lattice $\mathbb L= \mathbb L(L)$ rather than $\Z^d$).
To see \eqref{eq:L2.1applied}, consider any $*$-path $\gamma_{\mathbb{L}}$ in $\mathbb{L}$ crossing $V\setminus U$. By suitably interpolating between successive vertices of $\gamma_{\mathbb{L}}$, one creates a $*$-path $\gamma$ in $\Z^d$ such that $\gamma |_{\mathbb{L}}= \gamma_{\mathbb{L}}$. Let $\mathcal{C} \in \mathcal{A}_i$ be the admissible collection corresponding to $\gamma$, i.e.~such that \eqref{def:admissible2} holds. Then by definition, see \eqref{def:good_box_event2}, if $H_{N,L}^{\,i}$ occurs, there exists $\tilde{\mathcal{C}}\subset \mathcal{C}$ such that 
$\mathcal{H}_z^i$ occurs for all $z \in\tilde{\mathcal{C}} $, hence $\tilde{\mathcal{C}} \subset \Sigma$, and \eqref{eq:L2.1applied} follows since $|\tilde{\mathcal{C}}| \geq k$ on account of \eqref{def:good_box_event2} and \eqref{def:admissible1}.

With \eqref{eq:L2.1applied} in force, applying Lemma~\ref{lem:blocking_layers} we deduce the existence of disjoint $*$-connected subsets $O_1 \preceq \ldots \preceq O_k$ (each part of $\mathbb{L}$) of $\Sigma$ all of which surround $U$. By definition of $ \preceq$, it follows that $(O_i + D_{0, L}) \cap (O_j 
+ D_{0, L}) = \emptyset$ as soon as $|i - j| \ge 7$. Consequently we can extract from 
$\{O_i: 1 \le i \le k\}$ a subcollection $\{O_i': 1 \le i \le k'\}$ with $k' \geq k/8$ such 
that $(O_i' + D_{0, L})$ are pairwise disjoint subsets of $A$ with $V = L_A$. Now for each $j \in \{1, \ldots, k'\}$, by $*$-connectedness of $O_j'$, the fact that $\mathcal{H}_z^i$ occurs for every $z \in O_j' (\subset \Sigma)$ and using property~\eqref{property:gluing_neighbor}, one finds 
that the connected sets $\mathscr{C}_{z,z'}$, for $z, z' \in O_j'$ with 
$|z-z'|_{\infty}=L$, are contained in a single connected subset $\mathscr C_j$ of $\{\varphi \ge 
h_1 - \varepsilon\} \cap (O_j' + D_{0, L})$. The sets $\mathscr C_j$, $1\leq j \leq k'$, are 
disjoint by construction. Moreover by \eqref{property:gluing_neighbor2},
\begin{equation}
\label{property:gluing_neighbor3}
\begin{split}
&\text{for any $*$-path $\gamma$ crossing $\tilde{D}_{z, L} \setminus \tilde C_{z, L}$, for some $z \in O_j'$, there exists a set $S \subset \gamma$}\\
&\text{with $|S| \ge a$ such that $B_{L_0}(x) \cap  \tilde D_{z, L} \cap \mathscr{C}_j\cap\mathcal{M}_{M+ \varepsilon}\neq \emptyset$, for all $x \in S$.}
\end{split}
\end{equation}
In addition, if $i=1$, \eqref{property:gluing_neighbor} yields that
\begin{equation}
\begin{split}
&\text{all conn.~subsets of $\{\varphi \ge h_2 + \varepsilon\}$ crossing $\tilde{D}_{z, L} \setminus \tilde C_{z, L}$ for some $z \in O_j'$ intersect $\mathscr C_j$.}
%&\text{comp. $\mathscr{C}_{z,z'}$ of $ \{\varphi \ge h_1 - \varepsilon\} \cap (  D_z\cup   D_{z'})$, which crosses both ${D}_z \setminus  C_z$ and ${D}_{z'} \setminus  C_{z'}$,}
\end{split}
\label{property:gluing_neighbor4}
\end{equation}

We now explain how the inclusions \eqref{eq:inclusion_goodevent.1} follow from this, and first consider the case $i=1$.
In view of Definition~\ref{def:goodevent2}, this amounts to verifying \eqref{def:goodevent2.1} for $\chi^z, \chi^{z'}=\varphi$, with $N$ in place of $L$ and at the heights given by \eqref{eq:inclusion_goodevent.1}, as well as \eqref{def:goodevent2.2} with $a'$ in place of $a$. First, the connection required in \eqref{def:goodevent2.1} is ensured by $\Omega_N^1$. 

To proceed further, we need the following observation. For a $*$-path $\gamma$ on $\Z^d$, define the trace $ \gamma^{\mathbb{L}}$ of $\gamma$ on $\mathbb{L}$ as follows: $ \gamma^{\mathbb{L}}(0)$ is the unique point in $\mathbb{L}$ such that $\gamma(0)\in C_{\gamma^{\mathbb{L}}(0),L}$ (recall that these boxes partition $\Z^d$, see \eqref{eq:boxes}). Set $n_0 =0$. Given $\gamma^{\mathbb{L}}(0),\dots, \gamma^{\mathbb{L}}(k-1)$ and $n_0,\dots n_{k-1}$ for some $k \geq 1$, set $n_k= \inf\{n > n_{k-1}: \gamma(n) \notin C_{\gamma^{\mathbb{L}}(k-1),L}\}$ and $\gamma^{\mathbb{L}}(k) \in \mathbb{L}$ is such that $\gamma(n_k) \in C_{\gamma^{\mathbb{L}}(k),L}$.
By construction $\gamma^{\mathbb{L}}$ is a $*$-path on $\mathbb{L}$. Moreover, if $\gamma$ crosses $\tilde{D}_{0,N}\setminus \tilde{C}_{0,N}$, then $\gamma^{\mathbb{L}}$ crosses $V\setminus U$. As the sets $\mathscr C_j$, $1\leq j \leq k'$, each surround $U$, it follows that for each $*$-path $\gamma$ crossing $\tilde{D}_{0,N}\setminus \tilde{C}_{0,N}$, \begin{equation}
\label{property:gluing_neighbor5}
\text{there exists $\{z_j: 1 \leq j \leq k'\} \subset \gamma^{\mathbb{L}}$ such that $\text{dist}_{\ell^{\infty}(\mathbb{L})}(z_j, O_j') \leq 1$.}
\end{equation}
Indeed, \eqref{property:gluing_neighbor5} follows for instance by extending $\gamma^{\mathbb{L}}$ to a nearest-neighbor path $\bar\gamma^{\mathbb{L}}$ on $\mathbb{L}$, which only requires adding vertices at unit $\ell^{\infty}(\mathbb{L})$-distance from $\gamma^{\mathbb{L}}$. The path $\bar\gamma^{\mathbb{L}}$ crossing $V\setminus U$ in turn intersects $O_j'$ for all $1 \leq j \leq k'$ by \cite[Lemma~2.1]{DeuschelPisztora96} and the surrounding property of each $O_j'$.

Now, returning to the verification of \eqref{def:goodevent2.1}, consider a cluster of $\{\varphi \ge h_2 + \varepsilon \}$ crossing $\tilde{D}_{0,N}\setminus \tilde{C}_{0,N}$. Extracting a crossing path $\gamma$ from this cluster, it follows by definition of $\gamma^{\mathbb{L}}$ that $\gamma$ induces a crossing in $D_{z_1,L} \setminus C_{z_1,L}$, with $z_1$ as in \eqref{property:gluing_neighbor5} (recall that $O_1' + D_{0, L} \subset \tilde D_{0, N}$). Hence, $\gamma$ induces a crossing in $\tilde{D}_{y_1,L} \setminus \tilde{C}_{y_1,L}$ for some $y_1 \in O_1'$ with $|z_1-y_1|_{\ell^{\infty}(\mathbb{L})} \leq 1$. Thus, $\gamma$ intersects $\mathscr{C}_1$ by \eqref{property:gluing_neighbor4}. All in all, each cluster of $\{\varphi \ge h_2 + \varepsilon \}$ crossing $\tilde{D}_{0,N}\setminus \tilde{C}_{0,N}$ intersects $\mathscr{C}_1 \subset \{\varphi \ge h_1-\varepsilon\}$. Since $\mathscr{C}_1$ is connected, the second part of \eqref{def:goodevent2.1} follows.

To deduce \eqref{def:goodevent2.2} relative to the event $\mathcal{G}_0 (\tilde{\varphi}, N, a', h_1 - \varepsilon, h_2 + \varepsilon , h_3- \varepsilon, M+ \varepsilon)$ in \eqref{eq:inclusion_goodevent.1}, one proceeds as follows. Repeating the above argument for all $j \in \{ 1,\dots,k'\}$ using \eqref{property:gluing_neighbor4} and \eqref{property:gluing_neighbor5}, one first observes that the cluster of $\{\varphi \ge h_3 - \varepsilon\}$ crossing $D_{0,N} \setminus C_{0,N}$ stipulated by $\Omega_N^1$ intersects each $\mathscr{C}_j$, and therefore
\begin{equation}
\label{property:gluing_neighbor6}
\text{$\mathscr C_0$, the cluster of $C_{0,N}$ in $\{\varphi \ge h_1-\varepsilon\}$, contains $\mathscr{C}_j$ for all $1 \leq j \leq k'$.}
\end{equation}
Now, still by the same argument, every $*$-path $\gamma$ crossing $\tilde{D}_{0,N}\setminus \tilde{C}_{0,N}$ induces a crossing in $\tilde{D}_{y_j,L} \setminus \tilde{C}_{y_j,L}$ for some $y_j \in O_j'$ and all $1\leq j \leq k'$ (in fact $\gamma$ is also connected to every $\mathscr{C}_j$ but we won't use this). By \eqref{property:gluing_neighbor3}, there exist sets $S_j \subset \gamma$ for all $1\leq j \leq k'$, each of cardinality at least $a$, such that $B_{L_0}(x) \cap  \tilde D_{y_j} \cap \mathscr{C}_j\cap\mathcal{M}_{M+\varepsilon}\neq \emptyset$ for all $x \in S_j$ and one can replace $\mathscr{C}_{j}$ by 
$\mathscr C_0$ in the previous intersection due to \eqref{property:gluing_neighbor6}.
By construction the sets $S_j$ are disjoint, Thus letting $S\coloneqq \bigcup_j S_j$, one obtains by \eqref{property:gluing_neighbor3} that $|S| \geq k' a \geq (k/8) a$, whence $|S| \geq a'$ on account of \eqref{eq:L2.1applied} and \eqref{eq:inclusion_goodevent.2} whenever $\tfrac{\rho N}{u(KL)} \geq C$. All in all, $S$ has all the properties required by \eqref{def:goodevent2.2}. %Indeed for every $x \in S$, $B_{L_0}(x) \cap  \tilde D_{0,N} \cap \mathscr{C}_{0} \cap\mathcal{M}\neq \emptyset$ due to \eqref{property:gluing_neighbor6}.

We now verify \eqref{eq:inclusion_goodevent.1} in case $i=2$. Consider a $*$-path $\gamma$ crossing $\Lambda_N$. By \eqref{property:gluing_neighbor5}, $\gamma$ induces crossings in $\tilde{D}_{y_j,L} \setminus \tilde{C}_{y_j,L}$ for suitable $y_j = y_j(\gamma) \in O_j'$ and all $1\leq j \leq k'$. Applying \eqref{property:gluing_neighbor3}, we obtain for every $j \in \{ 1,\dots,k'\}$ a set of points $\{ x_{j,k}= x_{j,k} (\gamma): 1 \leq k \leq  a\}$ such that $B_{L_0}(x_{j,k}) \cap  \tilde D_{y_j} \cap \mathscr{C}_j\cap\mathcal{M}_{M+\varepsilon}\neq \emptyset$. In view of \eqref{def:goodevent}, the inclusion \eqref{eq:inclusion_goodevent.1} for $i=2$ follows since the sets $\mathscr{C}_{j}$ are connected subsets of $\Lambda_N \cap \{ \varphi \geq h_1-\varepsilon \}$ and $k' \geq b$. 
\end{proof}

Our last missing ingredient needed prior to proceeding to the proofs of \eqref{eq:main_ubsupercrit_d=3} and \eqref{eq:main_ubsupercrit_d>4} is an a-priori estimate for the event $\mathcal{G}_z$ from Definition~\ref{def:goodevent2} at levels below $h_*$, which is available by current methods and which we supply next. This a-priori bound will play a role akin to \eqref{eq:bootstrapinput} in the subcritical case and enable us to initiate the bootstrap argument in Proposition~\ref{lem:inclusion_fxn_goodevent}.

\begin{lemma}[A-priori estimate]
	\label{lem:loc_amb_cluster1} $h_1  < h_2 < h_3 < h_*, \,  \varepsilon \in (0, \varepsilon_0 \wedge \frac{1}{3}(h_*-h_3)).$
	There exist $L_0 \geq 1$, $M>1$ and $\Cl[c]{c:supcrit_apriori} > 0$, each depending on $\underbar{h} =(h_1,h_2,h_3)$ and $\varepsilon$ only, such that for ${\mathcal G}_0= {\mathcal G}_0(\tilde\chi, L_0, L_n, a=1, \underbar{h},M)$ with $\tilde{\chi} \in \{ \tilde \varphi, \tilde{\psi}\}$ and $\frac{L_{n}}{L_{n+1}}=c(d)$, one has
	\begin{equation}
	\label{eq:good_event_typical1}
\lim_{n \to \infty} \frac{1}{L_n^{\Cr{c:supcrit_apriori}}}	\log \P[\mathcal G_0^c] < 0.
	\end{equation}%\todo{Check if a weaker form, i.e. tending to 1 suffices.}
\end{lemma}

\begin{proof}
It suffices to consider the case $\tilde{\chi} =  \tilde \varphi$. The case $\tilde{\chi} =   \tilde{\psi}$ then follows by applying
\eqref{eq:ellinfapproxsupcrit}. The bound \eqref{eq:good_event_typical1} (with $\tilde{\chi}=\tilde{\varphi}$) will follow by applying results of \cite{DPR18} to the graph $\tilde{\Z}^d$ (with unit weights). For $x \in \Z^d$, we consider the events (at scale $L_0$, see \eqref{eq:boxes} for notation)
\begin{align}
&A_x^1 \coloneqq \{\lr{}{\varphi \ge h_3 + \varepsilon}{C_{x,L_0}}{\partial D_{x,L_0}}\} \label{eq:RG1}\\
&A_x^2 \coloneqq \left\{
\begin{array}{c} \text{all clusters of $\{\varphi \ge h_2-\varepsilon \}$ crossing $\tilde D_{x,L_0} \setminus \tilde C_{x,L_0}$} \\ 
\text{are connected inside $D_{x,L_0} \cap \{ \varphi \geq h_1+\varepsilon \}$} \end{array} \right\}\label{eq:RG2}\\
&A_x^3 \coloneqq \{ \mathcal{M} \supset D_{x,L_0}\} \label{eq:RG3}
\end{align}
with $\mathcal{M}=  \mathcal{M}(\tilde{\varphi})$ as defined in \eqref{eq:ubsuper1} and $M \coloneqq (\log L_0)^2$. For $\lambda=\lambda(d) \, (\geq 100)$ sufficiently large --the choice of $\lambda$ corresponds to the constant $20c_{18}C_{10}$ appearing e.g.~in (8.3) of \cite{DPR18}, and in the present case $C_{10}=1$ and $c_{18}$ is determined by the isoperimetric constant on $\Z^d$) -- one then sets, for $x \in \Z^d$, with $\bar{\lambda}=1.1\lambda$, $\ell_0=3^d \vee 12 \bar{\lambda}$ and $L_n=\ell_0^n L_0$, for $1\leq k \leq3$,
\begin{equation}
 \label{eq:RG4}
\begin{split}
&\widetilde{A}_{x,0}^k= \bigcap_{y\in B_{\lambda L_0}(x)}A_y^k\\
&\widetilde{A}_{x,n}^k=  \bigcap_{y,z\in (L_{n-1}\Z^d\cap B_{\bar{\lambda} L_n}(x)): \, d(y,z)\geq L_n} \big(\widetilde{A}_{y,n-1}^k \cup \widetilde{A}_{z,n-1}^k\big),  \quad \text{for $n \geq 1$}.
\end{split}
\end{equation}
Since $h_3 + \varepsilon < h_*$, combining the bounds in \eqref{eq:EXISTUNIQUE} and applying a second union bound over $B_{\lambda L_0}$, one infers that that $ \lim_{L_0}\P[\widetilde{A}_{x,0}^k] =1$ for $k=1,2$. Similarly, one shows that 
$ \lim_{L_0}\P[\widetilde{A}_{x,0}^3] =1$ using a standard Gaussian tail estimate (note that $\text{var}((A\tilde{\varphi})_0) \geq c$) and applying a union bound over $D_{x,L_0}$ and $B_{\lambda L_0}$. All in all, one obtains that $ \lim_{L_0}\P[\widetilde{A}_{x,0}^k] =1$ for all $1 \leq k\leq 3$. In particular, by choosing $L_0=L_0(\underbar{\textit{h}},\varepsilon)$ sufficiently large, one can ensure that the conditions (7.5) and (7.6) in \cite{DPR18} are satisfied, whence Proposition 7.1 therein applies and yields that
\begin{equation}
 \label{eq:RG5}
\P[(\widetilde{A}_{x,n}^{k \prime})^c] \leq 2^{-2^n}, \text{ for all $1\leq k \leq 3$ and $n \geq 0$},
\end{equation} 
where the primed events $\widetilde{A}_{x,n}^{k \prime}$ refer to those defined in \eqref{eq:RG4}, but with sprinkled parameters $(h_1,h_2, h_3)$ in place of $(h_1+\varepsilon , h_2-\varepsilon , h_3+\varepsilon )$ in \eqref{eq:RG1}-\eqref{eq:RG3}. Note to this effect that the event $\widetilde{A}_{x,0}^k$ is measurable with respect to the restriction of $\tilde{\varphi}$ to $\tilde{\Z}^d \cap B_{\bar{\lambda} L_0}(x)$, as required for Proposition~7.1 in \cite{DPR18} to apply, and that a slight extension (of the underlying decoupling inequality (2.20)) is required when $k=2$, cf.~\eqref{eq:RG2}, in order to take care of the two opposite directions of monotonicity. To conclude, one applies Lemma~8.6 in \cite{DPR18}, which implies that whenever $\bigcap_{k} \widetilde{A}_{x,n}^{k \prime}$ occurs, any two connected sets in $B_{\lambda L_n}$ of diameter at least $(\lambda/20) L_n$ each, are connected by a path $\gamma \subset B_{2\lambda L_n}$ such that $ \bigcap_{k} \widetilde{A}_{x,0}^{k \prime}$ occurs for all $x \in \gamma$. Due to the (primed versions of the) choices \eqref{eq:RG1}-\eqref{eq:RG3}, this event is readily seen to imply $\mathcal G_0$ with $L= \lfloor (\lambda/10) L_n \rfloor$. The bound \eqref{eq:good_event_typical1} then follows from \eqref{eq:RG5}.
\end{proof}

We are now ready to assemble the pieces and prove \eqref{eq:main_ubsupercrit_d=3} and \eqref{eq:main_ubsupercrit_d>4}. In view of \eqref{eq:inclusion_goodevent.1}, this entails probing into the complements of the events $H_{N,L}^{\,i}$, $i=1,2$, from \eqref{def:good_box_event2}. In the spirit of \eqref{eq:psibadsub}, \eqref{eq:xibadsub}, for $z \in \mathbb L= \mathbb L(L)$ we say that
\begin{align}
	\{\text{$z$ is $\psi$-bad}\} &\coloneqq \mathcal G_z^c(\tilde{\psi}), \mbox{ and} \label{eq:psibadsup}\\
	\{\text{$z$ is $(\xi,i)$-bad}\} &\coloneqq 
	\begin{cases}
	\{ \sup_{D_z}|\tilde{\xi}_z| > \varepsilon \},& \text{ if }i=1\\
	\{ \inf_{D_z}\tilde{\xi}_z < -\varepsilon \},& \text{ if }i=2,
	\end{cases}
	 \label{eq:xibadsup}
\end{align}
whence $z$ is either $\psi$-bad or $(\xi,i)$-bad whenever $\mathcal{H}_z^i$ occurs, cf.~\eqref{eq:tildeGdef}, \eqref{eq:starGdef}. By \eqref{def:good_box_event2}, it then follows that for any $\rho' \in (0, 1 - \rho)$ and $i \in \{1,2\}$,
\begin{equation}
\label{eq:EF_decomp1}
\big(H_{N,L}^{\,i}\big)^c \subset E_{N, L}^{i} \cup F_{N, L}^{i},
\end{equation}
where
\begin{align*}
&E_{N, L}^i% = E_{N, L}^{K, \star}(\Lambda_N, \rho_2, \rho)
\coloneqq \left\{
\begin{array}{c}
\text{$\exists\,\mathcal{C}\in \mathcal A^i$ and $\tilde{\mathcal{C}}\subset \mathcal{C}$ with  $|\tilde{\mathcal{C}}| = \lceil \rho'|\mathcal{C}|\rceil$} \\ 
\text{such that all the sites in $\tilde{\mathcal{C}}$ are $\psi$-bad}
\end{array}
\right\}, \\
&F_{N, L}^i %= F_{N, L}^{K, \star}(\Lambda_N,\rho_2, \rho):=
\coloneqq \left\{
\begin{array}{c}
\text{$\exists\,\mathcal{C}\in \mathcal{A}^i$ and $\tilde{\mathcal{C}}\subset \mathcal{C}$ with  $|\tilde{\mathcal{C}}|= |\mathcal{C}|-\lceil \rho|\mathcal{C}|\rceil- \lceil \rho'|\mathcal{C}|\rceil$} \\ 
\text{such that all the sites in $\tilde{\mathcal{C}}$ are $(\xi,i)$-bad}
\end{array}
\right\}.
\end{align*}
%\begin{align*}
%	&E_{N, L} = E_{N, L}(\rho) := \bigcup_{\mathcal C \in \mathcal A_{N, L}^{100}(C_w, \tilde D_w)}\,\, \bigcup_{\substack{\tilde {\mathcal C} \subset \mathcal C\\|\tilde {\mathcal C}| \ge \rho |\mathcal C|}}\,\, \bigcap_{z \in \tilde {\mathcal C}} \,\mathcal G_z^c(\psi, \ldots),
%	\\
%	&F_{N, L} = F_{N, L}^{K}(\rho):= \bigcup_{\mathcal C \in \mathcal A_{\sigma, N, L}^K}\,\, \bigcup_{\substack{\tilde {\mathcal C} \subset \mathcal C\\|\tilde {\mathcal C}| \ge (1/2 - \rho) |\mathcal C|}}\,\, \bigcap_{z \in \tilde {\mathcal C}} \,\big\{\inf_{D_z}|\xi^z| \ge \varepsilon \big\}.
%\end{align*}

\medskip

At this point we consider the cases $d = 3$ and $d \ge 4$ separately.

\medskip

\noindent {\bf Upper bound for $d = 3$.} Combining Proposition~\ref{lem:inclusion_fxn_goodevent}, Lemmas~\ref{lem:loc_amb_cluster1} and \ref{lem:truncated} with a bootstrap argument (similar in spirit to the one leading to the corresponding subcritical upper bound), we proceed to give the

\begin{proof}[Proof of \eqref{eq:main_ubsupercrit_d=3}]

Let $h< h_*$. We assume in the sequel that $\varepsilon\in (0, \tfrac{h_* - h}{16})$ and set $h_1 = h_* - 12\varepsilon$, $h_2 = h_* - 8 \varepsilon$ and $h_3 = h_* - 4 \varepsilon$, whence $ \varepsilon \in (0, \varepsilon_0 \wedge \frac{1}{3}(h_*-h_3))$ (cf.~below \eqref{eq:starGdef} regarding $\varepsilon_0$).

In the first step, we take the bound given by Lemma~\ref{lem:loc_amb_cluster1} as our input and improve it (along with the parameter $a$) via Proposition~\ref{lem:inclusion_fxn_goodevent} applied with $i=1$. To this end, we first choose $L_0$ and $M$, both depending on $ h, \varepsilon$ only, such that Lemma~\ref{lem:loc_amb_cluster1} is in force. Then, applying Proposition~\ref{lem:inclusion_fxn_goodevent} with these choices for $L_0$ and $M$, as well as $K=100$, $\rho=1/2$ (see \eqref{def:good_box_event2}) and $a=1$, the inclusion \eqref{eq:inclusion_goodevent.1} and \eqref{eq:inclusion_goodevent.2} yield that 
\begin{align*}
%\label{eq:prob_decomp1}
\P[(\mathcal{G}_{0,N})^c] \le \P[(H_{N,L}^{\,1})^c ] + \P[\nlr{}{\varphi \ge h_3 - \varepsilon}{C_{0,N}}{\partial D_{0,N}}].
\end{align*}
for all $L > 2L_0$ and $N \geq CL$, where $ \mathcal{G}_{0,N} = \mathcal{G}_0 (\tilde{\varphi}, L_0, N, a= \lfloor \frac{cN}{L} \rfloor, h_1 - \varepsilon, h_2 + \varepsilon , h_3- \varepsilon, M+ \varepsilon)$ and $H_{N,L}^{\,1}= H_{N,L}^{\,1} (L_0,a= 1,  h_1, h_2, h_3, \varepsilon, M, \rho=\frac12) $. Incorporating \eqref{eq:EF_decomp1} with the choice $\rho'=\frac12$ and the upper bound on disconnection probability given 
by Theorem~5.5 in \cite{Sz15}, we obtain from the previous display, under the same assumptions on $L$ and $N$ that
\begin{align}
	\label{eq:prob_decomp_sup_3d_1}
	\P[(\mathcal{G}_{0,N})^c] \le  \P[E_{N, L}^1] + \P[F_{N, L}^1] + e^{-c(h, \varepsilon) N}.
\end{align}
In view of \eqref{eq:xibadsup}, adapting the argument used in the proof of Lemma~\ref{lem:bnd_xibad} to the present case where $\mathcal{A}^1 = \mathcal A_{N, L}^{100}(\tilde D_{0,N} \setminus \tilde C_{0,N})$, using symmetry and applying a union bound (costing an inconsequential factor $n\choose \lceil  n/2 \rceil$ where $n =|\mathcal{C}|$, $\mathcal{C} \in \mathcal{A}^1$) to get rid of absolute values in \eqref{eq:xibadsup} and $F_{N, L}^1$, we deduce that 
\begin{equation*}
	%\label{eq:xibad_sup_3d_1}
	\P[F_{N, L}^1] \le e^{-c (h,\varepsilon) \frac{N}{\log N}}, \text{ for all $N \ge 
C( h, \varepsilon)$ and $ (\log N)^3 \leq L \leq L_1(N)$}
\end{equation*} 
(with $L_1(N)$ as fixed above Lemma~\ref{lem:bnd_xibad}). On the other hand, in view of \eqref{eq:psibadsup}, retracing the steps that led to the proof of Lemma~\ref{lem:bnd_psibadstep1} and \eqref{eq:bootstrapENL}, replacing the input bound \eqref{eq:inputbnd} by \eqref{eq:good_event_typical1} for $\tilde{\chi} = \tilde{\psi}$ (whence $f(L)=c(h,\varepsilon)L^{\Cr{c:supcrit_apriori}}$), one finds that for all $L=L_n$ as appearing in  Lemma~\ref{lem:loc_amb_cluster1} satisfying $L \in [C(h,\varepsilon)(\log N)^{2/\Cr{c:supcrit_apriori}}, cN]$ and $N \ge C(h, \varepsilon)$,
\begin{equation*}
	%\label{eq:psibad_sup_3d_1}
	\P[E_{N, L}^1] \le e^{-c(h,\varepsilon)\frac{N}{L^{1 - \Cr{c:supcrit_apriori}}}}.
	\end{equation*}
Substituting the estimates for $\P[F_{N, L}^1]$ and $\P[E_{N, L}^1]$ into \eqref{eq:prob_decomp_sup_3d_1} and choosing 
$L=L_{n_0}$ with $n_0 =n_0(N, h,\varepsilon) \coloneqq \inf\{ n \geq 0 : L_n \geq C(h, \varepsilon)(\log N)^{\max(3,\, 2/ \Cr{c:supcrit_apriori})}\}$, we obtain for all $N \geq C(h,\varepsilon)$, 
\begin{equation}
\label{eq:mathcalG_bnd1}
\P[\mathcal{G}_0 (\tilde{\varphi}, L_0, N, a_N, h_1 - \varepsilon, h_2 + \varepsilon , h_3- \varepsilon, M+ \varepsilon)^c] \le e^{- f'(N)},
\end{equation}
where $a_N\coloneqq \lfloor \tfrac{N}{(\log N)^{\Cl[c]{c:supcrit_apriori2}}} \rfloor$ and $f'(N)\coloneqq  \frac{c(h, \varepsilon)N}{(\log N)^{C(h, \varepsilon)}}$.
%where $\beta'' \coloneqq \max(3, 2 / \beta') + 1$.
This yields the desired improvement, both in terms of $a$ and the probabilistic bound, over the a-priori estimate from Lemma \ref{lem:loc_amb_cluster1}.

%In the second step we will redo the computations in the first step with the improved bound \eqref{eq:mathcalG_bnd1} in place of \eqref{eq:good_event_typical1}. More precisely, we start with the inequality
%\begin{align*}
%	%\label{eq:prob_decomp2}
%	\P[\mathcal G_{w}^c(\tilde \varphi,&f', h_1 - \varepsilon/2, h_2 + \varepsilon/2 , M + \varepsilon/2)]\\ &\le \P[\mathscr G_w^c\big(\tfrac{L}{(\log L)^{\beta''}}, h_1 - \varepsilon/4, h_2 + \varepsilon/4, M + \varepsilon/4\big)] + \P[\nlr{}{\varphi \ge h_2 + \varepsilon/4}{B_w}{D_w}].
%\end{align*}
%where $f' \coloneqq \tfrac{N}{800(\log L)^{\beta''}}$, and bound all the relevant terms like 
%in the previous step with $L = (\log N)^3$. Consequently,  we get
%\begin{equation}
%\label{eq:mathcalG_bnd2}
%\P[\mathcal G_{w}^c(\tilde \varphi, \tfrac{N}{(\log \log N)^{\beta''}}, h_1 - \varepsilon/2, h_2 + \varepsilon/2 , M + \varepsilon/2)] \le e^{- c\varepsilon^2\frac{N}{\log N}}
%\end{equation}
%for $N \ge C(h_1, h_2, \varepsilon)$. As a consequence 
%

\medskip

In the second step, we start with the improved bound \eqref{eq:mathcalG_bnd1} 
and feed it to Proposition~\ref{lem:inclusion_fxn_goodevent} (in case $i=2$) to derive an estimate for the event $G_N$ in \eqref{def:goodevent}, for a suitable choice of the parameters. Thus applying \eqref{eq:inclusion_goodevent.1} with the  height parameters from \eqref{eq:mathcalG_bnd1} in place of $(h_1,h_2,h_3)$, $\varepsilon' \coloneqq h_1 - h - 2\varepsilon (>0)$ in place of $\varepsilon$ and $L_0=L_0(h,\varepsilon)$ as in the previous step, we obtain, for all $L \geq C(h,\varepsilon)$, $K \geq 100$, $\rho \in (0,1)$ and $N$ such that $\tfrac{\rho N}{u(KL)} \geq C$, using the decomposition \eqref{eq:EF_decomp1},
\begin{align}
	%\label{eq:prob_decomp3}
		\label{eq:prob_decomp_sup_3d_2}
	\P[G_{N}(a_L, b%= \lfloor \tfrac{\Cr{c:nLB}}{10} \,\tfrac{\rho N}{u(KL)} \rfloor
	,  h +\varepsilon, M')^c]  \le \P[E_{N, L}^{2}] + \P[F_{N, L}^{2}], %\P\big[(\mathscr G^\star)^c\big(\tfrac{L}{(\log L)^{\beta''}}, h_1-\varepsilon, h_2 + \varepsilon, M + \varepsilon\big)\big]
\end{align}
with $b $ as given by \eqref{eq:inclusion_goodevent.2}, $M'=M + h_1 - h - \varepsilon$ and where the events $E_{N, L}^{2}$, $F_{N, L}^{2}$ inherit the parameters from $H_{N,L}^{\,2} = H_{N,L}^{\,2} ( a_L,  h_1 - \varepsilon, h_2 + \varepsilon , h_3- \varepsilon, \varepsilon', M+\varepsilon, \rho)$, and depend on an additional $\rho' \in (0,1-\rho)$.
% $f'\coloneqq c(\Lambda_N)\rho_2\tfrac{N}{30K(\log L)^{\beta''}}$. Using the decomposition 
%\eqref{eq:EF_decomp1} %and the upper bound on disconnection probability from \cite{Sz15}, 
%we then get for any $\rho \in (0, 1 - \rho_2)$,
Now, mimicking the arguments of the previous step to bound the probabilities on the right-hand side in \eqref{eq:prob_decomp_sup_3d_2}, but this time using $f'$ instead $f$ (as implied by \eqref{eq:mathcalG_bnd1}) when estimating $\P[E_{N, L}^{2}]$, and with a view to \eqref{eq:xibadsup} (compare with the definition of $F_{N,L}$ in \eqref{eq:EandF} and the proof of \eqref{eq:bnd_xibad}) when dealing with $\P[F_{N, L}^{2}]$, one obtains the following for the choice $L\coloneqq L_{n_1}$ where $n_1(N, h,\varepsilon) \coloneqq \inf\{ n \geq 0 : L_n \geq  (\log N)^3\}$, whence $(\log N)^3 \leq L_{n_1} \leq C(h,\varepsilon) (\log N)^3$: For $K$ large enough, $\sigma$ (in the case $\Lambda_N = B_N \setminus B_{\sigma N}$) and $\rho $ close enough to 0, all depending on $h$ and $\varepsilon$ and $\rho'=1-2\rho$, 
\begin{equation}
\label{eq:mathcalG_bnd2}
\P[G_{N}(a_N', b'_N,  h +\varepsilon, M')^c]  \le e^{- \frac{\pi}{6} (h_1 - h - 2\varepsilon)^2 \frac{N}{\log N}}, \text{ for all $N \ge C(h,\varepsilon)$},
\end{equation}
where $a'_N= a_{L_{n_1}}$ and $b'_N= \lfloor \tfrac{c(h,\varepsilon) N}{L_{n_1}} \rfloor$. %$f'' \coloneqq \tfrac{N}{(\log \log N)^{\beta'' + 1}}$. 

\medskip

With \eqref{eq:mathcalG_bnd2} at hand, we now deduce  \eqref{eq:main_ubsupercrit_d=3} and first consider the finite-volume event ${\rm LocUniq}(N, h)$. Plugging \eqref{eq:mathcalG_bnd2} (with $\Lambda_N = B_{N} \setminus B_{\sigma N}$) into \eqref{eq:locauniq} with $h ' = h + \varepsilon$, we obtain for $N \ge C(h, \varepsilon)$,
\begin{equation}
\label{eq:ub_locuiq_d=3}
\begin{split}
\P[{\rm LocUniq}(N, h)^c] &\le e^{- \frac{\pi}{6} (h_* - h - C\varepsilon)^2 \frac{N}{\log N}} + \P[\nlr{}{\varphi \ge h'}{B_{N}}{\partial B_{2N}}] + e^{-cb'_N a'_N} \\
&\le 2e^{- \frac{\pi}{6} (h_* - h - C\varepsilon)^2 \frac{N}{\log N}} + e^{-c(h')N},
\end{split}
\end{equation}
where the second line follows by Theorem 5.5 in \cite{Sz15} and since $b'_N a'_N \geq 
c(h,\varepsilon)\frac N{ (\log \log N)^{C(h,\varepsilon)}}$. The claim readily follows from 
\eqref{eq:ub_locuiq_d=3} by taking logarithms, multiplying by $\frac{\log N}{N}$ on both 
sides, letting $N \to \infty$ and then $\varepsilon \downarrow 0$. The same upper bound then automatically holds 
 for the event $\text{2-arms}(N,h) ( \subset {\rm LocUniq}(N, h)^c)$.
%\begin{equation}
%\limsup_{N\to\infty}\, \frac{\log N}{N} \log \P[{\rm LocUniq}(N, h)] \leq -\frac{\pi}{6}(h_* - h - 7 \varepsilon)^2,
%\end{equation}
%whence the upper bound as in \eqref{eq:main_ubsupercrit_d=3} follows by sending $\varepsilon$ to 0.

The upper bound in \eqref{eq:main_ubsupercrit_d=3} for the truncated one-arm event will follow similarly from  
\eqref{eq:truncated} and \eqref{eq:mathcalG_bnd2} upon supplying a suitable upper bound for the probability of disconnecting $B_{\sigma n}$ from infinity above level $h' =  h + \varepsilon$ (here $\sigma=\sigma(h,\varepsilon)> 0$ refers to the choice that leads to \eqref{eq:mathcalG_bnd2}), which is not readily available for us to use. %Unfortunately we do not 
%have the latter bound for all values of $h'$ below $h_*$. In order to get around this 
%problem, we will use the upper bound on disconnection probabilities from \cite{Sz15} along 
%with the bound on $\P[{\rm LocUniq}(N, h)^c]$ to obtain a sub-optimal bound on the 
%disconnection probability from $\infty$ which would be sufficient for our purpose. To this 
To circumvent this issue, we combine the upper bound \eqref{eq:main_ubsupercrit_d=3} for the local uniqueness event derived above and the disconnection upper bound from \cite{Sz15} as follows: consider the sequence of events, for some integer $\overline{M} > 1$,
%\begin{gather*}%\label{def:gluingevent}
%\mathcal E_1 \coloneqq \{\lr{}{\varphi \ge h'}{B_{\sigma N}}{\partial B_{6\beta N}}\}, \mbox{ whereas for }k\ge 1,\\  \mathcal E_{2k} \coloneqq {\rm LocUniq}(h', 2^{k}\beta N) \mbox{ and }%\{\lr{}{\varphi \ge h'}{C_{0, 2^k\beta N}}{\partial \tilde D_{0, 2^k\beta N}}\}, 
% \mathcal E_{2k+1} \coloneqq \{\lr{}{\varphi \ge h'}{B_{2^k \overline{M} N}}{\partial B_{6\cdot2^k\overline{M} N}}\}.
%\end{gather*}
\begin{align*}%\label{def:gluingevent}
	&\mathcal E_1 \coloneqq \{\lr{}{\varphi \ge h'}{B_{ N}}{\partial B_{4\overline{M} N}}\}, %\mbox{ whereas for }k\ge 1,
	\\  
	&\mathcal E_{2k} \coloneqq {\rm LocUniq}(h', 2^{k}\overline{M} N) \mbox{ and }%\{\lr{}{\varphi \ge h'}{C_{0, 2^k\overline{M} N}}{\partial \tilde D_{0, 2^k\overline{M} N}}\}, 
	\mathcal E_{2k+1} \coloneqq \{\lr{}{\varphi \ge h'}{B_{2^{k} \overline{M} N}}{\partial B_{6\cdot2^k\overline{M} N}}\}, \text{ for $k \geq1$}.
\end{align*}
%Since $B_{2^{k + 1}\overline{M} N} \setminus B_{2^{k}\overline{M} N} \subset (B_{2^{k + 2}\overline{M} N} \setminus B_{2^{k}\overline{M} N}) \cap (B_{2^{k + 1}\overline{M} N} \setminus B_{2^{k-1}\overline{M} N})$ %$B_{2^{k-1}\overline{M} N} \subset \tilde C_{0, 2^k \overline{M} N} \subset \tilde D_{0, 2^k \overline{M} N} \subset B_{6\cdot2^{k}\overline{M} N}$ 
%for all $k \ge 1$,
It readily follows from the definition of {\rm LocUniq}, see \eqref{eq:def_locuniq}, that $\{\lr{}{\varphi \ge h'}{B_{ N}}{\infty}\} \subset \bigcap_{k \ge 1}\mathcal E_k$. Therefore, applying a union bound and subsequently using the bounds from \cite{Sz15} and \eqref{eq:ub_locuiq_d=3} for the respective probabilities, we get for any $\overline{M} > 1$ and $N \ge C(h, \varepsilon, \overline{M})$,
\begin{equation}
\label{eq:disconnect_infty0}
\begin{split}
&\P[\nlr{}{\varphi \ge h'}{B_{ N}}{\infty}] \\
&\qquad \quad \le \sum_{k \ge 1} \P[\mathcal E_k^c] \le \sum_{k \ge 0}e^{-c(h') 2^k\overline{M}  N} + \sum_{k \ge 1}e^{-c(h_* - h')^2\frac{2^k \overline{M} N}{\log (2^k \overline{M} N)}} \le e^{-c(h_* - h')^2\frac{\overline{M} N}{\log(N)}}.
\end{split}
\end{equation}
Since $\overline{M}$ is arbitrary, \eqref{eq:disconnect_infty0} implies that
\begin{equation}
\label{eq:disconnect_infty}
\lim_{N\to\infty}\, \frac{\log N}{N} \log \P[\nlr{}{\varphi\geq h}{B_N}{\infty } ] = -\infty.
\end{equation}%Now setting $\overline{M}  =  100 / c$ for example, we find that
%\begin{equation}
%\P[\nlr{}{\varphi \ge h'}{B_{\sigma N}}{\infty}] \le \e^{-100(h_* - h')^2\frac{N}{\log(N)}}
%\end{equation}
%for all $N \ge N(h, \varepsilon)$. 
for all $h < h_*$. Similarly as with \eqref{eq:ub_locuiq_d=3}, \eqref{eq:main_ubsupercrit_d=3} readily follows from \eqref{eq:disconnect_infty}, \eqref{eq:mathcalG_bnd2} and \eqref{eq:truncated}.
%\begin{equation*}
%	\limsup_{N\to\infty}\, \frac{\log N}{N} \log \P[\lr{}{\varphi \ge h}{0}{\partial B_{L}}, \nlr{}{\varphi \ge h}{0}{\infty}] \leq -\frac{\pi}{6}(h_* - h - 7 \varepsilon)^2,
%\end{equation*}
%and \eqref{eq:main_ubsupercrit_d=3} follows by taking the limit $\varepsilon\to 0$.
\end{proof}

\noindent {\bf Upper bound for $d \ge 4$.} Similarly to the subcritical phase, the proof of upper bounds simplifies in higher dimensions. 

\begin{proof}[Proof of \eqref{eq:main_ubsupercrit_d>4}]
Let $h<h_*$. We choose $h_1$, $h_2$, $h_3$ as in the beginning of the proof of \eqref{eq:main_ubsupercrit_d=3}
and simply fix $\varepsilon \coloneqq (h_* - h) / 20$. Following the same line of reasoning that led to the bound 
\eqref{eq:mathcalG_bnd1}, except for applying Proposition~\ref{lem:inclusion_fxn_goodevent} with $i=2$ directly (instead of $i=1$) with $\sigma= \frac14$ in case $\Lambda_N = B_N \setminus B_{\sigma_N}$, and using analogues of Lemmas~\ref{lem:bnd_psibadd4} and~\ref{lem:bnd_xibad4}
in place of Lemmas~\ref{lem:bnd_psibadstep1} and~\ref{lem:bnd_xibad}, respectively, to bound $\P[E_{N, L}^1]$ and $\P[F_{N, L}^1]$, thereby choosing $L=C(h)$ large enough, one finds that for all $N \geq C(h)$,
\begin{equation}
\label{eq:mathcalG_bnddgeq4}
\P[G_{N}(a=1, b= \lfloor c(h)N\rfloor,  h_1 -\varepsilon, M+ \varepsilon)]
%\P[\mathcal{G}_0 (\tilde{\varphi}, L_0, N, a=1, h_1 - \varepsilon, h_2 + \varepsilon , h_3- \varepsilon, M+ \varepsilon)] 
%\P[\mathcal G_{w}^c(\tilde \varphi, cN, h_1 - \varepsilon, h_2 + \varepsilon, M + \varepsilon)] 
\le e^{- c'(h)N}.
\end{equation}
Here $M=M(h)$ and $L_0=L_0(h)$, implicit in the definition $G_{N}$, cf.~\eqref{eq:ubsuper1} and \eqref{def:goodevent}, are chosen as in the proof of \eqref{eq:main_ubsupercrit_d=3} when applying Lemma~\ref{lem:loc_amb_cluster1}.
%and hence by Lemma~\ref{lem:inclusion_fxn_goodevent2}
%\begin{align*}
	%\label{eq:prob_decomp3}
%	\P[G^c\big(cN, h + \varepsilon, M + h_1 - h - \varepsilon\big)] \le e^{- c(h)N}.
%\end{align*}
Plugging \eqref{eq:mathcalG_bnddgeq4} for the choice $\Lambda_N=B_{2N}\setminus B_{N}$ into \eqref{eq:locauniq} with $h'=h_1=\varepsilon (> h)$ and using Theorem 5.5 in \cite{Sz15} in order to bound the disconnection probability, we get
\begin{equation}
\label{eq:ub_locuiq_dgeq4}
%\limsup_{N\to\infty}\, \frac{1}{N} \log \P [\text{2-arms}(N,h)] \le 
\limsup_{N\to\infty}\, \frac{1}{N} \log \P[{\rm LocUniq}(N, h)^c] < 0,
\end{equation}
along with a similar bound for $\text{2-arms}(N,h)$.
Regarding \eqref{eq:main_ubsupercrit_d>4} for the truncated one-arm event, proceeding similarly as in the case $d=3$,
we first derive from \eqref{eq:ub_locuiq_dgeq4} and the disconnection upper bound from \cite{Sz15} that $ \frac{1}{N} \log\displaystyle \P[\nlr{}{\varphi\geq h}{B_N}{\infty } ]$ tends to $-\infty$ as $N\to\infty$, from which \eqref{eq:main_ubsupercrit_d>4} follows upon applying \eqref{eq:truncated} and using \eqref{eq:mathcalG_bnddgeq4} with $\Lambda_N=B_{N}\setminus B_{N/4}$.
\end{proof}

We conclude with a few comments.

\begin{remark} \label{R:final}
\begin{enumerate}[label={\arabic*)}]
\item Analogues of Remarks \ref{remark:stronger_bootstrap} and \ref{Rk:onionsd=4} also hold in the supercritical regime. Namely, at the expense of iterating a few more times, initializing the above scheme does not require the full strength of the stretched exponential a-priori bound provided by Lemma~\ref{lem:loc_amb_cluster1}. Moreover, when $d \geq 4$, replacing the admissible collection $\mathcal{A}^i$ inherent to the definition of $H_{N,L}^i$ in \eqref{def:good_box_event2} by the corresponding collection $(\mathcal{A}')^i$ (cf.~Remark~\ref{R:cg3to4}), which employs the coarse-graining strategy used for $d=3$, one can derive analogues of \eqref{eq:subcritconcl3} for ${\rm LocUniq}(N, h)^c$ or the event in \eqref{eq:main_ubsupercrit_d>4} when $h<h_*$. The proof essentially follows the line of argument of Remark~\ref{Rk:onionsd=4}, applying Proposition~\ref{lem:inclusion_fxn_goodevent} with $i=1$ for all but the $k$-th step (where $k \geq 1$ refers to any target number of iterated logarithms, cf.~\eqref{eq:subcritconcl3}). The important thing to notice is that the number $a$ of contact points, albeit always growing  sublinearly in the macroscopic scale due to the choice of $L$, improves suitably along with the bound on $\mathcal{G}_0$ through intermediate steps of the iteration.

\item (Two-point functions). We now briefly discuss the amendments to our methods that would be required to obtain the asymptotics~\eqref{eq:intro2point} for the truncated two-point function. Regarding lower bounds, in the arguments leading to \eqref{eq:main_lbsubcrit_d=3} and \eqref{eq:main_lbsupercrit_d=3}, one would need to `tilt' the construction, thus following the $\ell^2$-geodesic between $x$ and $y$) rather than a horizontal line, in order to force a connection between $x$ and $y$. An asymptotic capacity estimate for such a discretized $\ell^2$-geodesic similar to \eqref{eq:cap_line_asymp} (when $\lfloor |x-y| \rfloor =N$) is given by Remark~\ref{remark:cap_tube}. Corresponding analogues of \eqref{eq:cap_tube_d=3} and \eqref{eq:cap_tube_d=3bis} would also be required. The former relies on the visibility Lemma~\ref{lem:visibility}, which is robust with respect to `tilting' of the above kind. To adapt the proof of \eqref{eq:cap_tube_d=3bis}, one could compare to a random walk on $\R^d$ with Gaussian increments, whose law inherits the symmetries of $\R^d$ and to which the arguments leading to the lower bound in \eqref{eq:scape_line} can be extended, and then rely on the results of~\cite{PS_1998__2__41_0} for comparison with $X$. The relevant upper bound for \eqref{eq:intro2point} would naturally follow by adapting our coarse-graining and bootstrapping scheme to a framework with Euclidean balls replacing $\ell^{\infty}$-ones (although, since $L \ll N$ in practice, one may in fact get away by coarse-graining using $\ell^{\infty}$-boxes at scale $L$).
%\item ($\ell^p$-balls). The lower bound \eqref{eq:main_lbsubcrit_d=3}, part of Theorem~\ref{thm:main_d=3}, continues to hold if one replaces $B_N$ by an $\ell^p$-ball for any $p \geq 1$, for the proof of \eqref{eq:main_lbsubcrit_d=3} yields a connection between $0$ and $Ne_1$. An adaptation of our coarse-graining (as in 2) above) then expectedly yields a matching upper bound if $p \geq 2$. The restriction is due to the projection argument in the proof of Lemma~\ref{Claim:cap1} and specifically, \eqref{eq:capcompa-0.5}, which fails when $p<2$.
\end{enumerate}
\end{remark}

\medskip

%%%%%%%%%%%%%%%%%%%%%%%%%%%%%%%%%%%%%%%%%%%%%%
%% Support information (funding), if any,   %%
%% should be provided in the                %%
%% Acknowledgements section.                %%
%%%%%%%%%%%%%%%%%%%%%%%%%%%%%%%%%%%%%%%%%%%%%%
\noindent \textbf{Acknowledgements.} 
 Part of this research was supported by the ERC Grant CriBLaM and an IDEX grant from 
 Paris-Saclay. S.G.'s research was supported by the SERB grant SRG/2021/000032 and in part by a grant from the Infosys Foundation as a member of the Infosys-Chandrasekharan virtual center for Random Geometry. F.S.'s work was partially supported by the Swiss FNS. We thank Jian Ding, 
 Alexis Pr\'evost and Mateo Wirth for discussions at various stages of this project. We are grateful to an anonymous referee for her/his numerous and valuable suggestions on a previous version 
 of this manuscript.

\bibliography{rodriguez}
\bibliographystyle{abbrv}

\end{document}